\documentclass[11pt]{amsart}
\usepackage{amsmath,amssymb, amsthm, mathabx,amscd, amsfonts, amssymb, hyperref, mathrsfs, textcomp, epsfig, a4wide, graphicx, IEEEtrantools, tikz, verbatim, xypic, subcaption}
\usepackage[mathscr]{euscript}

\usetikzlibrary{shapes,arrows,cd}
\usetikzlibrary{matrix,decorations.pathreplacing,angles,quotes}
\usetikzlibrary{decorations.markings}

 \tikzset{->-/.style={decoration={
  markings,
  mark=at position .5 with {\arrow{>}}},postaction={decorate}}}

\newtheorem{thm}{Theorem}[section]
\newtheorem{prop}[thm]{Proposition}
\newtheorem{lemma}[thm]{Lemma}

\newtheorem{Quest}{Question}

\newtheorem{cor}[thm]{Corollary}
\newcommand{\spin}{\mathfrak{s}}
\theoremstyle{definition}
\newtheorem{defn}{Definition}[section]

\theoremstyle{remark}
\newtheorem{remark}{Remark}[section]
\newtheorem{example}{Example}[section]

     \RequirePackage{rotating}                   
    \def\HSt{%
       \setbox0=\hbox{$\widehat{\mathit{HS}}$}
       \setbox1=\hbox{$\mathit{HS}$}
       \dimen0=1.1\ht0
       \advance\dimen0 by 1.17\ht1
       \smash{\mskip2mu\raise\dimen0\rlap{%
          \begin{turn}{180}
              {$\widehat{\phantom{\mathit{HS}}}$}
           \end{turn}} \mskip-2mu    
                \mathit{HS}
    }{\vphantom{\widehat{\mathit{HS}}}}{}}

     \RequirePackage{rotating}                   
    \def\HMt{%
       \setbox0=\hbox{$\widehat{\mathit{HM}}$}
       \setbox1=\hbox{$\mathit{HM}$}
       \dimen0=1.1\ht0
       \advance\dimen0 by 1.17\ht1
       \smash{\mskip2mu\raise\dimen0\rlap{%
          \begin{turn}{180}
              {$\widehat{\phantom{\mathit{HM}}}$}
           \end{turn}} \mskip-2mu    
                \mathit{HM}
    }{\vphantom{\widehat{\mathit{HM}}}}{}}
    
    \newcommand{\HMb}{\overline{\mathit{HM}}}
\newcommand{\HMf}{\widehat{\mathit{HM}}}
\newcommand{\HMred}{{\mathit{HM}}}

\newcommand{\vol}{\mathrm{vol}}
\newcommand{\C}{\mathbb{C}}
\newcommand{\R}{\mathbb{R}}
\newcommand{\SW}{\mathsf{SW}}
\newcommand{\W}{\mathsf{W}}

\newcommand{\SL}{\mathrm{SL}}
\newcommand{\GL}{\mathrm{GL}}

\newcommand{\ZZ}{\mathbb{Z}}
\newcommand{\PSL}{\mathrm{PSL}}
\newcommand{\SO}{\mathrm{SO}}

\newcommand{\sgn}{\mathrm{sgn}}
\begin{document}

\title{Closed geodesics and Fr\o yshov invariants of hyperbolic three-manifolds} 

\author{Francesco Lin}
\address{Department of Mathematics, Columbia University} 
\email{flin@math.columbia.edu}

\author{Michael Lipnowski}
\address{Department of Mathematics and Statistics, McGill University} 
\email{michael.lipnowski@mcgill.ca}
\begin{abstract}
Fr\o yshov invariants are numerical invariants of rational homology three-spheres derived from gradings in monopole Floer homology. In the past few years, they have been employed to solve a wide range of problems in three and four-dimensional topology. In this paper, we look at connections with hyperbolic geometry for the class of minimal $L$-spaces. In particular, we study relations between Fr\o yshov invariants and closed geodesics using ideas from analytic number theory. We discuss two main applications of our approach. First, we derive effective upper bounds for the Fr\o yshov invariants of minimal hyperbolic $L$-spaces purely in terms of volume and injectivity radius. Second, we describe an algorithm to compute Fr\o yshov invariants of minimal $L$-spaces in terms of data arising from hyperbolic geometry. As a concrete example of our method, we compute the Fr\o yshov invariants for all spin$^c$ structures on the Seifert-Weber dodecahedral space. Along the way, we also prove several results about the eta invariants of the odd signature and Dirac operators on hyperbolic three-manifolds which might be of independent interest.
\end{abstract}
\maketitle

\tableofcontents

\section*{Introduction}

Understanding the relationship between hyperbolic geometry and Floer theoretic invariants of three-manifolds is one of the outstanding problems of low-dimensional topology. In our previous work \cite{LL}, as a first step in this direction, we studied sufficient conditions for a hyperbolic rational homology sphere to be an $L$-space (i.e. to have simplest possible Floer homology \cite{KMOS}) in terms of its volume and complex length spectrum. Our approach was based on spectral geometry and its relation to hyperbolic geometry via the Selberg trace formula. It was implemented explicitly (taking as input computations from SnapPy \cite{CDGW}) to show that several manifolds of small volume in the Hodgson-Weeks census \cite{HW} are $L$-spaces. 
\\
\par
In the present paper we focus our attention on the Fr\o yshov invariants of rational homology spheres. These are numerical invariants $h(Y,\spin)\in\mathbb{Q}$ indexed by spin$^c$ structures which are extracted from the gradings in monopole Floer homology (\cite{KM}, Chapter $39$). The corresponding invariants in the context of Heegaard Floer homology are known as correction terms \cite{OSgr}, and the identity $d(Y,\spin)=-2h(Y,\spin)$ holds under the isomorphism between the theories (see \cite{Tauiso}, \cite{CGHiso}, \cite{KLTiso} and subsequent papers). These invariants have been applied in recent years to a wide range of problems in three and four dimensional topology, see among the many \cite{OSun}, \cite{OStr}, \cite{LidLev}. Despite this, their computation in specific examples is still a very challenging problem, even under the assumption that $Y$ is an $L$-space.
\\
\par
The first basic question about them in the spirit of the present paper is the following. Recall that given constants $V,\varepsilon>0$, the set of hyperbolic rational homology spheres $Y$ for which $\mathrm{vol}(Y)<V$ and $\mathrm{inj}(Y)>\varepsilon$ is finite (\cite{BP}, Chapter $5$), and therefore so is the set of possible Fr\o yshov invariants $h(Y,\spin)$.
\begin{Quest}\label{quest1}
Given $V,\varepsilon>0$, can one provide an \textbf{effective} upper bound on $|h(Y,\spin)|$ for all hyperbolic rational homology spheres with $\mathrm{vol}(Y)<V$ and $\mathrm{inj}(Y)>\varepsilon$?
\end{Quest}
Of course, this is a challenging question even when restricting to the smaller class of $L$-spaces with $\mathrm{vol}(Y)<V$ and $\mathrm{inj}(Y)>\varepsilon$. Another natural question in this spirit is the following.
\begin{Quest}\label{quest2}
Can one explicitly determine $h(Y,\spin)$ in terms of data coming from hyperbolic geometry (e.g. volume, injectivity radius, etc.)?
\end{Quest}
While we are not able to address these questions in the stated generality, we will answer them under the additional assumption that $Y$ is a \textit{minimal hyperbolic $L$-space}, i.e. a rational homology sphere $(Y,g_{\mathrm{hyp}})$ equipped with a hyperbolic metric $g_{\mathrm{hyp}}$ for which sufficiently small perturbations of the Seiberg-Witten equations do not admit irreducible solutions. One of the main results of \cite{LL} is that a hyperbolic rational homology sphere $Y$ for which $\lambda_1^*>2$ (where $\lambda_1^*$ is the first eigenvalue of the Hodge Laplacian acting on coexact $1$-forms) is a minimal hyperbolic $L$-space; furthermore, the inequality $\lambda_1^*>2$ can be verified algorithmically in concrete examples, taking as input the length spectrum up to a certain cutoff. The first result we present, which addresses Question \ref{quest1}, is the following.

\begin{thm}\label{thm1}
Suppose $Y$ is a minimal hyperbolic $L$-space with $\mathrm{vol}(Y)<V$ and $\mathrm{inj}(Y)>\varepsilon$, then there exists an \textbf{effectively computable} constant $K_{V,\varepsilon}$ for which $\lvert h(Y,\spin)|\leq  K_{V,\varepsilon}$ for all spin$^c$ structures on $Y$. For example, if $Y$ is a minimal hyperbolic $L$-space with $\mathrm{vol}(Y)<6.5$ and $\mathrm{inj}(Y)>0.15$ (e.g. any rational homology sphere in the Hodgson-Weeks census\footnote{The Hodgson-Weeks census \cite{HW} consists of approximatively $ 11$ thousand closed oriented hyperbolic manifolds with $\mathrm{vol}(Y)<6.5$ and $\mathrm{inj}(Y)>0.15,$ and most of these are rational homology spheres. It is currently not known what percentage of the such manifolds it encompasses.} with $\lambda_1^*> 2$), then for every spin$^c$ structure $\spin$, the inequality $\lvert h(Y,\spin)\lvert\leq 67658$ holds.
\end{thm}

In general, the dependence of $K_{V,\varepsilon}$ on $V$ and $\varepsilon$ is readily computable but does not admit a particularly pleasant closed form. 
In order to streamline our discussion, we take a simple approach to prove Theorem \ref{thm1} leading to bounds which are not asymptotically optimal. More refined arguments might lead to significantly sharper estimates, especially in the case of small injectivity radius (cf. Remark \ref{betterbound}).
\begin{remark}While there are many examples of minimal $L$-spaces with volume $<6.5$ and injectivity radius $>0.15$, it is not known whether there are examples with arbitrarily large volume or arbitrarily small injectivity radius, cf. \cite[Section 5]{LL}.
\end{remark}

The key observation behind the proof of Theorem \ref{thm1} is the following: for minimal hyperbolic $L$-spaces, the Fr\o yshov invariant $h(Y,\spin)$ can be expressed in terms of the \textit{eta invariants} $\eta_{\mathrm{sign}}$ and $\eta_{\mathrm{Dir}}$ of the odd signature operator $\ast d$ acting on coexact $1$-forms and the Dirac operator $D_{B_0}$ corresponding to the flat connection $B^t_0$ on the determinant line bundle. Recall that both operators are first order, elliptic and self-adjoint, and are therefore diagonalizable in $L^2$ with real discrete spectrum unbounded in both directions. The eta invariant is a numerical invariant that intuitively measures the spectral asymmetry of an operator, i.e. the difference between the number of positive and negative eigenvalues \cite{APS1}. Of course, in our cases of interest, both of these quantities are infinite, and the eta invariant is defined via suitable analytic continuation. While the latter was originally obtained using the heat kernel, in our setup it can also be understood in terms of closed geodesics via the Selberg trace formula for \textit{odd test functions}. One should compare this with classical work of Millson \cite{Mil} and Moscovici-Stanton \cite{MosSta} expressing the eta invariants in terms of values of suitable odd Selberg zeta functions. In particular, it is possible to provide explicit expressions for $\eta_{\mathrm{sign}}$ and $\eta_{\mathrm{Dir}}$ in terms of spectral and geometric data:
\begin{itemize}
\item in the case of $\eta_{\mathrm{sign}}$, the geometric input is the complex length spectrum we have already exploited in \cite{LL}. The main difference is that in our previous paper we only needed the trace formula for \text{even} test functions. This is because we were interested in the Hodge Laplacian $\Delta$ acting on coexact $1$-forms, and the trace formula involved spectral parameters $t_j$ with $t_j^2=\lambda_j^*$. When using even test functions, the choice of the sign $t_j$ is irrelevant, but in fact there is a natural choice (for a fixed orientation of $Y$) because $\Delta$ is the square of $\ast d$. 
\item in the case of $\eta_\mathrm{Dir}$, the relevant new geometric data is encoded in the \textit{spin$^c$ length spectrum} of $(Y,\spin)$; this can be used to obtain information about the spectra of the corresponding Dirac operator via a specialization of the Selberg trace formula for the group
\begin{equation*}
G=\left\{g\in\mathrm{GL}_2(\mathbb{C})\text{ for which } |\mathrm{det}(g)|=1\right\}.
\end{equation*}
This should be thought as the spin$^c$ analogue of the group $\mathrm{PGL}_2(\mathbb{C})=\mathrm{Isom}^+(\mathbb{H}^3)$ which we studied in our previous paper.
\end{itemize}
Given this, Theorem \ref{thm1} follows by applying ideas of analytic number theory and choosing suitable compactly supported test functions. The most notable inputs are \textit{local Weyl laws} \cite{Mul}, which allow one to effectively bound from above the number of (coexact) $\ast d$ or the number of Dirac eigenvalues in any given interval. The trace formula relates both the odd signature and Dirac eta invariants quite explicitly to the hyperbolic geometry of the underlying manifold, which allows us to prove effective upper bounds.  For example, we will show that for any hyperbolic rational homology sphere $Y$ with volume $<6.5$ and injectivity radius $>0.15$, the explicit inequalities
\begin{align*}
\lvert\eta_{\mathrm{sign}}\lvert&\leq108267\\
\lvert\eta_{\mathrm{Dir}}\lvert&\leq 108249
\end{align*}
hold, where the second estimate is independent of the choice of spin$^c$ structure. Let us remark again that these estimates are not optimal even within the range of our techniques.
\\
\par
The injectivity radius makes its appearance in the assumptions of these results because it equals half of the length of the shortest closed geodesic. This implies that the relevant trace formula takes a particularly simple form when evaluated using test functions supported in the interval $[-2 \cdot \mathrm{inj}(Y),2 \cdot \mathrm{inj}(Y)]$, as the sum over closed geodesics vanishes.
\par
More generally, explicit knowledge of the length spectrum up to a certain cutoff provides much more detailed information about Fr\o yshov invariants and can in fact be used to provide explicit computations in the spirit of Question \ref{quest2}. In the second part of the paper, after showcasing the main ideas behind this approach in the simple case of the Weeks manifold (whose Fr\o yshov invariants can be computed in a purely topological fashion \cite{MO}), we will focus on making this process explicit in the challenging case of the Seifert-Weber dodecahedral space $\SW$. While this is one of the first examples of hyperbolic manifolds to be discovered \cite{WS}, it is a complicated space to study from the point of view of three-dimensional topology; for example it took $30$ years to verify Thurston's conjecture that $\SW$ is not Haken \cite{BRT}. In \cite{LL2} we used our techniques to show that $\SW$ is a minimal hyperbolic $L$-space by taking into account its large symmetry group. The next result determines its Fr\o yshov invariants for all the spin$^c$ structures on $\SW$ (recall that $H_1(\SW;\mathbb{Z})=(\mathbb{Z}/5\mathbb{Z})^3)$.

\begin{thm}\label{thm2}
Let $\spin_0$ be the unique spin structure on $\SW$, and consider a spin$^c$ structure $\spin=\spin_0+x$ on $\SW$ for $x\in H^2(\SW,\mathbb{Z})$. Then the Fr\o yshov invariant $h(\SW,\spin)$ is computed as in the following table, according to the value of the linking form $\mathsf{lk}(x,x)\in \mathbb{Q}/\mathbb{Z}$.
\end{thm}
\begin{center}
\begin{tabular} { | c | c |}
\hline
$\mathsf{lk}(x,x)$ & $h(\SW,\spin)$  \\
\hline
$0$ &  $-1 / 2$  \\
\hline
$1/5$ & $- 1/10$ \\
\hline
$2/5$ & $- 7/10$\\
\hline
$3/5$ & $- 3/10$\\
\hline
$4/5$ & $ 1/10$\\
\hline
\end{tabular}
\label{SWfro}
\end{center}
\begin{remark}
As we will see in the proof, the case $\mathsf{lk}(x,x)=0$ really encompasses two distinct cases: the unique spin structure and a family of $24$ (non-spin) spin$^c$ structures.
\end{remark}

An important observation here is that the fractional parts of Fr\o yshov invariants of $Y$ admit partial interpretations in terms of the linking form of $Y$ and the topology of manifolds bounding $Y$; given the extra topological input, the problem boils down to computing eta invariants up to a certain small (but reasonable) error. Towards this end, the main limitation is that we can only access a limited amount of the length spectrum of $\SW$. To prove Theorem \ref{thm2}, it is most convenient to use dilated Gaussians as test functions, because both the function and its Fourier transform (which is again a Gaussian) are rapidly decaying; as Gaussians are not compactly supported, we need to truncate certain infinite sums over closed geodesics that arise when estimating the $\eta_{\mathrm{sign}}$ and $\eta_{\mathrm{Dir}}$ via the odd trace formula for coexact $1$-forms and spinors. The main step in the proof is then to estimate the error introduced in this procedure. This can be done by deriving effective bounds on the number of closed geodesics of a given length, in the spirit of the prime geodesic theorem with error terms (\cite{Bus}, Section $9.6$).

\begin{remark}
The approach we implement for $\SW$ can be in principle carried out for any minimal hyperbolic $L$-space for which the linking form is known; the fundamental limitation comes from computing length spectra. In fact, the same ideas can also be exploited to obtain closed formulas for the Fr\o yshov invariants of minimal hyperbolic $L$-spaces with finitely many terms with an effective (but impractical) upper bound on the number of terms. We will not pursue such a closed formula in the present paper.
\end{remark}

Let us remark that explicit computations for the eta invariant of the odd signature operator $\eta_{\mathrm{sign}}$ for hyperbolic three-manifolds have been implemented in Snap \cite{CGHW}, and are based on a Dehn filling approach \cite{Ou}. The key insight is that for a fixed oriented compact $4$-manifold $X$ bounding $Y$, the general Atiyah-Patodi-Singer index theorem \cite{APS1} relates  $\eta_{\mathrm{sign}}$ to the kernel of the odd signature operator on $Y$ and the signature of $X$, both of which are topological invariants.
\par
However, while the APS index theorem also holds for the spin$^c$ Dirac operator, it is well known that the dimension of its kernel on $Y$ is not a topological invariant \cite{Hit}; this fact makes the computation of $\eta_{\mathrm{Dir}}$ much more subtle than its odd signature operator counterpart.
In fact, the computations of Fr\o yshov invariants carried out in this paper for some explicit minimal hyperbolic $L$-spaces provide as a byproduct explicit examples of hyperbolic three-manifolds for which one can compute Dirac eta invariants to high accuracy (which are not zero for obvious geometric reasons, e.g. the existence of an orientation-reversing isometry), and to the best of our knowledge these are the first such examples. For example, our methods will show that for the unique spin structure on the Weeks manifold,
\begin{equation*}
\eta_{\mathrm{Dir}}=0.989992\dots
\end{equation*}
This relies on the fact that the Weeks manifold is a minimal hyperbolic $L$-space, together with the computation of $\eta_{\mathrm{sign}}$ given in \cite{CGHW}; in particular, the value is as precise as the computations provided by Snap.
\par
More generally, the odd trace formula allows one to obtain bounds on $\eta_{\mathrm{Dir}}$ provided one can access the spin$^c$ length spectrum of $(Y,\spin)$; in turn, we will describe an algorithm to compute $\mathrm{spin}^c$ length spectra taking as input information computed using SnapPy. In particular, our method could in principle be applied to compute the invariants $\eta_{\mathrm{Dir}}$ for any hyperbolic three-manifold, even though at a practical level it might be infeasible to obtain a decent approximation in a reasonable time.
\\
\par

\textit{Note for the reader. }The paper is structured so that the various trace formulas (see Section \ref{traceformulas} for the statements) can be treated as black boxes, and all subsequent sections are written in a way that is hopefully self-contained. In particular, in Section \ref{continuation} we only use some complex analysis to provide an explicit formula for the eta invariants in terms of eigenvalues and complex (spin$^c$) lengths. Given this, the remainder of the paper only uses basic facts about Fourier transforms, and we will provide motivation and context for the tools from analytic number theory which we employ. Detailed proofs of the various trace formulas which we use can be found in the appendices; our discussion there assumes the reader to be familiar with the proof of the even trace formula for coexact $1$-forms in our previous work \cite[Appendix B]{LL}.
\\
\par
\textit{Plan of the paper.} Sections \ref{reviewfro}, \ref{spinlength} and \ref{traceformulas} provide background about the main protagonists of the paper: Fr\o yshov invariants, spin$^c$ length spectra, and trace formulas (both even and odd). In Section \ref{continuation} we use the odd trace formulas to provide an explicit expression for the eta invariants of the odd signature and Dirac operators; this will be the main tool for the present paper. In particular we prove Theorem \ref{thm1} in Sections \ref{localweyl} and \ref{proofthm1} by providing explicit bounds on the terms appearing in the sum. In Section \ref{weeks}, we show how our analytical expressions for eta invariants can be used to perform explicit computations on the Weeks manifold, the simplest minimal hyperbolic $L$-space.  This example is propaedeutic for the significantly more challenging case of the Seifert-Weber space, which we discuss in detail in Sections \ref{linkingSW}, \ref{geodesicSW} and \ref{proofthm2}.
\\
\par
\textit{Acknowledgements. }We are greatly indebted to Nathan Dunfield for all the help with SnapPy. The first author was partially supported by NSF grant DMS-1948820 and an Alfred P. Sloan fellowship.

\vspace{0.5cm}
\section{Background on Fr\o yshov and eta invariants}\label{reviewfro}
In this section we review some background topics that will be central for the purposes of the paper.

\subsection{Formal structure of monopole Floer homology and Fr\o yshov invariants} In \cite{KM} the authors associate to each three-manifold $Y$ three $\ZZ[U]$-modules fitting in a long exact sequence
\begin{equation*}
\cdots\stackrel{i_*}{\longrightarrow} \HMt_*(Y)\stackrel{j_*}{\longrightarrow} \HMf_*(Y)\stackrel{p_*}{\longrightarrow}  \HMb_*(Y)\stackrel{i_*}{\longrightarrow} \cdots.
\end{equation*}
These are read respectively \textit{HM-to}, \textit{HM-from} and \textit{HM-bar}, and we collectively refer to them as monopole Floer homology groups. Such invariants decompose along spin$^c$ structures on $Y$; for example we have
\begin{equation*}
\HMt_*(Y)=\bigoplus_{\spin\in \mathrm{Spin}^c(Y)}\HMt_*(Y,\spin).
\end{equation*}
The reduced Floer homology group $\HMred_*(Y,\spin)$ is defined to be the kernel of the map $p_*$. In this paper, we will be only interested in the case of rational homology spheres. In this situation the Floer homology groups for a fixed spin$^c$ struture admit an absolute grading by a $\ZZ$-coset in $\mathbb{Q}$, and the action of $U$ has degree $-2$. Furthermore, we have that
\begin{equation*}
\HMb_*(Y,\spin)\cong \ZZ[U,U^{-1}]
\end{equation*}
as graded modules (up to an overall shift), $\HMf_*(Y,\spin)$ vanishes in degrees high enough, and $p_*$ is an isomorphism in degrees low enough.
\par
A rational homology sphere $Y$ is called an $L$-\textit{space} if $\HMred_*(Y,\spin)=0$ for all spin$^c$ structures.
Given a spin$^c$ rational homology sphere $(Y,\spin)$, denote the minimum degree of a non-zero element in $i_*\left(\HMb_*(Y,\spin)\right)\subset \HMt_*(Y,\spin)$ by $-2h(Y,\spin)$. The quantity $h(Y,\spin)$ is then called the \textit{Fr\o yshov invariant} of $(Y,\spin)$.

\subsection{Fr\o yshov invariants in terms of eta invariants.} Recall \cite[Theorem $4.14$]{APS1} that for an oriented three-manifold the odd signature operator $B$ acts on even forms $\Omega^{\mathrm{even}}$ as
\begin{equation*}(-1)^p(\ast d-d\ast) \text{ on $2p$-forms.}\end{equation*} We identify $2$-forms and $1$-forms using $\ast$, so that the operator is
\begin{equation*}
\left( \begin{array}{cc} \ast d & d\\
d^* & 0
\end{array} \right)
\end{equation*}
acting on $\Omega^1\oplus \Omega^0$. This is a first order elliptic self-adjoint operator, and is diagonalizable in $L^2$ with real discrete spectrum unbounded in both directions. Its eta function is defined to be
\begin{equation*}
\eta_{\mathrm{sign}}(s)=\sum_{\lambda\neq0\text{ eigenvalue }} \frac{\mathrm{sgn}(\lambda)}{|\lambda|^s};
\end{equation*}
this sum defines a holomorphic function for $\mathrm{Re}(s)$ large. One of the key results of \cite{APS1} is that it admits an meromorphic continuation to the entire complex plane, with a regular value at $s=0$. In Section \ref{continuation}, we will quickly review the original proof of this (via the heat kernel) and then provide an alternative interpretation via the trace formula. Intuitively, $\eta_{\mathrm{sign}}=\eta_{\mathrm{sign}}(0)$ measures the spectral asymmetry of the operator. Similarly, the same procedure works for the Dirac operator $D_{B_0}$ (and its perturbations), leading to $\eta_{\mathrm{Dir}}$.
\begin{remark}
For our purposes it is convenient to notice (cf. \cite[Proposition $4.20$]{APS1}) that the odd signature operator (assuming for simplicity $b_1(Y)=0$) under the Hodge decomposition $\Omega^1=d \Omega^0\oplus d^*\Omega^2$ can be written as 
\begin{equation*}
\left( \begin{array}{ccc} \ast d&0& 0\\
0&0 & d\\
0&d^*&0
\end{array} \right)
\end{equation*}
acting on $ d^*\Omega^2\oplus d \Omega^0\oplus \Omega^0$. Now, the block
\begin{equation*}
\left( \begin{array}{cc}
0 & d\\
d^*&0
\end{array} \right)
\end{equation*}
has symmetric spectrum because $d$ and $d^*$ are adjoints; in particular we have for $\mathrm{Re}(s)$ large enough
\begin{equation}\label{eta1form}
\eta_{\mathrm{sign}}(s)=\sum_{t_j} \frac{\mathrm{sgn}(t_j)}{|t_j|^s};
\end{equation}
where the sum runs only on the eigenvalues $\{t_j\}$ of $\ast d$ on \textit{coexact} $1$-forms (notice that all the $t_j$ are non-zero). In particular, $\eta_{\mathrm{sign}}$ coincides with the spectral asymmetry of the action of $\ast d$ on coexact $1$-forms.
\end{remark}
\begin{remark}
We have $(\ast d)^2=\Delta$ when acting on coexact $1$-forms, and therefore the squares of the parameters $t_j$ are exactly the eigenvalues $\lambda_j^*$ of $\Delta$ we studied in \cite{LL}. The crucial extra information for the purposes of the present paper is the sign of these parameters. 
\end{remark}

The relation between eta invariants and the Fr\o yshov invariant is the following. Recall that a minimal $L$-space is a rational homology sphere admitting a metric for which small perturbations of the Seiberg-Witten equations do not have irreducible solutions.
\begin{prop}\label{froyeta}
Suppose $(Y,g)$ is a minimal $L$-space. Then for each spin$^c$ structure $\spin$,
\begin{equation*}
h(Y,\spin)=-\frac{\eta_{\mathrm{sign}}}{8}-\frac{\eta_{\mathrm{Dir}}}{2}
\end{equation*}
where $\eta_\mathrm{Dir}$ is the eta invariant of the Dirac operator $D_{B_0}$ corresponding to the flat connection on the determinant line bundle $B^t_0$.
\end{prop}

Classical examples of minimal $L$-spaces are rational homology spheres admitting metrics with positive scalar curvature. In \cite{LL}, we showed that a hyperbolic rational homology sphere for which the first eigenvalue of the Hodge Laplacian coexact $1$-forms $\lambda_1^*$ satisfies $\lambda_1^*>2$ is a minimal $L$-space (using the hyperbolic metric); we furthermore provided several examples of such spaces.
\par
While Proposition \ref{froyeta} is well-known to experts, we will dedicate the rest of the section to its proof; our discussion will assume some familiarity with the content of \cite{KM}, and will not be needed later in the paper except in Section \ref{proofthm2} where we will briefly use the explicit form of the absolute grading (Equation (\ref{formulaabsgrX}) below). The main idea behind the proof is the following. Under the assumption that there are no irreducible solutions, after a small perturbation the Floer chain complex has generator corresponding to the positive eigenspaces of (a small perturbation of) the Dirac operator $D_{B_0}$. From this description it readily follows that they are $L$-spaces \cite[Chapter $22.7$]{KM}, and that $-2h(Y,\spin)$ is the absolute grading of the critical point corresponding to the first positive eigenvalue; the goal is then to express this absolute grading in terms of eta invariants using the APS index theorem.

\vspace{0.3cm}

\subsection{Proof of Proposition \ref{froyeta}}
We begin by recalling from \cite[Chapter $28.3$]{KM} how absolute gradings in monopole Floer homology are defined for torsion spin$^c$ structures. As we only consider rational homology spheres, the $\mathrm{spin}^c$ structures in our context are automatically torsion.  Consider a cobordism $(W,\spin_W)$ from $S^3$ to $(Y,\spin)$; equip $S^3$ with a round metric and a small admissible perturbation, and denote by $[\mathfrak{a}_0]$ the first stable critical point of $S^3$, corresponding to the first positive eigenvalue of the Dirac operator. Consider on $W$ a metric which is a product near the boundary, and denote by $W^*$ the manifold obtained by attaching cylindrical ends. We then define for a critical point $[\mathfrak{a}]$ of $Y$ the rational number
\begin{equation}\label{formulaabsgr}
\mathrm{gr}^{\mathbb{Q}}([\mathfrak{a}])=-\mathrm{gr}([\mathfrak{a}_0],W^*,\spin_W,[\mathfrak{a}])+\frac{c_1(\spin_W)^2-2\chi(W)-3\sigma(W)}{4}
\end{equation}
where:
\begin{itemize}
\item $\mathrm{gr}([\mathfrak{a}_0],W^*,\spin_W,[\mathfrak{a}])$ is the expected dimension of the moduli space of solutions of Seiberg-Witten equations in the spin$^c$ structure $\spin_W$ that converge to $[\mathfrak{a}_0]$ and $[\mathfrak{a}]$. Concretely, this is the index of the linearized equations, after gauge fixing. 
\item $c_1(\spin_W)^2\in\mathbb{Q}$ is the self-intersection of the class $c_1(\spin_W)\in H^2(W;\mathbb{Z})$. Recall that this is defined as
\begin{equation*}
(\tilde{c}\cup\tilde{c})[W,\partial W]
\end{equation*}
where $\tilde{c}\in H^2(W,\partial W,\mathbb{Q})$ is any class whose image in $H^2(W,\mathbb{Q})$ is the same as the image $c_1(\spin_W)$ under the change of coefficient map $H^2(W,\mathbb{Z})\rightarrow H^2(W,\mathbb{Q})$.
\end{itemize}

For our purposes, it will be convenient to work with a closed manifold $X$ with boundary $Y$ over which $\spin$ extends instead; this can be obtained from $W$ by gluing in a ball $D^4$ to fill the $S^3$ boundary component. In this case the formula
\begin{equation}\label{formulaabsgrX}
\mathrm{gr}^{\mathbb{Q}}([\mathfrak{b}])=-\mathrm{gr}(X^*,\spin_X,[\mathfrak{b}])+\frac{c_1(\spin_X)^2-2\chi(X)-3\sigma(X)-2}{4}\in\mathbb{Q}
\end{equation}
holds, where $\mathrm{gr}(X^*,\spin_X,[\mathfrak{b}])$ is again defined as the expected dimension of the relevant moduli space. This readily follows via the excision principle for the index from the definition in formula (\ref{formulaabsgr}) and the computation for $D^4$ in \cite[Chapter $27.4$]{KM}.
\\
\par
Let us point out the following observation.
\begin{lemma}\label{nokernel}
Suppose $(Y,g)$ is a minimal $L$-space. Then for each spin$^c$ structure $\spin$, the Dirac operator $D_{B_0}$ corresponding to the flat connection $B_0^t$ has no kernel.
\end{lemma}
\begin{remark}
As a consequence, the minimal hyperbolic $L$-spaces we exhibited in \cite{LL} (e.g. the Weeks manifold) do not admit harmonic spinors. These seem to be the first known examples of hyperbolic three-manifolds having no harmonic spinors; more examples can be found using the trace formula techniques we discuss later in this paper.
\end{remark} 
\begin{proof}
Consider the small perturbation of the Chern-Simons-Dirac functional
\begin{equation*}
\mathcal{L}_{\delta}=\mathcal{L}+\frac{\delta}{2}\|\Psi\|^2_{L^2},
\end{equation*}
for which the corresponding perturbed Dirac operator at $(B_0,0)$ is $D_{B_0}+\delta$. For small values of $\pm\delta$, this operator has no kernel and the Seiberg-Witten equations still have no irreducible solutions (because $Y$ is a minimal $L$-space). By adding additional small pertubations, we can assure that the spectra of these two operators are simple, and they still do not have kernel, so we obtain transversality in the sense of \cite[Chapter $12$]{KM}. In particular, they both determine chain complexes computing the Floer homology, hence both the two first stable critical points have the same absolute grading $-2h(Y,\spin)$. This implies that the moduli space of solutions to the perturbed equations connecting them on a product cobordism $\mathbb{R}\times Y$ (for which the induced map is an isomorphism) is zero dimensional, hence the spectral flow of the corresponding linearized operator at the reducible is also zero. After performing a small homotopy (cf. \cite[Chapter 14]{KM}), this implies that the spectral flow for a family of operators of the form $D_{B_0}+f(t)$, where $f(t)$ is a monotone function with $f(t)=\pm\delta$ for $t>1$ and $t<-1$ respectively, equals zero. In particular, $D_{B_0}$ has no kernel.
\end{proof}

In this argument, one can avoid adding extra small perturbations to make the spectra simple when working with Morse-Bott singularities instead \cite{LL}. To streamline the argument below, we will assume that the first positive eigenspace of $D_{B_0}$ is simple, and the general case can be dealt with by either adding a small perturbation or by working in a Morse-Bott setting. Consider first reducible critical point $[\mathfrak{b}_0]$, which lies over the reducible solution $[B_0,0]$. Under our assumptions, its absolute grading is exactly $-2h(Y,\spin)$. Taking $A$ to be a spin$^c$ connection of $\spin_X$ restricting to $B_0$ in a neighborhood of the boundary, linearizing the equations at $(A,0)$ we have
\begin{equation*}
\mathrm{gr}(X^*,\spin_X,[\mathfrak{b}_0])=\mathrm{ind}_{L^2(X)}(d^*+d^+)+2\mathrm{ind}^{\mathbb{C}}_{L^2(X)}(D^+_A)
\end{equation*}
where
\begin{equation*}
d^*+d^+:i\Omega^1(X)\rightarrow i\Omega^0(X)\oplus i\Omega^+(X)
\end{equation*}
has index
\begin{equation*}
{b_1(X)-b_2^+(X)-1}=-\frac{\chi(X)+\sigma(X)+1}{2}.
\end{equation*}
The Atiyah-Patodi-Singer for the odd signature operator says in our case
\begin{equation*}
\sigma(X)=\eta_{\mathrm{sign}}+\frac{1}{3}\int_X p_1 
\end{equation*}
where $p_1$ is the first Pontryagin form of the metric, and $\eta_{\mathrm{sign}}$ is the eta invariant of the odd signature operator. Let us point out that the different sign comes from our convention for boundary orientations from that of \cite{APS1}: we think of $X$ as a `filling' of $Y$, rather than a `cap', and this in turns implies that our boundary conditions involve the negative spectral projection \cite[Chapter 17]{KM}. For the Dirac operator, we obtain
\begin{equation*}
\mathrm{ind}^{\mathbb{C}}_{L^2(X)}(D^+_A)=\frac{1}{2}\eta_{D_{\mathrm{Dir}}}+\frac{1}{8}\int_X \left(-\frac{1}{3}p_1 +c_1(A)^2\right)
\end{equation*}
where $c_1(A)$ is the Chern form of $A$, and the term involving the dimension of the kernel of $D_{B_0}$ is not present because of Lemma \ref{nokernel}. Putting everything together, we see that
\begin{equation*}
\mathrm{gr}^{\mathbb{Q}}([\mathfrak{b}_0] )= \frac{1}{4}  \eta_{\mathrm{sign}} + \eta_{\mathrm{Dir}},
\end{equation*}
and the result follows.

\vspace{0.5cm}
\section{Torsion spin$^c$ structures and spin$^c$ length spectra}\label{spinlength}
Recall that a closed geodesic $\gamma$ in an oriented hyperbolic three-manifold admits complex length
\begin{equation*}
\mathbb{C}\ell(\gamma) := \ell(\gamma)+i \cdot \mathrm{hol}(\gamma)\in \mathbb{R}^{>0}+i\mathbb{R}/2\pi\mathbb{Z}.
\end{equation*}
In this section we discuss a refinement of this notion when the manifold is equipped with a torsion spin$^c$ structure; this refinement of complex length appears on the geometric side of the trace formula for spinors. We also discuss how the refinement can be computed in explicit examples, taking as input SnapPy's Dirichlet domain and complex length spectrum computations.  

\subsection{Torsion spin$^c$ structures and $G$} Let us begin by discussing with the simpler case of a genuine spin structure $\spin_0$ on $Y$. The choice of a spin structure $\spin_0$ allows us to lift the holonomy of a closed geodesic $\gamma:S^1\rightarrow Y$ from an element in $\mathbb{R}/2\pi\mathbb{Z}$ to an element in $\mathbb{R}/4\pi\mathbb{Z}$. From a topological perspective, this is because, using the trivial spin structure on the tangent bundle of $\gamma$, $\spin_0$ this induces a spin structure on the normal bundle of $\gamma$; and a spin structure on a rank $2$ real vector is simply a trivialization well defined up to adding an even number of twists (see \cite[Chapter IV]{Kir}). In particular, it makes sense to talk about $\mathrm{hol}(\gamma)/2$ as an element in $\mathbb{R}/2\pi\mathbb{Z}$.
\begin{remark}
Recall that (confusingly) the trivial spin structure on $S^1$ corresponds to the \textit{non-trivial} double cover of $S^1$; with this convention, we have $\mathrm{hol}(\gamma^n)/2=n\cdot (\mathrm{hol}(\gamma)/2)$ in $\mathbb{R}/2\pi\mathbb{Z}$. This does not hold if we choose the Lie spin structure.
\end{remark}
From the point of view of hyperbolic geometry, thinking of $\pi_1(Y)\subset\mathrm{Isom}^+(\mathbb{H}^3)\cong \mathrm{PSL}_2(\mathbb{C})$, a spin  structure is a lift
\begin{center}
\begin{tikzcd}
  & \SL_2(\mathbb{C})\arrow[d] \\
\pi_1(Y)\arrow[r,hook]\arrow[ur]& \mathrm{PSL}_2(\mathbb{C})
\end{tikzcd}
\end{center}
where the vertical map is the quotient by the subgroup $\{\pm1\}$. Notice that all orientable $3$-manifolds are spin because their second Stiefel-Whitney class $w_2$ vanishes \cite[Chapter $11$]{MilSta}, so such a lift always exists. Any two such lifts differ by a homomorphism $\pi_1(Y)\rightarrow \mathbb \{\pm1\}$, i.e. an element in $H^1(Y,\mathbb{Z}/2)$. From this viewpoint, recall that complex length of a hyperbolic element $\gamma$ is given by
\begin{equation*}
\C \ell(\gamma)=2\log \lambda_{\gamma}
\end{equation*}
where $\lambda_{\gamma}$ is the root of
\begin{equation*}
x^2-(\mathrm{tr}\gamma)x+1=0
\end{equation*}
with $|\lambda_{\gamma}|>1$. Of course, for an element $\gamma\in \mathrm{PSL}_2(\mathbb{C})$ the trace is well-defined only up to sign, and we see that the two choices of the trace correspond to $\pm\lambda_{\gamma}$; this implies that the argument of $ \lambda_{\gamma}^2$, which is $\mathrm{hol}(\gamma)$, is well defined modulo $2\pi$.
\par
It is then clear that a spin structure, which provides a well-defined trace, also fixes a choice between $\pm\lambda_{\gamma}$; the argument of this choice, which is $\mathrm{hol}(\gamma)/2,$ is well defined modulo $2\pi$. Very concretely, this is implies that the lift $\tilde{\gamma}\in \SL_2(\mathbb{C})$ is conjugate to
$$\left( \begin{array}{cc}
e^{\frac{\ell(\gamma)}{2}+i\frac{\mathrm{hol}(\gamma)}{2}}&0\\
0&e^{-\frac{\ell(\gamma)}{2}-i\frac{\mathrm{hol}(\gamma)}{2}}
\end{array} \right).$$
The case of general torsion spin$^c$ structures is analogous, where we use instead the group
\begin{equation*}
G=\{g \in \mathrm{GL}_2(\C):  |\det(g)|=1\}\subset \mathrm{GL}_2( \mathbb{C}).
\end{equation*}
There is a natural map
\begin{equation*}
\pi:G\rightarrow \mathrm{PSL}_2(\mathbb{C})
\end{equation*} 
obtained by expressing
\begin{equation*}
G=\left(\SL_2(\mathbb{C})\times U_1 \right)/\{\pm1\},
\end{equation*}
where $U_1\subset G$ is the subgroup of multiples of the identity, and projecting onto the first factor. A torsion spin$^c$ structure, together with a flat spin$^c$ connection, is then a homomorphism
\begin{equation*}
\varphi:\pi_1(Y)\rightarrow G
\end{equation*}
for which $\pi\circ\varphi:\pi_1(Y)\rightarrow \mathrm{PSL}_2(\mathbb{C})$ is the inclusion map. We see that two spin$^c$ structures then differ by a homomorphism
\begin{equation*}
\pi_1(Y)\rightarrow U_1,
\end{equation*}
which can be thought of as an element in $H^1(Y,\mathbb{R}/\mathbb{Z})$. The Bockstein long exact sequence for the coefficients
\begin{equation*}
0\rightarrow \mathbb{Z}\rightarrow\mathbb{R}\rightarrow \mathbb{R}/\mathbb{Z}\rightarrow 0
\end{equation*}
reads in this case
\begin{equation*}
H^1(Y;\mathbb{Z})\rightarrow H^1(Y,\mathbb{R})\rightarrow H^1(Y,\mathbb{R}/\mathbb{Z})\rightarrow H^2(Y,\mathbb{Z})\rightarrow H^2(Y,\mathbb{R}).
\end{equation*}
The latter allows us to interpret topological classes of torsion spin$^c$ structure as an affine space over the torsion of $H^2(Y;\mathbb{Z})$.
\par
Throughout the paper, we will study rational homology spheres, for which all spin$^c$ structures are torsion. In this case, we have identifications
\begin{equation*}
H^1(Y,\mathbb{Q}/\mathbb{Z})\cong H^1(Y,\mathbb{R}/\mathbb{Z})\cong H^2(Y;\mathbb{Z})
\end{equation*}
We will further fix a base spin structure $\spin_0$ and accordingly identify spin$^c$ structures with $H^1(Y,\mathbb{Q}/\mathbb{Z})$. For simplicity, $\varphi$ will interchangeably refer to both the lift and the difference homomorphism. In this case, the element $\gamma$ has lift conjugate to
$$\left(  \begin{array}{cc}
e^{i\varphi(\gamma)}e^{\frac{\ell(\gamma)}{2}+i\frac{\mathrm{hol}(\gamma)}{2}}&0\\
0&e^{i\varphi(\gamma)}e^{-\frac{\ell(\gamma)}{2}-i\frac{\mathrm{hol}(\gamma)}{2}}
\end{array} \right)$$
and the spin$^c$ length spectrum keeps track of the value $\varphi(\gamma)$. We will refer to the latter as the \textit{twisting character}.

\vspace{0.3cm}

\subsection{Computations of spin$^c$ lifts} For eventual use of the trace formula for spinors on $G,$ we need to explicitly compute the torsion spin$^c$ data, namely $\frac{\mathrm{hol}(\gamma)}{2}$ and $\phi(\gamma)$ for representatives of all conjugacy classes in $\Gamma$ up to some specified real length.  We refer to this data as the \emph{spin$^c$ length spectrum} of $(Y,\spin)$. We discuss a concrete algorithm to compute such information, taking as input data provided by SnapPy.

\subsubsection{Input for spin$^c$ length spectrum: SnapPy Dirichlet domain and complex length spectrum}
In order to compute the spin$^c$ length spectrum up to some real length threshold $R,$ we take as input the following objects (both pre-computable in SnapPy): 

\begin{itemize}
\item
A Dirichlet domain $D$ for $\Gamma=\pi_1(Y)$ centered at $o = (1,0,0,0)$ in the hyperboloid model of $\mathbb{H}^3,$ i.e. the upper sheet of 
$$\bigg\{ (t,x,y,z)\text{ such that } Q(t,x,y,z)= -1 \bigg\},$$
where $Q(t,x,y,z)= -t^2 + x^2 + y^2 + z^2$. In particular, this includes a list of all elements $\gamma$ of the identity component of the orthogonal group $\SO(Q)^0$ for which some domain within the bisector of $o$ and $\gamma o$ is a face $F \subset D.$  For hyperbolic 3-manifolds $Y,$ such elements $\gamma$ come in inverse pairs: for the above element $\gamma, \gamma^{-1} F$ is the face of $D$ opposite $F.$  The group $\Gamma$ is generated by the face-pairing elements (see \S \ref{lifttoSL2C}).
\item
For each conjugacy class $C$ in $\Gamma$ with $\ell(C)\leq R$, a matrix $g_C$ in $\SO(Q)^0,$ which is the image in $\SO(Q)^0$ of some element of $\Gamma$ representing $C.$ One can extract the complex length of $C$ from $g_C$, but of course it contains more information.
\end{itemize}

\subsubsection{Converting $\Gamma \subset \SO(Q)^0$ to $\Gamma \subset \mathrm{PSL}_2(\C)$}  \label{O31toSL2C}

Minkowski space $\mathbb{R}^{1,3}$ can be identified with the space of $2 \times 2$ hermitian matrices $i\mathfrak{u}_2$ via the map
$$(t,x,y,z) \mapsto \left( \begin{array}{cc} t + z & x - iy \\ x + iy & t - z \end{array} \right).$$
Under this identification, the quadratic form $Q$ equals the determinant of the above matrix.  The group $\PSL_2(\C)$ acts on $i\mathfrak{u}_2$ by 
$$g \cdot X = g X g^\ast,$$
preserving the determinant.  The resulting homomorphism  
$$\PSL_2(\C) \xrightarrow{\sim} \SO(Q)^0$$
is in fact an isomorphism.  To transform $\Gamma$ to a corresponding subgroup of $\PSL_2(\C)$ amounts to inverting the above isomorphism.  In turn, this amounts to solving the following linear algebra problem: 

Given $f \in \mathrm{End}(i\mathfrak{u}_2)$ preserving $Q,$ find $g \in \GL_2(\C)$ satisfying
\begin{equation} \label{invertingSL2toO31}
Xg^\ast = \left( \det(g) \cdot g^{-1} \right) f(X) \text{ for all } X \in i\mathfrak{u}_2.
\end{equation}

The entries of $\det(g) \cdot g^{-1}$ are linear in the entries of $g,$ so \eqref{invertingSL2toO31} defines a linear algebra problem (over $\mathbb{R}$), which is readily solved. Rescaling the solution $g$ to \eqref{invertingSL2toO31} as needed, we solve \eqref{invertingSL2toO31} with some matrix $g \in \PSL_2(\C).$ 

\vspace{0.3cm}
\subsubsection{Reduction theory via Dirichlet domains for conjugacy classes $C \subset \Gamma$} \label{dirichletdomainreduction}
We discuss how to express a given element $\gamma\in\Gamma$ in terms of the face-pairing elements $\gamma_i$ of $D$.
\begin{lemma} \label{distancereduction}
Suppose $x \in \mathbb{H}^3 \setminus D.$  There is some face pairing element $\gamma$ for which $d(\gamma^{-1} \cdot x, o) < d(x,o).$
\end{lemma}

\begin{proof}
Because $\gamma_1,\ldots,\gamma_m$ support the faces of the Dirichlet domain $D,$  
$$D = \{ x \in \mathbb{H}^3: d(x,o) \leq d(x, \gamma_i \cdot o) \text{ for } i =1,\ldots, m  \}.$$

In particular, if $x \notin D,$ then
$$d(x,o) > d(x, \gamma_i \cdot o) = d(\gamma_i^{-1} \cdot x, o)$$
for some face-pairing element $\gamma_i.$
\end{proof}

This yields the following \textbf{Dirichlet domain reduction algorithm for $\Gamma$}:
\begin{itemize}
\item
Initialize $x = \gamma \cdot o.$  Let $L = []$ be an empty list.  Note: if $\gamma$ is not the identity element, then $x \notin D.$

\item
While $x$ is not in $D$: find a face pairing element $\gamma$ for which $d(\gamma^{-1} x, o) < d(x,o)$; this is always possible by Lemma \ref{distancereduction}.  Then: 
\begin{itemize}
\item
Append $\gamma$ to the right end of $L.$ 
\item
Replace $x$ by $\gamma^{-1} x.$ 
\end{itemize}

\item
Once $x \in D,$ output $\delta = \prod_{g \in L} g,$ product taken in left-to-right order.
\end{itemize}

\begin{lemma}
The Dirichlet domain reduction algorithm for $\Gamma$ terminates.  Its output equals $\gamma.$
\end{lemma}

\begin{proof}
Let $x = \gamma \cdot o.$
If the while loop did not terminate, we would find an infinite sequence of face pairing elements $\gamma_1, \gamma_2, \ldots$ for which $x, \gamma_1^{-1} \cdot x, \gamma_2^{-1} \cdot \gamma_1^{-1} \cdot x, \ldots$ all do not lie in $D$ and for which
$$d(x,o) > d(\gamma_1^{-1} \cdot x, o) > d(\gamma_2^{-1} \cdot x,o) > \cdots$$
But the latter is not possible since the action of $\Gamma$ on $\mathbb{H}^3$ is properly discontinuous. Suppose now $L = [\gamma_1,\cdots,\gamma_n]$ at the point of the algorithm when $x \in D.$  That means that
$$\gamma_n^{-1} \gamma_{n-1}^{-1} \cdots \gamma_1^{-1} \cdot (\gamma \cdot o) \in D.$$
The latter point is evidently $\Gamma$-equivalent to $o$ and lies in $D.$  Since $o$ does not lie on the boundary of $D,$ $o$ is the unique point of $D$ which is $\Gamma$-equivalent to $o.$  So
$$\left( \gamma_n^{-1} \gamma_{n-1}^{-1} \cdots \gamma_1^{-1} \cdot \gamma \right) \cdot o = o.$$
Since the action of $\Gamma$ on $\mathbb{H}^3$ is fixed-point free, it follows that $\gamma = \gamma_1 \cdots \gamma_n.$\end{proof}

\subsubsection{Computing one spin structure} \label{lifttoSL2C}
As mentioned above, it is convenient to fix a base genuine spin structure $\spin_0$ (so that all the others spin$^c$ structures are obtained by twisting it via a character). To compute the corresponding lift $\tilde{\Gamma}\subset\mathrm{SL}_2(\C)$, we use the presentation of $\Gamma \subset \PSL_2(\C)$ afforded to us by the Dirichlet domain $D.$  Let $\Gamma_{\mathrm{abs}}$ denote the (abstract) group with a generator $[\gamma]$ for each face pairing element $\gamma$ and satisfying the following relations: 
\begin{itemize}
\item
(opposite face relations) $[\gamma] \cdot [\gamma^{-1}] = 1$ for all face pairing elements $\gamma$
\item
(edge cycle relations) The edges of $D$ are partitioned into \emph{edge cycles}.  This is an ordered sequence of edges $e_1,e_2, \ldots, e_n$ for which $\gamma_1 \cdot e_1 = e_2, \gamma_2 \cdot  e_2 = e_3, \ldots, \gamma_n \cdot e_n = e_1$ for some $\gamma_1,\gamma_2\dots,\gamma_n$. Every edge cycle yields a corresponding relation: $[\gamma_n] \cdots [\gamma_1] = 1.$
\end{itemize}
The opposite face and edge cycle relations hold for the corresponding elements of $\Gamma.$  The assignment
\begin{align} \label{poincarepolyhedron}
\iota: \mathrm{\Gamma_{\mathrm{abs}}} &\rightarrow \Gamma \subset \PSL_2(\C) \nonumber \\
[\gamma] &\mapsto \gamma
\end{align}
thus extends to a well-defined homomorphism.  According to the Poincar\'{e} polyhedron theorem \cite{Mas}, $\iota$ is an isomorphism.  
\bigskip

For every face-pairing element $\gamma \in \Gamma \subset \PSL_2(\C),$ choose arbitrarily a lift $\widetilde{\gamma} \in \SL_2(\C).$  Because $\iota$ is well-defined, all of the defining relations have an associated sign.  For example, if $[\gamma_n] \cdots [\gamma_1] = 1$ is an edge cycle relation, then
$$\widetilde{\gamma_n} \cdots \widetilde{\gamma_1} = \epsilon_{[\gamma_n] \cdots [\gamma_1]} \left( \begin{array}{cc} 1 & 0 \\ 0 & 1 \end{array} \right),$$
where $\epsilon_{[\gamma_n] \cdots [\gamma_1]} = \pm 1.$  Let $\epsilon_{[\gamma]} = \pm 1$ be unknowns we will solve for.   Then 
$$\gamma \mapsto \epsilon_{[\gamma]} \cdot \widetilde{\gamma}$$
extends to a well-defined lift of $\Gamma$ to $\widetilde{\Gamma} \subset \SL_2(\C)$ iff the opposite face and edge cycle relations are satisfied.  This immediately reduces to the following linear algebra problem (over $\mathbb{Z}/2\mathbb{Z}$):    
$$\begin{cases}
\epsilon_{[\gamma]} \epsilon_{[\gamma^{-1}]} = \epsilon_{[\gamma] \cdot [\gamma^{-1}]} &\text{ opposite face relations} \\
\epsilon_{[\gamma_n]} \cdots \epsilon_{[\gamma_1]} = \epsilon_{[\gamma_n] \cdots [\gamma_1]} &\text{ edge cycle relations}. \\
\end{cases}$$
We know abstractly that some solution must exist (because $Y$ is spin), and we readily solve for one.  

\begin{remark}
The collection of lifts $\widetilde{\Gamma}$ is a torsor for $H^1(\Gamma \backslash \mathbb{H}^3, \mathbb{Z}/2).$  No preferred lift exists in general.  We note, however, that $\SW,$ the protagonist of the second part of our paper, satisfies $H^1(\SW, \mathbb{Z} / 2) = 0,$ so $\pi_1(\SW)$ admits a unique lift to $\widetilde{\pi_1(\SW)} \subset \SL_2(\C).$   
\end{remark}
\vspace{0.3cm}

\subsubsection{Computing all homomorphisms $\Gamma \rightarrow \mathbb{Q} / \mathbb{Z}$} \label{smithnormalform}

From the previous subsection, $\Gamma_{\mathrm{abs}}$ gives a presentation for $\Gamma.$  Let $e_1,\cdots, e_m$ denote the images of the face-pairing elements $\gamma_1, \ldots, \gamma_m$ in the abelianization $\Gamma^{\mathrm{ab}}.$   The abelian group $\Gamma^{\mathrm{ab}}$ can be presented as
\begin{equation*}
\mathbb{Z}\langle e_1, e_2,\dots, e_m\rangle / R,
\end{equation*}
where $R$ are the abelian versions of the opposite face and edge cycle relations from Subsection \ref{lifttoSL2C}. We assume that $b_1(Y)=0$, so that $\Gamma^{\mathrm{ab}}=H_1(Y,\mathbb{Z})$ is a finite group. Using the Smith normal form, we compute a basis $b_1, \dots, b_m$ of $\mathbb{Z}\langle e_1, e_2,\dots, e_m\rangle$ for which a basis for the span of $R$ is given $d_1 b_1, \dots, d_m b_m$ for some non-zero integers $d_1, d_2 , \dots, d_m$   with $d_i\lvert d_{i+1}$. In particular,
\begin{equation*}
\Gamma^{\mathrm{ab}}=H_1(Y,\mathbb{Z})=\bigoplus \mathbb{Z}/d_i\mathbb{Z}.
\end{equation*}
Homomorphisms from $\Gamma^{\mathrm{ab}}$ to $\mathbb{Q} / \mathbb{Z}$ are uniquely determined by the images $\mathbf{y} := (y_1,\ldots,y_m)$ of $b_1,\ldots, b_m$; we define $\phi_{\mathbf{y}}$ as the one sending
\begin{equation*}
 b_i  \mapsto y_i \in \left(\frac{1}{d_i} \mathbb{Z}\right) / \mathbb{Z}.
\end{equation*}
If $A = (b_1 | \cdots | b_m)$ is the matrix with $j$th column $b_j,$ then 
$$\mathbf{y} A^{-1} = (\phi_{\mathbf{y}}(e_1), \ldots, \phi_{\mathbf{y}}(e_m)).$$

\subsubsection{Indexing all lifts of $\Gamma$ to $\widetilde{\Gamma} \subset G$}
We index the lifts of $\Gamma \subset \PSL_2(\C)$ to $\widetilde{\Gamma} \subset G$ as follows:
\begin{itemize}
\item
Compute one spin lift $s_0: \Gamma \rightarrow \Gamma_{\mathrm{spin}} \subset \SL_2(\C),$ by the procedure described in \S \ref{lifttoSL2C}. 

\item
A general lift of $\Gamma$ to $G$ is of the form $\gamma \mapsto \phi(\gamma) \cdot s_0(\gamma)$ for some twisting character $\phi: \Gamma \rightarrow \mathrm{U}_1.$
\end{itemize}

Thus,
$$\left\{ \gamma \mapsto e^{2\pi i \phi_{\mathbf{y}}(\gamma)} \cdot s(\gamma) : y_1 \in \left(\frac{1}{d_1} \mathbb{Z}\right) / \mathbb{Z}, \ldots, y_m \in \left(\frac{1}{d_m} \mathbb{Z}\right) / \mathbb{Z} \right\}$$ 

parametrizes all lifts of $\Gamma$ to $G.$

\vspace{0.3cm}
\subsubsection{Computing the spin$^c$ length spectrum}
Suppose we have computed one lift $s_0: \Gamma \rightarrow \SL_2(\C)$ specified by the images of the face-pairing generators for $D$; SnapPy provides the images in $\SO(Q)^0$ for all of the face-pairing generators, so we can apply the inverse isomorphism from \S \ref{O31toSL2C} to each face-pairing generator followed by the procedure described in \S \ref{lifttoSL2C} to compute a lift $s_0$. Suppose also that we have computed a homomorphism $\phi: \Gamma \rightarrow \mathbb{Q} / \mathbb{Z},$ specified by its values on the face-pairing generators, e.g. the homomorphisms $\phi_{\mathbf{y}}$ from \S \ref{smithnormalform}.
\\
\par
For every $\Gamma$-conjugacy class of translation length at most $R,$ SnapPy specifies the image in $\SO(Q)^0$ of some representative element $\gamma$ of its conjugacy class. Notice that $\gamma$ lies in the group generated by the face-pairing elements of the Dirichlet domain $D$. Applying the Dirichlet domain reduction algorithm from \S \ref{dirichletdomainreduction}, we can express $\gamma = \gamma_1 \cdots \gamma_k$ for some face-pairing elements $\gamma_1,\ldots,\gamma_k$ for $D.$  Having computed $s_0$ on the face-pairing generators, $s_0(\gamma) = s_0(\gamma_1) \cdots s_0(\gamma_k)$ is some explicit element of $\SL_2(\C).$  We readily compute its $\SL_2(\C)$-conjugacy class:
$$s_0(\gamma) \sim_{\SL_2(\C)} \left( \begin{array}{cc} e^{u/2} e^{i\theta} & 0 \\ 0 & e^{-u/2} e^{-i\theta} \end{array} \right),$$
and we have that $u=l(\gamma)$ and $\theta=\mathrm{hol}(\gamma)/2\in\mathbb{R}/(2\pi\mathbb{Z})$.
\\
\par
Finally, since the values of the twisting character $\phi$ are known on the face-pairing generators, we readily compute $\phi(\gamma) = \phi(\gamma_1) + \cdots + \phi(\gamma_k) =: \varphi(\gamma).$ In the lift $\widetilde{\Gamma} \subset G$ of $\Gamma$ corresponding to $s_0$ and $\phi$, we have
$$\widetilde{\gamma} \sim_G e^{2\pi i \varphi} \cdot \left( \begin{array}{cc} e^{u/2} e^{i\theta} & 0 \\ 0 & e^{-u/2} e^{-i\theta} \end{array} \right)  =  \left( \begin{array}{cc} e^{u/2} e^{i\theta} e^{2\pi i \varphi} & 0 \\ 0 & e^{-u/2} e^{-i\theta} e^{2\pi i \varphi} \end{array} \right).$$
Following this procedure for every conjugacy class $C$ of length at most $R$ computes the spin$^c$ length spectrum for the lift $\widetilde{\Gamma} \subset G$ of $\Gamma$ corresponding to the pair $s_0: \Gamma \rightarrow \SL_2(\C)$ and $\phi: \Gamma \rightarrow \mathbb{Q} / \mathbb{Z}.$

\vspace{0.5cm}
\section{Trace formulas for functions, forms and spinors}\label{traceformulas}
While our previous work \cite{LL} only used the trace formula to sample coexact 1-form eigenvalues with even test functions, the present paper requires that we extend our toolkit to sample the eigenvalue spectrum for functions and spinors (the latter using both even and odd test functions).  In this section, we collect statements of the trace formula specialized to the latter three contexts. For the purposes of our work, the various statements can be treated as a black box, and their detailed proofs can be found in the appendices to the paper. For simplicity, in the statements we will restrict to smooth compactly supported functions, but we will ultimately need to apply the trace formula using more general test functions; this is made precise in Subsection \ref{regularity}. 

\subsection{Notation and conventions}
\begin{itemize}
\item 
Throughout the paper, we will use the convention for Fourier tranforms
\begin{equation*}
\widehat{H}(t)=\int_{\mathbb{R}} H(x)e^{-itx}dx.
\end{equation*}

\item
For a closed geodesic $\gamma,$ we denote by $\gamma_0$ a prime geodesic of which $\gamma$ is a multiple of. 

\item
When dealing with the odd trace formula, it will be important to have a clear orientation convention, as for example the spectrum of $\ast d$ on $Y$ and $\bar{Y}$ are opposite to each other. We use the identification $\mathbb{H}^3=\mathrm{PSL_2}/\mathrm{PSU_2}$ (obtained by thinking of the upper half plane model), for which the tangent space at $(0,1)$ is $i\mathfrak{su}(2)$. We declare the (physicists') Pauli matrices
\begin{equation*}
\sigma_1 = \left( \begin{array}{cc} 0 & 1 \\ 1 & 0  \end{array} \right) \quad\sigma_2 = \left( \begin{array}{cc} 0 & -i \\ i & 0  \end{array} \right) \quad
\sigma_3 = \left( \begin{array}{cc} 1 & 0 \\ 0 & -1  \end{array} \right)
\end{equation*}
to be a positively oriented basis.
\end{itemize}

\subsection{Trace formula for functions} 
We begin with the classical trace formula for functions. Here, to each eigenvalue
\begin{equation*}
0=\lambda_0<\lambda_1\leq \lambda_2\leq\dots
\end{equation*}
we associate the parameter $r_n\in\mathbb{R}\cup i[0,1]$ for which $\lambda_n=1+r_n^2$ (in particular, $\lambda_0=i$). Non-zero eigenvalues less than $1$ (i.e. such that $r_n\in i(0,1)$ ) are referred to as \textit{small}; the value $1$ is important because it is the bottom of the $L^2$-spectrum of the Laplacian on $\mathbb{H}^3$ (see \cite{Cha}).
\begin{thm}\label{funformula}For $H\in C^{\infty}_c(\mathbb{R})$ an even test function, the identity
\begin{equation*}
\sum_{n=0}^{\infty} \widehat{H}(r_n)= \frac{\mathrm{vol}(Y)}{2\pi} \cdot \left( - H''(0) \right) +\sum_{[\gamma] \neq 1}\frac{\ell(\gamma_0)}{|1-e^{\mathbb{C}\ell(\gamma)}||1-e^{-\mathbb{C}\ell(\gamma)}|}H(\ell(\gamma))
\end{equation*}
holds.
\end{thm}
This formula is well-known, and essentially follows from the computation in \cite[Appendix B]{LL}. It can also be proved by direct integration via the Abel transform and integral kernels, see for example \cite{Cha} (beware of an erroneous extra factor of $2$).
\vspace{0.3cm}

\subsection{Trace formula for coexact $1$-forms}
In the case of coexact $1$-forms, each eigenvalue
\begin{equation*}
0<\lambda_1^*\leq \lambda_2^*\leq\dots
\end{equation*}
is of the form $\lambda_n^*=t_j^2$ for some $j$, where 
\begin{equation*}
\dots\leq t_{-1}<0< t_0\leq t_1\leq t_2\leq \dots
\end{equation*}
are the eigenvalues of $\ast d.$ In our previous paper, we were only concerned with the absolute values $|t_j|$ of the parameters and so even test functions sufficed for our purposes.  In the present paper, however, the signs of the parameters $t_j$ are crucial.  Below, we state a variant of the trace formula which samples the coexact 1-form eigenvalue spectrum using odd test functions and is thus sensitive to these signs.    
\begin{thm}\label{coexformula}For $H\in C^{\infty}_c(\mathbb{R})$ an even test function, the identity
\begin{equation*}
\frac{1}{2}(b_1(Y)-1) \cdot \widehat{H}(0)+\frac{1}{2}\sum \widehat{H}(t_j)=\frac{\mathrm{vol(Y)}}{2\pi} \cdot (H(0)-H''(0))+\sum \ell(\gamma_0)\frac{\cos(\mathrm{hol}(\gamma))}{|1-e^{\mathbb{C}\ell(\gamma)}||1-e^{-\mathbb{C}\ell(\gamma)}|}H(\ell(\gamma)).
\end{equation*}
holds. For $K\in C^{\infty}_c(\mathbb{R})$ an odd test function, the identity
\begin{equation*}
\frac{1}{2}\sum \widehat{K}(t_j)=i \sum \ell(\gamma_0)\frac{\sin(\mathrm{hol}(\gamma))}{|1-e^{\mathbb{C}\ell(\gamma)}||1-e^{-\mathbb{C}\ell(\gamma)}|}K(\ell(\gamma)).
\end{equation*}
holds.
\end{thm}
For the last identity, recall that if $K$ is odd then its Fourier transform is purely imaginary.

\vspace{0.3cm}
\subsection{Trace formula for spinors} As in Section \ref{spinlength}, we fix a base spin structure $\spin_0$ and consider its twist by a character $\varphi$. The corresponding Dirac operator $D_{B_0}$ has discrete spectrum
\begin{equation*}
\dots\leq s_{-1}<0\leq s_0\leq s_1\leq s_2\leq \dots
\end{equation*}
unbounded in both directions.  Also, recall that we introduced the notation $\tilde{\gamma}$ for the lift of $\gamma$ corresponding to the base spin structure $\spin_0$.
\begin{thm}\label{spinorformula}For $H\in C^{\infty}_c(\mathbb{R})$ an even test function, the identity
\begin{equation*}
\frac{1}{2}\sum \widehat{H}(s_j)=\frac{\mathrm{vol(Y)}}{2\pi} \cdot \left(\frac{1}{4}H(0)-H''(0) \right)+\sum \ell(\gamma_0)\frac{\cos(\mathrm{hol}(\tilde{\gamma}))\cos(\varphi({\gamma}))}{|1-e^{\mathbb{C}\ell(\gamma)}||1-e^{-\mathbb{C}\ell(\gamma)}|}H(\ell(\gamma)).
\end{equation*}
holds. For $K\in C^{\infty}_c(\mathbb{R})$ an odd test function, the identity
\begin{equation*}
\frac{1}{2}\sum \widehat{K}(s_j)=i \sum \ell(\gamma_0)\frac{\sin(\mathrm{hol}(\tilde{\gamma})) \cos(\varphi({\gamma})) }{|1-e^{\mathbb{C}\ell(\gamma)}||1-e^{-\mathbb{C}\ell(\gamma)}|}K(\ell(\gamma)).
\end{equation*}
holds.
\end{thm}

\subsection{Allowing less regular functions}\label{regularity}
While for simplicity we stated the formulas only for smooth, compactly supported functions, the same conclusions hold with less regularity. For example, in our previous work \cite[Appendix $C$]{LL} we showed that the trace formula for coexact $1$-forms holds for functions of the form $\left(\mathbf{1}_{[-a,a]} \right)^{\ast 4}$, the fourth convolution power of the indicator function of the interval $[-a,a]$. Rather than providing general statements, let us point out two specific instances that will be used in this work:
\begin{enumerate}
\item the even trace formulas hold for $H(x)=\left(\mathbf{1}_{[-a,a]} \right)^{\ast k}$, with $k\geq 6$ and the odd trace formulas hold for $K=H'$. This follows directly from the results in \cite[Appendix $C$]{LL}.
\item the even trace formulas hold for a Gaussian $H(x)=e^{-x^2/2c}$, and the odd trace formulas hold for $K=H'$. This is proved in Appendix \ref{gaussiantest}.
\end{enumerate}
In the proof of Theorem \ref{thm1} we will only use functions of the first type, but our proof of Theorem \ref{thm2} uses Gaussians.

\vspace{0.5cm}
\section{Analytic continuation of the eta function via the odd trace formula}\label{continuation}
The classical approach to the analytic continuation of the eta function (\ref{eta1form}) involves the Mellin transform and the asymptotic expansion of the trace of the heat kernel; we start by quickly reviewing the fundamental aspects of this for the reader's convenience. For our purposes, we will discuss a different interpretation using the odd trace formulas described in the previous section.

\subsection{Analytic continuation via the Mellin transform}\label{mellin}
For the reader's convenience, we recall basic facts about the Mellin transform and its relevance for the definition of the $\eta$ invariant. Our account loosely follows \cite[Section $3.4$]{Cha}, suitably adapted to the simpler situation we are dealing with.  Given a continuous function $f:\mathbb{R}^{>0}\rightarrow \mathbb{R}$, we define its \textit{Mellin transform} to be
\begin{equation*}
M_f(s)=\int_0^{\infty} f(t)t^s \; \frac{dt}{t}.
\end{equation*}
This is essentially the Fourier transform for the multiplicative group $\mathbb{R}^{>0}$ equipped with the Haar measure $dt/t$. Assuming that $f(t)=O(t^{-a})$ for $t\rightarrow 0$, and $f(t)=O(t^{-n})$ for $t\rightarrow \infty$ for all $n$, we have that the integral absolutely converges on the half-plane $\mathrm{Re}(s)>a$, and that $M_f$ is a holomorphic function there.
\par
Under additional assumptions, $M_f$ admits an explicit analytic continuation to the whole plane. For our purposes, suppose for example that $f$ admits an asymptotic expansion
\begin{equation}\label{asymptoticexpansion}
f(t)\sim \sum_{k=0}^{\infty}c_k t^{i_k} \quad\text{ for }t\rightarrow 0
\end{equation}
where $\{i_k\}$ is a strictly increasing sequence of real numbers with $\lim i_k=+\infty$. Then $M_f$ extends to a meromorphic function on $\mathbb{C}$ with a simple pole at $-i_k$ with residue $c_k$ for all $k$, and holomorphic everywhere else. To see this, choose $L>0$ and let us write $M_f(s)=M_-(s)+M_+(s)$ where
\begin{align*}
M_-(s) &=\int_0^{L} f(t)t^s \; \frac{dt}{t}, \\
M_+(s) &=\int_L^{\infty} f(t)t^s \; \frac{dt}{t}.
\end{align*}
Our assumption on the behavior of $f$ at infinity implies that $M_+$ is an entire function. Regarding the first integral, the asymptotic expansion  (\ref{asymptoticexpansion}) implies that for a given $N$ we can write
\begin{equation*}
f(t)=\sum_{k=0}^{N-1} c_k t^{i_k}+r_N(t)
\end{equation*}
with $r_N(t)=O(t^{i_N})$ for $t\rightarrow 0$. Integrating, for $\mathrm{Re}(s) > -i_N$ we can therefore write
\begin{align} \label{mellinanalyticcontinuation}
M_-(s) &=\int_0^{L} f(t)t^s \; \frac{dt}{t} \nonumber \\
&=\sum_{k=0}^{N-1}\frac{c_k \cdot L^{s + i_k} }{s+i_k}+\int_0^L r_N(t)t^s \; \frac{dt}{t}.
\end{align}
This expression provides a meromorphic continuation of $M_{-}$ to the half-plane $\mathrm{Re}(s)>-i_N.$  This continuation has simple poles at $-i_0,\dots -i_{N-1}$ with residues $c_0,\dots c_{N-1}$ and is holomorphic everywhere else in the half-plane. Because our choice of $N$ was arbitrary, this proves the claim.

\begin{remark}
Though $M_{-}$ and $M_{+}$ both depend on $L,$ one readily checks that their sum does not.
\end{remark}

The above discussion readily generalizes to the case in which $f(t)=O(t^{-n})$ for some fixed $n$ when $t\rightarrow \infty$. In this case, the analytic continuation (obtained again by breaking the integral $M_f(s)=M_-(s)+M_+(s)$) defines a meromorphic continuation on the half-plane $\mathrm{Re}(s)<n$.

\vspace{0.3cm}
\subsection{The heat kernel}\label{heatkernel}
We briefly recall the classical analytic continuation of the eta function of the odd signature operator
\begin{equation*}
\eta_{\mathrm{sign}}(s)=\sum_{t_n\neq0} \frac{\sgn(t_n)}{|t_n|^s}.
\end{equation*}
using the asymptotic expansion of the heat kernel, see \cite{APS1} for details. The key observation to relate this to the previous discussion is the identity
\begin{equation}\label{gamma}
|a|^{-s}=\frac{1}{\Gamma(s)}\int_0^{\infty} e^{-|a|t}t^s \; \frac{dt}{t},
\end{equation}
where $\Gamma(s)=\int_0^{\infty} e^{-t}t^s \; \frac{dt}{t}$ is the classical Gamma function.
\\
\par
In the simpler case of the Riemann zeta function, we have for $\mathrm{Re}(s)>1$ (which ensures absolute convergence) the identity
\begin{align*}
\zeta(s) &=\sum_{n=1}^{\infty}\frac{1}{n^s} \\
&=\sum_{n=1}^{\infty} \frac{1}{\Gamma(s)}\int_0^{\infty} e^{-nt}t^s \; \frac{dt}{t} \\
&=\frac{1}{\Gamma(s)}\int_0^{\infty}\frac{1}{e^t-1}t^s \; \frac{dt}{t}.
\end{align*}
Recall that $\Gamma(s)$ has poles at $s=0,-1,-2,\dots$ with residue $(-1)^k/k$ at $s=-k$. Given the asymptotic expansion for $t\rightarrow 0$ in terms of Bernoulli numbers
\begin{equation*}
\frac{1}{e^t-1}\sim \sum_{k=-1}^{\infty}\frac{B_{k+1}}{(k+1)!} t^k,
\end{equation*}
and ${1}/({e^t-1})=O(t^{-n})$ for all $n$ when $t\rightarrow \infty$, we recover the classical fact that $\zeta(s)$ admits a meromorphic extension with only one pole at $s=1$.
\\
\par
Let us now focus on the case of the eta function of the odd signature operator (the case of the Dirac operator is analogous). In \cite{APS1}, the authors study the quantity\footnote{To be precise, they study the odd signature operator on \textit{coclosed} $1$-forms, which leads to an additional term involving $b_1(Y)$ in their discussion.}
\begin{equation*}
K(t)=-\sum \frac{\sgn(t_n)}{2}\mathrm{erfc}(|t_n|\sqrt{t})
\end{equation*}
where
\begin{equation*}
\mathrm{erfc}(t)=\frac{2}{\sqrt{\pi}}\int_t^{\infty}e^{-\xi^2}d\xi
\end{equation*}
is the complementary error function. As $0$ is not an eigenvalue by assumption, Weyl's law (cf. Section \ref{localweyl}) implies that $K(t)\rightarrow 0$ exponentially fast for $t\rightarrow \infty$. The derivative of $K(t)$ is computed to be
\begin{align*}
K'(t)&=\frac{1}{\sqrt{4\pi t}}\sum t_n e^{-t_n^2t}.
\end{align*}
Using $(\ref{gamma})$ and integrating by parts we get the formula
\begin{align}\label{etaheat}
\eta_{\mathrm{sign}}(2s) &=\frac{2\sqrt{\pi}}{\Gamma(s+1/2)}\int_0^{\infty} K'(t)t^{s+1} \; \frac{dt}{t} \nonumber \\
&=-\frac{2s\sqrt{\pi}}{\Gamma(s+1/2)}\int_0^{\infty} K(t)t^s \; \frac{dt}{t}.
\end{align}
The key point is that $K(t)$ admits an asymptotic expansion of the form
\begin{equation}\label{expansionK}
K(t)\sim \sum_{k\geq -n}a_k t^{k/2} \quad \text{ for }t\rightarrow 0,
\end{equation}
and therefore its Mellin transform has at worst a simple pole at $0$. In $(\ref{etaheat})$, this pole is canceled by the term $s$ in the numerator, and we see therefore that $\eta_{\mathrm{sign}}(s)$ is holomorphic near $0$. We conclude by mentioning that the asymptotic expansion $(\ref{expansionK})$ follows by interpreting $K$ as the trace of the heat kernel on $\mathbb{R}^{\geq 0}\times Y$ with the (now-called) APS boundary conditions.

\vspace{0.3cm}
\subsection{Analytic continuation via the trace formula}
While the analytic continuation via the heat kernel we have just discussed holds in general, for the specific case of hyperbolic three-manifolds we can take a different route using the odd trace formula. Let us discuss first in detail the case of the eta invariant for the odd signature operator (the case of the Dirac operator is essentially the same). Recall that in this case we only need to consider the spectral asymmetry of $\ast d$ acting on coexact $1$-forms (\ref{eta1form}). Given an even test function $H$ (in our case $H$ will either be a Gaussian or compactly supported, as in Subsection \ref{regularity}), we consider the trace formula applied to the odd test function $H'$. Recalling that with our convention
\begin{equation*}
\widehat{H'}(t)=it \widehat{H}(t),
\end{equation*}
we obtain the identity
\begin{equation} \label{traceformulaodd} 
\sum_n t_n \widehat{H}(t_n) = 2\sum_{1 \neq [\gamma]} \ell(\gamma_0) \cdot \frac{\sin( \mathrm{hol}(\gamma) )}{|1 - e^{\mathbb{C} \ell(\gamma)} | \cdot |1 - e^{-\mathbb{C} \ell(\gamma)} |} \cdot H'( \ell(\gamma) ).
\end{equation}
For a fixed even test function $G$, and given $T>0$, we apply this to the function $H(x)=G(x/T)$, for which $\widehat{H}(t)=T\widehat{G}(Tt)$, and obtain the family of identities
\begin{equation} \label{traceformulaodd} 
\sum_n Tt_n \widehat{G}(Tt_n) = 2\sum_{1 \neq [\gamma]} \ell(\gamma_0) \cdot \frac{\sin( \mathrm{hol}(\gamma) )}{|1 - e^{\mathbb{C} \ell(\gamma)} | \cdot |1 - e^{-\mathbb{C} \ell(\gamma)} |} \cdot \frac{1}{T}G' \left( \frac{\ell(\gamma)}{T} \right).
\end{equation}
We will denote either side of the identity by $G_T$ (and refer to them as spectral and geometric respectively). Now, for any real number $a\neq0$ we have the identity

\begin{align*}
\int_0^{\infty} aT\widehat{G}(aT)T^{s-1}dT&= a\cdot\int_0^{\infty} \widehat{G}(aT)T^sdT\\
&= a\cdot\int_0^{\infty} \widehat{G}(|a|T)T^sdT\\
&=\frac{\sgn(a)}{|a|^s}\cdot \int_0^{\infty} \widehat{G}(T)T^s dT,
\end{align*}
where we used that $\widehat{G}$ is even in the second equality and performed a change of variables in the last equality above. Therefore, for $\mathrm{Re}(s)$ large enough (so that the sums converge absolutely), we have
\begin{align} \label{etaGT}
\eta_{\mathrm{sign}}(s) &=\sum_{t_n}\frac{\sgn(t_n)}{|t_n|^s} \nonumber \\
&=\frac{\int_0^{\infty} \sum \left(t_n T\widehat{G}(t_nT)\right)T^s \; \frac{dT}{T}}{\int_0^{\infty} \widehat{G}(T)T^s dT} \nonumber \\
&=\frac{\int_0^{\infty} G_T T^s \; \frac{dT}{T}}{\int_0^{\infty} \widehat{G}(T)T^s dT}.
\end{align}
We can recognize at the numerator and denominator the Mellin transforms of $G_T$ and $T\widehat{G}(T)$ respectively. We have the following.

\begin{lemma}\label{functionsok}
Suppose that $G$ is either:
\begin{itemize}
\item a Gaussian function $x \mapsto e^{-x^2/2c}$.
\item $G$ is the $k$-th convolution power of an indicator function $\frac{1}{2a}1_{[-a,a]}$, for some even $n\geq 6$. 
\end{itemize}
Then the identity
\begin{equation*}
\eta_{\mathrm{sign}}(s)=\frac{\int_0^{\infty} G_T T^s \frac{dT}{T}}{\int_0^{\infty} T \widehat{G}(T) T^s \frac{dT}{T}}
\end{equation*}
holds in an open domain in $\mathbb{C}$ containing $s=0$, where we interpret the right hand side using the Mellin transform.
\end{lemma}
Of course, the lemma is valid in much greater generality, but for simplicity we have restricted to the class of functions that will be used in the rest of the paper.
\begin{proof}Let us start with the case of a Gaussian. We claim that the quantity $G_T$ is rapidly decaying for both $T\rightarrow 0$ and $T\rightarrow \infty$. The former follows by looking at the definition of $G_T$ using the geometric side because by the prime geodesic theorem the number of prime closed geodesics of length less than $x$ is $O(e^{2x}/2x)$ (cf. Section \ref{geodesicSW}); for the latter, it follows by looking at its definition using the spectral side, because the Weyl law implies that the number of spectral parameters less than $t$ in absolute value is $O(t^3)$ (cf. Section \ref{localweyl}), and $\widehat{G}$ is still a Gaussian. By the properties of the Mellin transform discussed in subsection \ref{mellin}, this implies that the numerator $\int_0^{\infty} G_T T^s \frac{dT}{T}$ is an entire function of $s$. On the other hand, the integrand in the denominator $T\widehat{G}(T)$ is rapidly decaying for $T\rightarrow \infty$, and by looking at the Taylor series at $t=0$ one obtains an asymptotic expansion with only odd exponents. Hence the denominator admits an analytic continuation to the whole plane which is regular and non-zero at $0$ because $\int_0^{\infty}\widehat{G}(T)dT>0$. The claim then follows by uniqueness of the analytic continuation.
\par
The case of $(\frac{1}{2a}1_{[-a,a]})^{\ast k}$, with $k\geq 6$ even is analogous. Assume for simplicity $a=1$. First of all, the derivative of this function is regular enough for the odd trace formula to hold, see Section \ref{regularity}. Notice that in this case
\begin{equation*}
\widehat{G}(t)=\left(\frac{\sin(t)}{t}\right)^k.
\end{equation*}
Looking at the spectral side, we see then that
\begin{equation*}
\lvert G_T\lvert \leq \frac{1}{T^{k-1}}\cdot\sum_n \frac{1}{|t_{n}|^{k-1}}=O(T^{-k+1})
\end{equation*}
for $T\rightarrow +\infty$. Furthermore $G_T$ is still rapidly decaying for $T\rightarrow 0$, because $G$ is compactly supported on the geometric side. The numerator is therefore a holomorphic function on the half-plane $\mathrm{Re}(s)<k-1$. As in the case of Gaussians, $T\widehat{G}(T)$ has a Taylor expansion at $0$ with only odd exponent terms, and $\int_0^{\infty} T\widehat{G}(T) \frac{dT}{T}>0$; it is also $O(T^{-k+1})$ for $T\rightarrow\infty$. The denominator is therefore a meromorphic function on $\mathrm{Re}(s)<k-1$, regular and non-vanishing at zero, and we conclude as in the case of the Gaussian.
\end{proof}

Let us unravel the statement of the above result in a way that is useful for concrete computations. If we choose a cutoff $L>0$, we have the formula
\begin{equation}\label{etazero}
\eta_{\mathrm{sign}}=\frac{\int_0^{L} G_T \; \frac{dT}{T}+\int_L^{\infty} G_T \; \frac{dT}{T}}{\int_0^{\infty} T \widehat{G}(T) \frac{dT}{T}}
\end{equation} 
where at the numerator we evaluate the integral near zero using the geometric definition, and the integral near infinity using the spectral definition. Explicitly, for the eta invariant of the odd signature operator:  
\begin{align*}
\int_0^{L} G_T \frac{dT}{T} &=\sum \ell(\gamma_0)\cdot\frac{2\sin(\mathrm{hol}(\gamma))}{|1 - e^{\mathbb{C}l(\gamma)}| \cdot |1 - e^{-\mathbb{C}l(\gamma)}|} \int_0^L \frac{1}{T} G' \left(\frac{\ell}{T} \right) \frac{dT}{T} \\
&= \sum \ell(\gamma_0)\cdot\frac{2\sin(\mathrm{hol}(\gamma))}{|1 - e^{\mathbb{C}l(\gamma)}| \cdot |1 - e^{-\mathbb{C}l(\gamma)}|} \cdot \frac{1}{\ell(\gamma)} \cdot \int_{\ell / L}^\infty y G'(y) \frac{dy}{y} \\
&= \sum \ell(\gamma_0)\cdot\frac{2\sin(\mathrm{hol}(\gamma))}{|1 - e^{\mathbb{C}l(\gamma)}| \cdot |1 - e^{-\mathbb{C}l(\gamma)}|} \cdot \frac{1}{\ell(\gamma)} \cdot \int_{\ell / L}^\infty G'(y) dy \\
&=  \sum \ell(\gamma_0)\cdot\frac{2\sin(\mathrm{hol}(\gamma))}{|1 - e^{\mathbb{C}l(\gamma)}| \cdot |1 - e^{-\mathbb{C}l(\gamma)}|} \cdot \frac{1}{\ell(\gamma)} \cdot \left(- G \left( \frac{\ell(\gamma)}{L} \right) \right).
\end{align*}
and
\begin{align*}
\int_L^{\infty} G_T \frac{dT}{T} &=\sum \int_L^{\infty} Tt_n\widehat{G}(Tt_n)\frac{dT}{T}\\
&=\sum \mathrm{sgn}(t_n)\int_{L|t_n|}^{\infty} T \widehat{G}(T) \frac{dT}{T}\\
&=\sum \mathrm{sgn}(t_n)\int_{L|t_n|}^{\infty}  \widehat{G}(T){dT},
\end{align*}
where we made a simple substitution in the integral.
\\
\par
Finally, the case of spinors follows in the same way by using the identity
\begin{align*}
G_T &=\sum_n Ts_n \widehat{G}(Ts_n) \\
&= 2\sum_{1 \neq [\gamma]} \ell(\gamma_0) \cdot \frac{\sin( \mathrm{hol}(\tilde{\gamma}) )\cdot {\cos(\varphi_{\tilde{\gamma}})} }{|1 - e^{\mathbb{C} \ell(\gamma)} | \cdot |1 - e^{-\mathbb{C} \ell(\gamma)} |} \cdot \frac{1}{T}G' \left( \frac{\ell(\gamma)}{T} \right),
\end{align*}
instead. The only additional observation is that the computation in (\ref{etaGT}) still holds even when the Dirac operator has kernel. In fact, we have
\begin{align*}
\eta_{\mathrm{Dir}}(s) &=\sum_{s_n\neq 0}\frac{\sgn(s_n)}{|s_n|^s} \nonumber \\
&=\frac{\int_0^{\infty} \sum \left(s_n T\widehat{G}(s_nT)\right)T^s \; \frac{dT}{T}}{\int_0^{\infty} \widehat{G}(T)T^s dT} \nonumber \\
&=\frac{\int_0^{\infty} G_T T^s \; \frac{dT}{T}}{\int_0^{\infty} \widehat{G}(T)T^s dT}
\end{align*}
as is the second row the parameters $s_n$ equal to zero do not contribute.

\vspace{0.5cm}
\section{Effective local Weyl laws}\label{localweyl}
In this section, we discuss upper bounds for the number of spectral parameters in a given interval for our operators of interest.  These upper bounds are expressed purely in terms of injectivity radius and volume, and we refer to such bounds as \textit{local Weyl laws}, see also \cite{Mul}. 
\par

To place this in context, recall the classical Weyl law: for every dimension $n,$ there exists a constant $C_n$ such that for every Riemannian $n$-manifold $(X,g)$, the asymptotic
\begin{equation}\label{Weyl}
\#\{\text{eigenvalues $\lambda$ of $\Delta_g$ with $\sqrt{\lambda}\leq T$}\}\sim C_n\cdot \mathrm{vol}_g(X)\cdot T^{n}.
\end{equation}
holds, where $\Delta_g$ is the Hodge Laplacian acting on functions. Analogous versions hold more generally for squares of Dirac type operators, such as the Hodge Laplacian $(d+d^*)^2$ acting on $k$-forms or the Dirac Laplacian $D_{B_0}^2$, see \cite{BGV}. From Equation (\ref{Weyl}), one expects the following local version of Weyl law to hold
\begin{equation*}
\#\{\text{eigenvalues $\lambda$ of $\Delta_g$ with $\sqrt{\lambda}\in[T,T+1]$}\}\sim n\cdot C_n\cdot \mathrm{vol}_g(X)\cdot T^{n-1}.
\end{equation*}
However, for our purposes it will be important not to just understand the asymptotic behavior of the number of the eigenvalues in a given interval, but also to provide effective upper bounds on it; in particular, our goal is to prove upper bounds of the form
\begin{equation} \label{localweyllawgeneral}
\#\{\text{eigenvalues $\lambda$ of $\Delta_g$ with $\sqrt{\lambda}\in[T,T+1]$}\}\leq A\cdot T^{n-1}+B
\end{equation}
and similar bounds for the other operators we are interested in, namely the odd signature and Dirac operators. For our application, it is crucial to express the upper bound in \eqref{localweyllawgeneral} (and its analogues for all operators we study) \emph{uniformly in the parameter $T,$ where the constants $A$ and $B$ are expressed explicitly in terms of the geometry of $X$}.  The leading constant $A$ appearing in our upper bounds will generally be larger than the optimal one $C'_n\cdot \mathrm{vol}_g(X).$
\\
\par
To prove these local Weyl laws, we will evaluate the trace formulas using even test functions of support so small that the sum over the closed geodesics vanishes. The estimates we will prove are effective in terms of volume and injectivity radius, but not of a nice form in terms of the input; for this reason, at the end of our computations we will specialize to the concrete case in which $\mathrm{inj}>0.15$ and $\mathrm{vol}<6.5$, e.g. the case of a manifold in the Hodgson-Weeks census. We will also assume throughout that $b_1=0$.

\subsection{Preliminaries on test functions}Fix a value of $R>0$ (which will later be a given lower bound on the injectivity radius). Choose an even function $\varphi$ with support in $[-2,2]$ and non-negative Fourier transform. Again, our convention is
\begin{equation*}
\widehat{\varphi}(t)=\int \varphi(x)e^{-ixt}dt.
\end{equation*}
To simplify our discussion, we will make the following:
\\
\par
\textbf{Assumption.} We will assume for simplicity throughout the section that both $\varphi$ and $\widehat{\varphi}$ achieve their maximum at $0$, and that the minimum $m_R$ of $\widehat{\varphi}$ in $[-R,R]$ is achieved at $\pm R$.
\begin{example}
For the specialization $R=0.15,$ we will choose $\varphi_0=\beta\ast\beta$, where
\begin{equation*}
\beta(x)=
\begin{cases}
e^{-1/(1-x^2)} \text{ if }|x|<1\\
0\text{ otherwise.}
\end{cases}
\end{equation*}
This function satisfies the assumptions, and $m_R\sim 0.19643$.
\end{example}
\begin{remark}
Of course, this implies that $\varphi_0$ satisfies the assumption also for $R<0.15$. To obtain a function that satisfies the assumption for large values of $R$, we can look at functions of the form $\varphi_0(Cx)$.
\end{remark}

Consider $\varphi_R(x)=\varphi \left( \frac{x}{R} \right)$, which is supported in $[-2R,2R]$. We have
\begin{align*}
\varphi_R(0) &=\varphi(0), \\ 
\varphi_R''(0) &=\frac{1}{R^2}\varphi''(0),
\end{align*}
and
\begin{align*}
\widehat{\varphi_R}(t) &=\int \varphi \left( \frac{x}{R} \right)e^{-ixt}dx \\
&=\int \varphi(x')e^{iRx't}Rdx' \\
&=R\widehat{\varphi}(Rt).
\end{align*}
Consider also
\begin{align*}
\varphi_{R,\nu} &= \varphi_R(x)(e^{i\nu x}+e^{-i\nu x}) \\
&=2\varphi_R(x) \cos(\nu x).
\end{align*}
Then
\begin{align*}
\varphi_{R,\nu}(0) &=2\varphi(0), \\  
\varphi_{R,\nu}''(0) &=2 \left(\frac{1}{R^2}\varphi''(0)-\nu^2\varphi(0) \right),
\end{align*}
and
\begin{equation*}
\widehat{\varphi_{R,\nu}}(t)=\widehat{\varphi_R}(t+\nu)+\widehat{\varphi_R}(t-\nu).
\end{equation*}

\subsection{Local Weyl law for coexact $1$-forms}
We begin with the case of coexact $1$-forms, which is the simplest to analyze. For $\nu\geq0$, denote by $\delta^*(\nu)$ the number of spectral parameters $t_j$ with $|t_j|\in[\nu,\nu+1]$. We evaluate the trace formula in Theorem \ref{coexformula} for the test function $\varphi_{R,\nu}(x)$ with $R$ less than the injectivity radius of $Y$. As $\varphi_{R,\nu}$ is supported in $[-2R,2R]$, and the injectivity radius is exactly half the length of the shortest geodesic of $Y$, the sum over the geodesics vanishes, and we have the identity
\begin{equation*}
\frac{\mathrm{vol}(Y)}{2\pi}\left(\varphi_{R,\nu}(0)-\varphi_{R,\nu}^{''}(0)\right)+\frac{1}{2}\widehat{\varphi_{R,
\nu}}(0)=\sum_j\frac{1}{2}\widehat{\varphi_{R,\nu}}(|t_j|).
\end{equation*}
By the identities of the previous subsection we therefore obtain
\begin{align*}
&\frac{\mathrm{vol}(Y)}{\pi}\left(\varphi(0)-\frac{1}{R^2}\varphi''(0)+\nu^2\varphi(0)\right)+{R}\widehat{\varphi}(R\nu) \\
&= \frac{1}{2}\sum_j \widehat{\varphi_R}(|t_j|+\nu)+\widehat{\varphi_R}(|t_j|-\nu)\\
&\geq \sum_j\frac{1}{2}\widehat{\varphi_R}(|t_j|-\nu)\\
&\geq \frac{1}{2}R\cdot m_R \cdot\delta^*(\nu)
\end{align*}
using first that $\widehat{\varphi}_{R}$ is non-negative, and then that
\begin{align*}
\widehat{\varphi_R}(t-\nu) &=R \cdot \widehat{\varphi}(R \cdot (t-\nu)) \\
&\geq R \cdot m_R
\end{align*}
for $t\in[\nu,\nu+1]$. For any choice of suitable test function $\varphi$, this provides the following upper bound on $\delta^*(\nu)$ in terms of $\mathrm{vol}(Y)$ and injectivity radius:
\begin{equation*}
\delta^*(\nu)\leq\left(\frac{2 \cdot \mathrm{vol(Y)}\cdot\varphi(0)}{\pi R \cdot m_R}\right)\nu^2+\frac{2}{R \cdot m_R}\left(R\widehat{\varphi}(0)+\frac{\mathrm{vol}(Y)}{\pi}\left(\varphi(0)-\frac{1}{R^2}\varphi''(0)\right)\right)
\end{equation*}
where we used that $\widehat{\varphi}$ achieve its maximum at $0$. Working out the computations using our test function $\varphi_0$, we obtain the following.
\begin{prop}
Let $Y$ be a hyperbolic rational homology sphere with $\mathrm{vol}(Y)<6.5$ and $\mathrm{inj}(Y)>0.15$. Then the inequality
\begin{equation*}
\delta^*(\nu)\leq 18.7\nu^2+2577.3
\end{equation*}
holds.
\end{prop}

\begin{remark}\label{betterbound}
The bound obtained with this approach gets worse as $R$ goes to zero; the same is true for the spectral density on functions and spinors. One way to obtain significantly better estimates in this case is to consider a Margulis number $\mu>0$ for $Y$. Recall that for such a number, the set of points of $Y$ with local injectivity radius $<\mu/2$ is the disjoint union of tubes around the finitely many closed geodesics with length $<\mu$. For example, in \cite{Mey} it is shown $\mu=0.1$ is a Margulis number for all closed oriented hyperbolic three-manifolds. As the number of closed geodesics shorter than $\mu$ can be bounded above in terms of the volume \cite{GMM}, one obtains estimates for the spectral density by considering a test function supported in $[-\mu,\mu]$. We will not pursue the exact output of this approach in the present work.
\end{remark}
\vspace{0.3cm}

\subsection{Local Weyl law for eigenfunctions}\label{localweylfun}
The case of functions is more involved because of the possible appearance of small eigenvalues (i.e. the ones corresponding to imaginary parameters $r_j$). Set $\delta(\nu)$ to be the number of parameters in $[\nu,\nu+1]$, and let $\delta_s$ be the number of small eigenvalues. We apply Theorem \ref{funformula} choosing as before $R$ to be less than the injectivity radius, again the sum over geodesics vanishes, and we get the identity
\begin{equation*}
\sum \widehat{\varphi_{R,\nu}}(r_n)=-\frac{\mathrm{vol}(Y)}{2\pi}\varphi_{R,\nu}''(0)
\end{equation*}
Here we recall that $r_n\in \mathbb{R}^{\geq0}\cup i[0,1]$ corresponds to the eigenvalue $\lambda_n=1+r_n^2$.
\par
We start by bounding the number of small eigenvalues. For this purpose, we set $\nu=0$. For real $t$, we have
\begin{align*}
\widehat{\varphi}(it) &=\int_{\mathbb{R}} \varphi(x)e^{-tx}dx \\
&= \int_{\mathbb{R}} \varphi(x) \cosh(tx) dx
\end{align*}
which as a function of $t \in \mathbb{R}$ is non-negative, even, convex, and has minimum at $0$. Therefore we have
\begin{align*}
-\frac{\mathrm{vol}(Y)\varphi''(0)}{2\pi R^2} &=\sum \widehat{\varphi_R}(r_n) \\
&\geq \sum_{r_n \text{ imaginary}} \widehat{\varphi_R}(r_n) \\
&= R\cdot \sum_{r_n \text{ imaginary}} \widehat{\varphi}(Rr_n) \hspace{0.5cm} \\
&\geq R\widehat{\varphi}(0)\cdot \delta_s,
\end{align*}
so that
\begin{equation*}
\delta_s\leq- \frac{\mathrm{vol}(Y)\varphi''(0)}{2\pi R^3 \widehat{\varphi}(0)}.
\end{equation*}

To bound large eigenvalues, we look at $\widehat{\varphi_{R,\nu}}$ for $\nu\neq0$. Unfortunately, this is not necessarily positive on the imaginary axis. On the other hand, we have for $t\in[0,1]$
\begin{align*}
|\widehat{\varphi_{R,\nu}}(it)|&=\left| 2\int_{-2R}^{2R}\varphi(x/R)\cos(\nu x)e^{-xt}dx\right| \hspace{0.5cm} (\text{because } \varphi \text{ is supported on } [-2,2])\\
&=2R\left| \int_{-2}^{2}\varphi(y)\cos(\nu R y)e^{-Ryt}dy\right|\\
&\leq 2R\varphi(0) \int_{-2}^2 e^{-Rty}dy \\
&=2\varphi(0)\cdot \frac{e^{2Rt}-e^{-2Rt}}{t} \\
&\leq 2\varphi(0)(e^{2R}-e^{-2R})
\end{align*}
where the last inequality holds because $(e^{2Rt}-e^{-2Rt})/t$ is increasing for $t\in[0,1]$. The trace formula then implies
\begin{align*}
& \frac{\mathrm{vol(Y)}}{\pi}\left(\nu^2\varphi(0)-\frac{1}{R^2}\varphi''(0)\right)+2\varphi(0)\cdot(e^{2R}-e^{-2R})\delta_s \\
&\geq \sum_{\text{real}}\widehat{\varphi_{R,\nu}}(r_n)\\
&\geq\sum_{\text{real}} \widehat{\varphi_R}(r_n-\nu)\\
&\geq R \cdot m_R\cdot \delta(\nu)
\end{align*}
so that
\begin{equation*}
\delta(\nu)\leq \left(\frac{\mathrm{vol(Y)}\cdot\varphi(0)}{\pi R \cdot m_R}\right)\nu^2+\frac{1}{R \cdot m_R}\left(-\frac{\mathrm{vol(Y)}\varphi''(0)}{\pi R^2}+   2\varphi(0)\cdot({e^{2R}-e^{-2R}})\delta_s\right).
\end{equation*}

Concretely, using again the test function $\varphi_0$, we obtain.
\begin{prop}
Let $Y$ be a hyperbolic rational homology sphere with $\mathrm{vol}(Y)<6.5$ and $\mathrm{inj}(Y)>0.15$. Then the inequalities
\begin{align*}
\delta_s&\leq 637\\
\delta(\nu)&\leq 9.4\nu^2+4782
\end{align*}
hold.
\end{prop}

\vspace{0.3cm}

\subsection{Local Weyl law for spinors}
This case is essentially the same as that of coexact $1$-forms. For a given spin$^c$ structure, denote by $\delta^D(\nu)$ the number of Dirac eigenvalues with absolute value in $[\nu,\nu+1]$. Choosing again $R=\mathrm{inj}(Y)$, we have
\begin{equation*}
\frac{\mathrm{vol}(Y)}{2\pi}\left(\frac{1}{4}\varphi_{R,\nu}(0)-\varphi_{R,\nu}^{''}(0)\right)=\sum_j\frac{1}{2}\widehat{\varphi_{R,\nu}}(s_j).
\end{equation*}
so that
\begin{equation*}
\frac{\mathrm{vol}(Y)}{\pi}\left(\frac{1}{4}\varphi(0)-\frac{1}{R^2}\varphi''(0)+\nu^2\varphi(0)\right)\geq \frac{1}{2}R\cdot m_R \cdot\delta^D(\nu)
\end{equation*}
and
\begin{equation*}
\delta^D(\nu)\leq \left(\frac{2\mathrm{vol}(Y)\cdot\varphi(0)}{\pi R\cdot m_R}\right)\nu^2+\frac{2\mathrm{vol(Y)}}{\pi R\cdot m_R}\left(\frac{1}{4}\varphi(0)-\frac{1}{R^2}\varphi''(0)\right).
\end{equation*}

\bigskip

Specializing using the test function $\varphi_0,$ we obtain:
\begin{prop}
Let $Y$ be a hyperbolic rational homology sphere with $\mathrm{vol}(Y)<6.5$ and $\mathrm{inj}(Y)>0.15$. Then for any spin$^c$ structure, the inequality
\begin{equation*}
\delta^D(\nu)\leq 18.7\nu^2+2561.3
\end{equation*}
holds.
\end{prop}

\vspace{0.5cm}
\section{Geometric bounds for the Fr\o yshov invariant}\label{proofthm1}

In this section we will use the local Weyl laws from the previous section to prove bounds on the eta invariants in terms of volume and injectivity radius. Combining this with Proposition \ref{froyeta}, we will be able to prove Theorem \ref{thm1}. For simplicity of notation, assume that we have the inequalities
\begin{align*}
\delta^*(\nu)&\leq A\nu^2+B \hspace{0.6cm} (\text{coexact 1-forms})\\
\delta_s&\leq C \hspace{1.75cm} (\text{0-form eigenvalues} < 1)\\
\delta(\nu)&\leq \frac{A}{2}\nu^2+D \hspace{0.5cm} (\text{0-form eigenvalues} \geq 1)\\
\delta^D(\nu)&\leq A\nu^2+E. \hspace{0.5cm} (\text{spinor eigenvalues})
\end{align*}
where the specific values of the constants for a manifold with $\mathrm{vol}<6.5$ and $\mathrm{inj}>0.15$ were determined in the previous section.

\subsection{Bounds on $\eta_{\mathrm{sign}}$.} \label{etasign} Recall from Subsection \ref{continuation} (after setting $L=1$): for an arbitrary admissible test function $G$,

\begin{equation*}
\eta_{\mathrm{sign}}=\frac{\int_0^{1} G_T \frac{dT}{T}+\int_1^{\infty} G_T \frac{dT}{T}}{\int_0^{\infty} \widehat{G}(T) dT}
\end{equation*} 
where
\begin{align*}  
G_T &=\sum_n Tt_n \widehat{G}(Tt_n) \\
&= 2\sum_{1 \neq [\gamma]} \ell(\gamma_0) \cdot \frac{\sin( \mathrm{hol}(\gamma) )}{|1 - e^{\mathbb{C} \ell(\gamma)} | \cdot |1 - e^{-\mathbb{C} \ell(\gamma)} |} \cdot \frac{1}{T}G' \left( \frac{\ell(\gamma)}{T} \right).
\end{align*}
For our purposes, it is convenient to restrict our attention to
$G(x)=(\frac{1}{2}\mathbf{1}_{[-1,1]})^{\ast k}$ for even $k\geq6$. This function has support in $[-k,k]$, and its Fourier transform is $\widehat{G}(t)=\mathrm{sinc}^k(t)$ where $\mathrm{sinc}(t) := \sin(t)/t$. We will denote
\begin{align*}
c_k :=\int_{0}^{\infty}\mathrm{sinc}^k(t) \; dt \\ 
d_k :=\int_{1}^{\infty}\mathrm{sinc}^k(t) \; dt.
\end{align*}
Let us consider the numerator of the expression for $\eta_{\mathrm{sign}}$; we will evaluate the first term using the geometric side of the trace formula and the second term using the spectral side. Starting with the second term, we have
\begin{align*}
\int_1^{\infty}G_T\frac{dT}{T} &=\int_{1}^{\infty}\sum_{t_n} Tt_n \cdot \mathrm{sinc}^k(Tt_n) \frac{dT}{T} \\
&=\int_{1}^{\infty}\sum_{t_n} \mathrm{sgn}(t_n) \cdot T |t_n| \cdot \mathrm{sinc}^k(T |t_n|) \frac{dT}{T} \\
&=\sum_{t_n} \mathrm{sgn}(t_n) \cdot \int_{|t_n|}^{\infty}\mathrm{sinc}^k(T)dT.
\end{align*}
We split the last sum in two parts $\sum_{|t_n|\leq 2}+\sum_{|t_n|>2}$. In absolute value, the first part can be bounded above as
\begin{align*}
\sum_{|t_n|\leq 2} \int_{|t_n|}^{\infty}\mathrm{sinc}^k(T)dT &\leq \delta^*(0)\cdot\int_0^{\infty}\mathrm{sinc}^k(T)dT+\delta^*(1)\cdot\int_1^{\infty}\mathrm{sinc}^k(T)dT \\
&=c_k\cdot B+d_k\cdot(A+B).
\end{align*}
For the second part, using $\mathrm{sinc}^k(T)\leq T^{-k}$, we have
\begin{align*}
\sum_{|t_n|\geq 2} \int_{|t_n|}^{\infty}\mathrm{sinc}^k(T)dT&\leq \sum_{|t_n|\geq 2} \int_{|t_n|}^{\infty}\frac{1}{T^k}dT \\
&=\frac{1}{k-1}\sum_{|t_n|\geq 2}\frac{1}{|t_n|^{k-1}} \\
&=\frac{1}{k-1}\sum_{m=2}^{\infty}\sum_{|t_n|\in[m,m+1]}\frac{1}{|t_n|^{k-1}} \\
&\leq\frac{1}{k-1} \sum_{m=2}^{\infty}\delta^*(m)\frac{1}{m^{k-1}}\\
&\leq \frac{1}{k-1}\sum_{m=2}^{\infty}(Am^2+B)\frac{1}{m^{k-1}}\\
&=\frac{1}{k-1}\left(A(\zeta(k-3)-1)+B(\zeta(k-1)-1)\right)
\end{align*}
and therefore
\begin{equation*}
\left|\int_1^{\infty}G_T\frac{dT}{T}\right|\leq  B\cdot c_k+(A+B)\cdot d_k+\frac{1}{k-1}\left(A(\zeta(k-3)-1)+B(\zeta(k-1)-1)\right).
\end{equation*}
For the geometric side, using a simple substitution, we have
\begin{align*}
\int_0^{1}G_T\frac{dT}{T} &= 2\int_0^1\sum \ell(\gamma_0)\frac{\sin(\mathrm{hol}(\gamma))}{|1-e^{\mathbb{C} \ell(\gamma)}||1-e^{-\mathbb{C} \ell(\gamma)}|}\frac{G' \left( \frac{\ell(\gamma)}{T} \right)}{T}\frac{dT}{T} \\
&=-2\sum \ell(\gamma_0)\frac{\sin(\mathrm{hol}(\gamma))}{|1-e^{\mathbb{C}\ell(\gamma)}||1-e^{-\mathbb{C}\ell(\gamma)}|}\frac{G(\ell(\gamma))}{\ell(\gamma)}
\end{align*}
As $\ell(\gamma)\geq 2\mathrm{inj}(Y)$, we have by taking absolute values
\begin{equation*}
\left|\int_0^{1}G_T\frac{dT}{T}\right|\leq \frac{1}{\mathrm{inj}(Y)}\sum \ell(\gamma_0)\frac{1}{|1-e^{\mathbb{C}\ell(\gamma)}||1-e^{-\mathbb{C}\ell(\gamma)}|}{G(\ell(\gamma))}.
\end{equation*}
To proceed, we notice that the right hand side is (up to a constant) the geometric side of the trace formula for functions, and we can provide bounds in terms of the spectral density of eigenfunctions. In particular, using Theorem \ref{funformula}, we have
\begin{align*}
\sum \ell(\gamma_0)\frac{1}{|1-e^{\mathbb{C}\ell(\gamma)}||1-e^{-\mathbb{C}\ell(\gamma)}|}{G(\ell(\gamma))} &=\frac{\mathrm{vol}(Y)}{2\pi}G''(0)+\sum_{r_n} \widehat{G}(r_n)\\
&\leq \sum_{r_n} \widehat{G}(r_n)
\end{align*}
where we used that $G''(0)$ is negative. We deal with imaginary and real parameters separately. For imaginary parameters, we have
\begin{equation*}
\widehat{G}(it)=\mathrm{sinc}^k(it)=\frac{\sinh^k(t)}{t^k}
\end{equation*}
which is increasing in $t$ and therefore we obtain the inequality
\begin{align*}
\sum_{\text{imaginary }r_n}\widehat{G}(r_n) &\leq \sinh(1)^k\cdot\delta_s \\
&\leq C\cdot\sinh(1)^k.
\end{align*}
For the real parameters, using $\widehat{G}\leq 1$ and $\widehat{G}\leq 1/x^k$ respectively we have
\begin{align*}
\sum_{\text{real } r_n}\widehat{G}(r_n)&=\sum_{r_n\in[0,1]}\widehat{G}(r_n)+\sum_{m=1}^{\infty}\sum_{r_n\in[m,m+1]}\widehat{G}(r_n) \\
&\leq \delta(0)\cdot 1+\sum_{m=1}^{\infty} \delta(m)\cdot\frac{1}{m^k}\\
&\leq D+\sum_{m=1}^{\infty}\frac{(A/2)m^2+D}{m^k}\\
&=\frac{A}{2}\zeta(k-2)+D(\zeta(k)+1).
\end{align*}
Hence, putting everything in \S \ref{etasign} together, we have
\begin{equation*}
\left|\int_0^{1} G_T\frac{dT}{T}\right|\leq \frac{1}{\mathrm{inj}(Y)}\left(C\cdot\sinh(1)^k+\frac{A}{2}\zeta(k-2)+D(\zeta(k)+1)\right).
\end{equation*}
Putting everything together, we obtain the following.
\begin{prop}
For every even $k\geq 6$ the inequality
\begin{align*}
\lvert\eta_{\mathrm{sign}}(Y)\lvert\leq \frac{1}{c_k}\bigg(B\cdot c_k+(A+B)\cdot d_k+\frac{1}{k-1}(A(\zeta(k-3)-1)+B(\zeta(k-1)-1))\\
+\frac{1}{\mathrm{inj}(Y)}(C\cdot\sinh(1)^k+\frac{A}{2}\zeta(k-2)+D(\zeta(k)+1))\bigg)
\end{align*}
holds, where we used the notation introduced above.
\end{prop}
Choosing for example $k=8$, and plugging in the constants we found in the previous section, we obtain the following effective estimate.
\begin{cor}\label{etasignbound}
If $Y$ is a hyperbolic rational homology sphere with $\mathrm{vol}(Y)<6.5$ and $\mathrm{inj}(Y)>0.15$, then $\lvert\eta_{\mathrm{sign}}(Y)\lvert\leq 108267$.
\end{cor}

\vspace{0.3cm}

\subsection{Bounds on $\eta_{\mathrm{Dir}}$} The discussion for the case of spinors is identical, with the final result obtained by substituting the constant $D$ with $E$. This is because the quantity to take it is based on the identity
\begin{align*}  
G_T &=\sum_n Ts_n \widehat{G}(Ts_n) \\
&= 2\sum_{1 \neq [\gamma]} \ell(\gamma_0) \cdot \frac{\sin( \mathrm{hol}(\tilde{\gamma}) )\cdot {\cos(\varphi_{\tilde{\gamma}})} }{|1 - e^{\mathbb{C} \ell(\gamma)} | \cdot |1 - e^{-\mathbb{C} \ell(\gamma)} |} \cdot \frac{1}{T}G' \left( \frac{\ell(\gamma)}{T} \right).
\end{align*}
Here we will again bound the integral $\int_1^{\infty}G_T\frac{dT}{T}$ using the bound for the spectral density $\delta^D(\nu)$, and the integral $\int_0^{1}G_T\frac{dT}{T}$ using the trace formula for functions. The final result is the following:
\begin{prop}
For every even $k\geq 6$ the inequality
\begin{align*}
\lvert\eta_{\mathrm{Dir}}(Y)\lvert\leq \frac{1}{c_k}\bigg(E\cdot c_k+(A+E)\cdot d_k+\frac{1}{k-1}(A(\zeta(k-3)-1)+E(\zeta(k-1)-1))\\
+\frac{1}{\mathrm{inj}(Y)}(C\cdot\sinh(1)^k+\frac{A}{2}\zeta(k-2)+D(\zeta(k)+1))\bigg)
\end{align*}
holds for all spin$^c$ structures.
\end{prop}
Setting again $k=8$, we have the following effective estimate.
\begin{cor}\label{etadiracbound}
If $Y$ is a hyperbolic rational homology sphere with $\mathrm{vol}(Y)<6.5$ and $\mathrm{inj}(Y)>0.15$, then for every spin$^c$ structure $\lvert\eta_{\mathrm{Dir}}(Y)\lvert\leq 108249$.
\end{cor}

\vspace{0.3cm}

\subsection{Proof of Theorem \ref{thm1}} Given our discussion, Theorem \ref{thm1} follows directly from the fact that for a minimal hyperbolic $L$-space
\begin{equation*}
h(Y,\spin)=-\frac{\eta_{\mathrm{sign}}}{8}-\frac{\eta_{\mathrm{Dir}}}{2},
\end{equation*}
see Proposition \ref{froyeta}.

\vspace{0.5cm}
\section{An explicit example: the Weeks manifold}\label{weeks}

In order to prove Theorem \ref{thm1}, we used test functions supported in $[-2\cdot\mathrm{inj},2\cdot\mathrm{inj}]$, as the only geometric input we had about the length spectrum was the injectivity radius. On the other hand, in specific examples one can access very concrete information about the length spectrum, and therefore one can apply the trace formula to a much larger class of test functions. In turn, one can use this to provide explicit computation of Fr\o yshov invariants. In this section, we show how this approach can be implemented in a simple example of minimal hyperbolic $L$-space, the Weeks manifold $\W$; a similar approach works also for other small volume hyperbolic minimal $L$-spaces in the Hodgson-Weeks census we discussed in \cite{LL}. 
\bigskip

Recall that $H_1(\W,\mathbb{Z})=(\mathbb{Z}/5\mathbb{Z})^2$. In our discussion, we will not discuss explicitly error bounds, to keep the section streamlined; we will deal with rigorous estimates of errors in our approximations when dealing with the Seifert-Weber dodecahedral space in the proof of Theorem \ref{thm2}. To this end, one can think of this section as both a warm-up exercise and a sanity check - the latter because the Fr\o yshov invariants of $W$ can be computed directly with other purely topological methods. We have indeed the following.
\begin{prop}\label{froyweeks}
The unique spin structure on $\W$ has Fr\o yshov invariant $-1/2$. For the remaining $24$ spin$^c$ structures the Fr\o yshov invariant either $\pm1/10$, or $\pm3/10$, and each value is the invariant of exactly six spin$^c$ structures.
\end{prop}
This is essentially showed in \cite{MO}, using the fact that $\W$ is the branched double cover of the knot $9_{49}$ \cite{MedVes}. The key point is that the $9_{49}$ differs from an alternating knot only for an extra twist, and the authors showed that under favorable circumstances this allows to compute the Fr\o yshov invariants in terms of a Goeritz matrix for the knot (generalizing the method of Ozsv\'ath and Szab\'o \cite{OSalt}, which applies to alternating knots).
\\
\par
Going back to our spectral approach, the signature eta invariant was computed in \cite{CGHW} to be
\begin{equation*}
\eta_{\mathrm{sign}}=0.040028711\dots\footnote{A good approximation of this value can also be computed directly using our approach; we will see this in detail in the case of the Seifert-Weber dodecahedral space in the next section.}
\end{equation*}
We compute the eta invariant for the Dirac operators using our explicit expression from Section \ref{continuation}. In order to do so, we first compute the spin$^c$ length spectra up to cutoff $R=6.5$ using the algorithm described in Section \ref{spinlength} (applied to data obtained from SnapPy). We then take the approach from \cite{LL} and use the even trace formula to obtain information about the spectrum. More specifically, using ideas of Booker and Strombergsson, for each spin$^c$ structure $\spin$ we determine an explicit function
\begin{equation*}
J_{\spin}:\mathbb{R}^{\geq0}\rightarrow \mathbb{R}^{\geq0}
\end{equation*}
with the property that if $\pm s$ are eigenvalues of the Dirac operator whose multiplicities add to $m$, then $J_{\spin}(|s|)\geq m$. The pictures in the range $[0,6]$ can be found in Figure \ref{WeeksDirac}; Class $1$ is the class of the spin structure, and the remaining $24$ spin$^c$ structures can be grouped in four groups (each consisting of six elements) with identical picture.
\begin{remark}
In fact, this is a consequence of the action of the isometry group $D_{12}$ on the set of spin$^c$ structures \cite{MedVes}. We will exploit symmetry under the isometry group in the more complicated case of $\SW$ in Section \ref{proofthm2}.
\end{remark}

\begin{figure}
\centering
\begin{subfigure}{.6\linewidth}
\includegraphics[width=\linewidth]{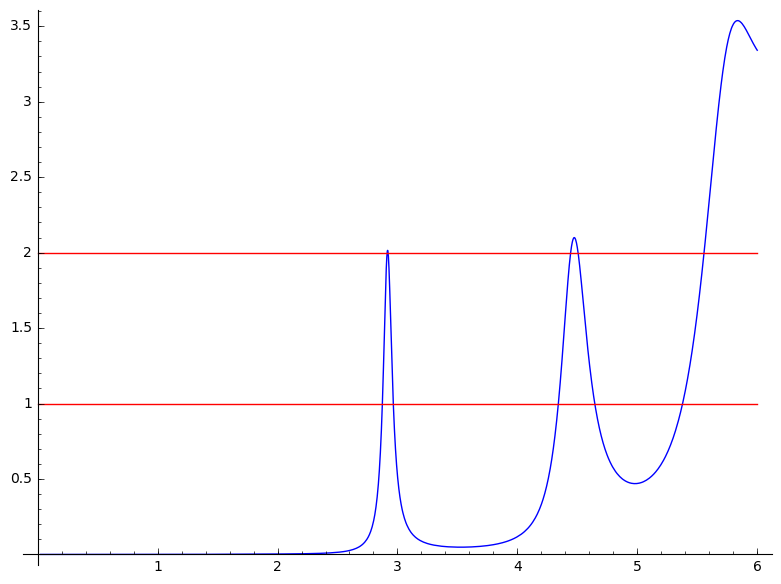}
\caption{Class $1$ (the spin structure)}
\end{subfigure}

\begin{subfigure}{.45\linewidth}
\includegraphics[width=\linewidth]{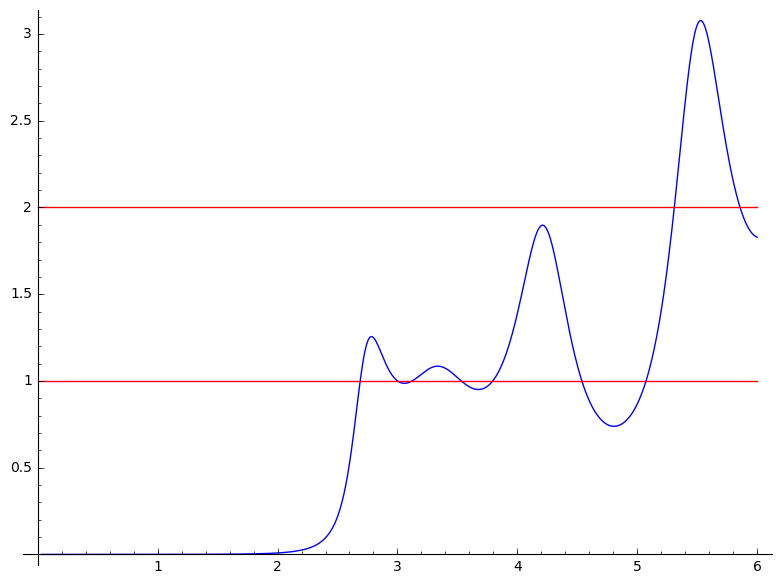}
\caption{Class $2$}
\end{subfigure}
\begin{subfigure}{.45\linewidth}
\includegraphics[width=\linewidth]{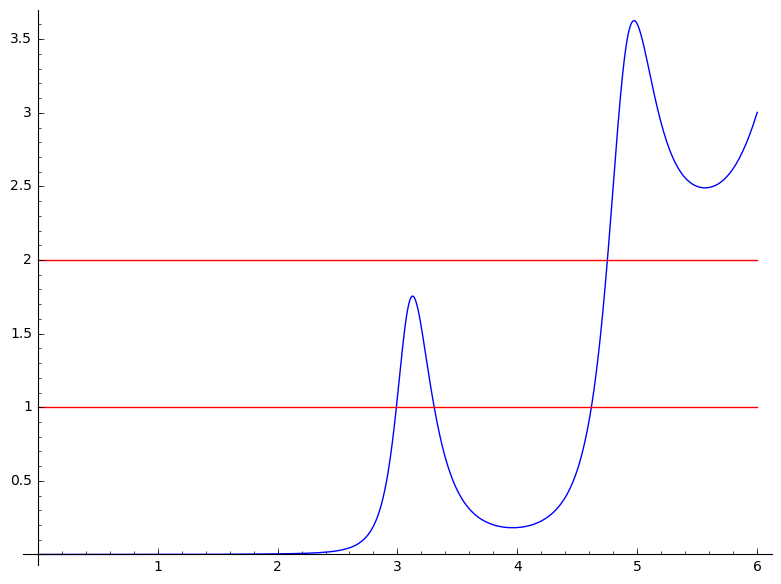}
\caption{Class $3$}
\end{subfigure}
\begin{subfigure}[b]{.45\linewidth}
\includegraphics[width=\linewidth]{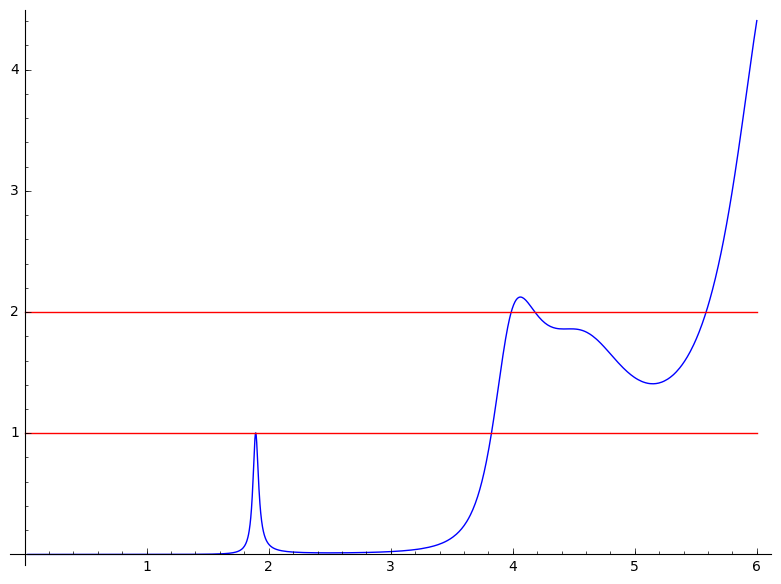}
\caption{Class $4$}
\end{subfigure}
\begin{subfigure}[b]{.45\linewidth}
\includegraphics[width=\linewidth]{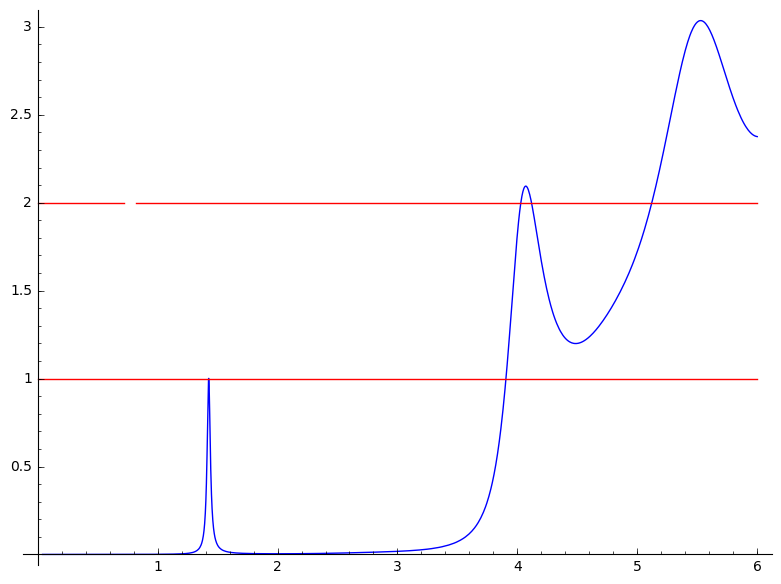}
\caption{Class $5$}
\end{subfigure}
\caption{Pictures for the Weeks manifold}
\label{WeeksDirac}
\end{figure}

We compute an approximation to $\eta_{\mathrm{Dir}}$ via the formula (\ref{etazero}) taking $G$ to be a Gaussian, see Lemma \ref{functionsok}. In this process, we only consider the sum over geodesics with length $\leq 6.5$, and only sum over eigenvalues where a precise guess can be made. More specifically:
\begin{itemize}
\item In the Classes $1$, $4$, $5$, we consider the contribution of the smallest eigenvalue as suggested by the picture (and consider the rest of the spectrum as an error term). Let us point out that can in principle use the method of \cite{LL2} to actually prove that an eigenvalue with the right multiplicity in the tiny window suggested by the picture; here we also use the fact that in the spin case eigenvalues always have even multiplicity, because the corresponding Dirac operator in quaternionic linear. The sign of the eigenvalue can also be determined by using the odd trace formula in Theorem \ref{spinorformula}.
\item In Classes $2$ and $3$, where a precise value for the first eigenvalue is not available, we simply consider the whole spectrum to be an error term, and take a more dilated Gaussian (which has a narrower Fourier transform) to make such error smaller. This works well because there is a good spectral gap (but increases of course the error coming from truncating the sum over the geodesics).
\end{itemize} 
Using the formula (\ref{etazero}) with $L=1$, we obtained the following approximate values for $h=-\eta_{\mathrm{sign}}/8-\eta_{\mathrm{Dir}}/2$:
\begin{center}
\begin{tabular} { | c | c |}
\hline
Spin$^c$ class & Approximate value of $h$  \\
\hline
1 &  $-0.49997\dots$  \\
\hline
2 & $0.09971\dots$ \\
\hline
3 & $-0.09994\dots$\\
\hline
4 & $0.29994\dots$\\
\hline
5 & $-0.29994\dots$\\
\hline
\end{tabular}
\end{center}
which are very close to the exact values in Proposition \ref{froyweeks}. 
\\
\par
In order to make this approximate computation into an actual proof of Proposition \ref{froyweeks}, we need to provide estimates on the terms we did not take into account in the sum. Let us discuss the various steps:
\begin{enumerate}
\item We know that the Fr\o yshov invariants are rational numbers, and in fact information about their fractional part can be obtained by purely topological means. For example, in the case of Weeks, one can show a priori that all Fr\o yshov invariants have the form $(2n+1)/10$. Given this, one only needs to prove that the errors are less than $\frac{1}{20}$ in order to conclude. In fact, by taking into account additional topological information coming from the linking form, one only needs to prove much weaker estimates for the error; this will be crucial in our proof of Theorem \ref{thm2}.
\item To bound the error resulting from truncating the sum over geodesics, we need to have a good understanding of the geodesics with length in a certain interval; in a specific example, this can be done effectively using the trace formula for functions (in the same spirit as the Prime geodesic theorem with errors \cite{Bus}).
\item To bound the error on the spectral side, the key observation is that large eigenvalues contribute very little to the error. For example (thinking about Classes $1$, $4$ and $5$), picking $G(x)=e^{-x^2/2}$ (for which $\widehat{G}(t)=\sqrt{2\pi}e^{-x^2/2}$ and $\int_0^{\infty} \widehat{G}(t)dt=\pi$) and $L=1$, an eigenvalue with $|t|>3.5$ contributes to the eta invariant at most
\begin{equation*}
\frac{1}{\pi}\int_{3.5}^{\infty}\sqrt{2\pi} e^{-x^2/2}dx=0.00046\dots
\end{equation*}
The key point is then to get reasonable bounds on the spectral density (sharper than those in Section \ref{localweyl}).
\end{enumerate}
While we will not pursue the details of this approach here for the Weeks manifold, we will do so in subsequent sections for the more challenging case of the Seifert-Weber manifold in order to prove Theorem \ref{thm2}. In particular, each of the next three sections will address one of the three aspects discussed above. Notice that in the case of $(2)$ and $(3)$, the error estimate can be improved by computing a larger portion of the (spin$^c$) length spectrum.

\begin{remark}
Let us point out that combining Propositions \ref{froyweeks} and \ref{froyeta} we obtain a formula for the Dirac eta invariants $\eta_{\mathrm{Dir}}$ of the Weeks manifold in terms of the signature eta invariant $\eta_{\mathrm{sign}}$. In particular, this allows us to provide a computation of $\eta_{\mathrm{Dir}}$ which is as precise as the one provided by Snap; for example, as mentioned in the introduction we have
\begin{equation*}
\eta_{\mathrm{Dir}}=0.989992\dots
\end{equation*}
for the unique spin structure on the Weeks manifold.
\end{remark}

\vspace{0.5cm}
\section{The linking form of the Seifert-Weber dodecahedral space}\label{linkingSW}

In this section we compute the linking form of the Seifert-Weber dodecahedral space $\SW$. This will be the key (and only) topological input for our computation of its Fr\o yshov invariants, and will give us concrete information on their fractional parts. Recall that for oriented rational homology three spheres, the linking form is the bilinear form
\begin{equation*}
Q: H_1(Y,\mathbb{Z})\times H_1(Y,\mathbb{Z})\rightarrow \mathbb{Q}/\mathbb{Z}
\end{equation*}
defined geometrically as follows: given elements $x,y$, choose $n$ for which $ny=0$ and pick a chain $T$ with $\partial T=ny$. Then $Q(x,y)$ is the intersection number of $x$ and $T$, divided by $n$. More abstractly, the corresponding map
\begin{equation*}
Q^{\hash}:H_1(Y,\mathbb{Z})\rightarrow \mathrm{Hom}(H_1(Y,\mathbb{Z}), \mathbb{Q}/\mathbb{Z})\cong H^1(Y,\mathbb{Q}/\mathbb{Z})
\end{equation*}
is the composition
\begin{equation*}
H_1(Y,\mathbb{Z})\cong H^2(Y,\mathbb{Z}) \cong H^1(Y,\mathbb{Q}/\mathbb{Z})
\end{equation*}
where the first isomorphism is Poincar\'e duality and the second is the Bockstein homomorphism for the short exact sequence of coefficients
\begin{equation*}
0\rightarrow \mathbb{Z}\rightarrow\mathbb{Q}\rightarrow\mathbb{Q}/\mathbb{Z} \rightarrow 0.
\end{equation*}
In particular, $Q^{\hash}$ is an isomorphism and $Q$ is non-degenerate.

\begin{prop}\label{spincSW}
For the explicit basis $a,b,c$ of $H_1(\SW)\cong (\mathbb{Z}/5\mathbb{Z})^3$ identified in the proof below, the linking form $Q_{\SW}$ is represented by the matrix
\begin{equation*}
\left( \begin{array}{ccc}
0&2/5&3/5\\
2/5&0&2/5\\
3/5& 2/5&0
\end{array} \right).
\end{equation*}

\end{prop}
\begin{remark}In fact, one can show that (up to scalar) every non-degenerate form has that shape with respect to some basis. But our point is that we also can identify the basis $a,b,c$ geometrically in this specific case. Our computation also suggests that the linking form of a hyperbolic three-manifold is computable by taking a Dirichlet domain with face pairing maps as input, so that the approach to the computation of the Fr\o yshov invariants of $\SW$ we will discuss in the upcoming sections works more generally.
\end{remark}
\begin{proof}
Recall that $\SW$ is obtained from a dodecahedron by identifying opposite faces by $3/5$ of a full counterclockwise rotation. The identification is described explicitly in \cite{Eve}, see Figure \ref{dodecahedron}. We have that:
\begin{itemize}
\item the $20$ vertices are identified to a single point;
\item the $30$ edges are identified in $6$ groups of $5$ elements, and we denote the corresponding generators of homology $a,b,c,d,e,f$.
\item the $12$ faces are identified in pairs as follows: $1 \leftrightarrow 12$, $2 \leftrightarrow 9$, $3 \leftrightarrow 10$, $4 \leftrightarrow 11$, $5 \leftrightarrow 7$, $6 \leftrightarrow 8$.
\end{itemize}
The six pairs of faces give us relations between the generators of the first homology. We will denote the corresponding chain by $F_i$, where $i$ is the smallest of the two indices, and orient it with the opposite orientation as the one inherited as a subset of the plane. We get:
\begin{align*}
(\partial F_1)& \hspace{0.9cm}  a+b+c+d+e&=0 \\
(\partial F_2) & \hspace{0.45cm} -a-b+c \hspace{0.7cm}+e+f&=0\\
(\partial F_3)& \hspace{0.9cm} a-b-c+d \hspace{0.7cm}+f&=0\\
(\partial F_4)& \hspace{1.55cm} b-c-d+e+f&=0\\
(\partial F_5)& \hspace{0.9cm} a \hspace{0.6cm}+c-d-e+f&=0\\
(\partial F_6)& \hspace{0.45cm} -a+b \hspace{0.6cm}+d-e+f&=0
\end{align*}
Simple elementary row operations reduce this system to
\begin{align*}
d&=a+2b+3c\\
e&=3a+2b+c\\
f&=3a+4b+3c
\end{align*}
and 
\begin{equation}\label{5is0}
5a=5b=5c=0
\end{equation}
We therefore see that $a,b,c$ generate the first homology $H_1(\SW)\cong (\mathbb{Z}/5)^3$. Furthermore, we can interpret the equations (\ref{5is0}) as the geometric identities
\begin{align*}
5a&=\partial (F_1-F_2+F_3+F_5-F_6)\\
5b&=\partial(F_1-F_2-F_3+F_4+F_6)\\
5c&=\partial(F_1+F_2-F_3-F_4+F_5).
\end{align*}
To compute the linking form it is convenient to introduce a second basis of the homology: let $A$, $B$, $C$ be the generators corresponding to the oriented segments connecting the centers of the faces $12$ to $1$, $9$ to $2$ and $10$ to $3$ respectively. Using Figure \ref{dodecahedron}, it is easy to compute that at the level of homology
\begin{align*}
A&=4a+2b+4c\\
B&=3a+b+4c\\
C&=4a+b+3c
\end{align*}
For our choice of orientation, we have that $A$ has intersection $+1$ with $F_1$, and $0$ with the other pairs; the analogous statement holds for $B$ and $F_2$ and $C$ and $F_3$ respectively. Using the geometric descriptions of chains bounding for $5a,5b$ and $5c$ above, we therefore get the following values for the linking numbers between elements in $A,B,C$ and $a,b,c$.
\begin{center}
\begin{tabular} { | c | c |c|c|}
\hline
lk & a&b&c  \\
\hline
A & $1/5$ & $1/5$ & $1/5$   \\
\hline
B & $-1/5$ & $-1/5$ & $1/5$ \\
\hline
C & $1/5$ & $-1/5$ & $-1/5$\\

\hline
\end{tabular}
\end{center}
Finally, a simple change of basis to express everything in terms of  the basis $a,b,c$ concludes the proof.
\end{proof}

A useful observation for what follows is that $Q_{\SW}(x,x)$ for $x\neq0$ attains possible values as in the following table.
\begin{center}
\begin{tabular} { | c | c |c|c|}
\hline
$Q_{\SW}(x,x)$ & number of $x\neq 0$ that attain the value  \\
\hline
0 &  24  \\
\hline
1/5 & 30 \\
\hline
2/5 & 20\\
\hline
3/5 & 20\\
\hline
4/5 & 30\\
\hline
\end{tabular}
\end{center}

\bigskip

Another useful observation can be made by looking at the isometry group of $\SW$. Recall that the latter is isomorphic to the symmetric group $S_5$, and acts faithfully on the first homology group \cite{Med}. In fact, from the description in \cite{Med} we see that the natural map
\begin{equation*}
\mathrm{Isom}(\SW)\hookrightarrow \mathrm{O}(Q_{\SW})
\end{equation*}
has image contained in $\mathrm{SO}(Q_{\SW})$. The latter has $120$ elements, and we get therefore an isomorphism $\mathrm{Isom}(\SW)\cong \mathrm{SO}(Q_{\SW})$. We have therefore the following.
\begin{cor}\label{spinisometry}
Consider two spin$^c$ structures $\spin$, $\spin'$ on $\SW$ which are not spin. Then there is an isometry $\varphi$ for which $\varphi^*\spin=\spin'$ if and only if $Q_{\SW}(c_1(\spin),c_1(\spin))=Q_{\SW}(c_1(\spin'),c_1(\spin'))$
\end{cor}
\begin{proof}
This follows by the Witt extension theorem \cite[Theorem $1.5.3$]{Sch}: because $Q_{\SW}$ is non degenerate, given two elements $x,y\neq 0$ satisfying $Q(x) = Q(y)$, there is $A\in \mathrm{SO}(Q_{\SW})$ for which $Ax=y$.
\end{proof}

\begin{figure}
\centering
\begin{tikzpicture}[thick,scale=1]

\foreach \s in {1,...,5}
{
  \node[circle, draw, fill=black!50,
                        inner sep=0pt, minimum width=4pt] (\s) at ({360/5 * (\s - 1)+90}:2) {};
}
\foreach \s in {6,...,10}
{
  \node[circle, draw, fill=black!50,
                        inner sep=0pt, minimum width=4pt]  (\s) at ({360/5 * (\s - 1)+90}:4) {};
}
\foreach \s in {11,...,15}
{
  \node[circle, draw, fill=black!50,
                        inner sep=0pt, minimum width=4pt]  (\s) at ({360/5 * (\s - 1)-90}:6){} ;
}
\foreach \s in {16,...,20}
{
  \node[circle, draw, fill=black!50,
                        inner sep=0pt, minimum width=4pt]  (\s) at ({360/5 * (\s - 1)-90}:8) {};
}

\node at (0,0) {$1$};
\foreach \s in {2,...,6}
{
  \node   at ({360/5 * (1-\s )+126}:3.3) {$\s$};
}
\foreach \s in {7,...,11}
{
  \node   at ({360/5 * (1-\s )+90}:5.5) {$\s$};
}

\draw[->-] (2) -- node[right] {$e$} (1) ;
\draw[->-] (1) -- node[left] {$a$} (5) ;
\draw[->-] (5) -- node[right] {$b$} (4) ;
\draw[->-] (4) -- node[above] {$c$} (3) ;
\draw[->-] (3) -- node[right] {$d$} (2) ;

\draw[->-] (6) -- node[right] {$b$} (1) ;
\draw[->-] (10) -- node[above] {$c$} (5) ;
\draw[->-] (9) -- node[left] {$d$} (4) ;
\draw[->-] (8) -- node[right] {$e$} (3) ;
\draw[->-] (7) -- node[above] {$a$} (2) ;

\draw[->-] (6) -- node[above] {$e$} (13) ;
\draw[->-] (13) -- node[right] {$f$} (10) ;
\draw[->-] (10) -- node[right] {$a$} (12) ;
\draw[->-] (12) -- node[above] {$f$} (9) ;
\draw[->-] (9) -- node[right] {$b$} (11) ;
\draw[->-] (11) -- node[right, above] {$f$} (8) ;
\draw[->-] (8) -- node[above] {$c$} (15) ;
\draw[->-] (15) -- node[right] {$f$} (7) ;
\draw[->-] (7) -- node[right] {$d$} (14) ;
\draw[->-] (14) -- node[above] {$f$} (6) ;

\draw[->-] (18) -- node[right] {$c$} (13) ;
\draw[->-] (17) -- node[above] {$d$} (12) ;
\draw[->-] (16) -- node[right] {$e$} (11) ;
\draw[->-] (20) -- node[above] {$a$} (15) ;
\draw[->-] (19) -- node[right] {$b$} (14) ;

\draw[->-] (17) -- node[above] {$a$} (16) ;
\draw[->-] (16) -- node[above] {$b$} (20) ;
\draw[->-] (20) -- node[right] {$c$} (19) ;
\draw[->-] (19) -- node[above] {$d$} (18) ;
\draw[->-] (18) -- node[right] {$e$} (17) ;
\end{tikzpicture}
\caption{A graphic depiction of the Seifert-Weber dodecahedral space. The twelfth face is the one formed of the five external edges, and we consider it to be on top of the others.} \label{dodecahedron}
\end{figure}
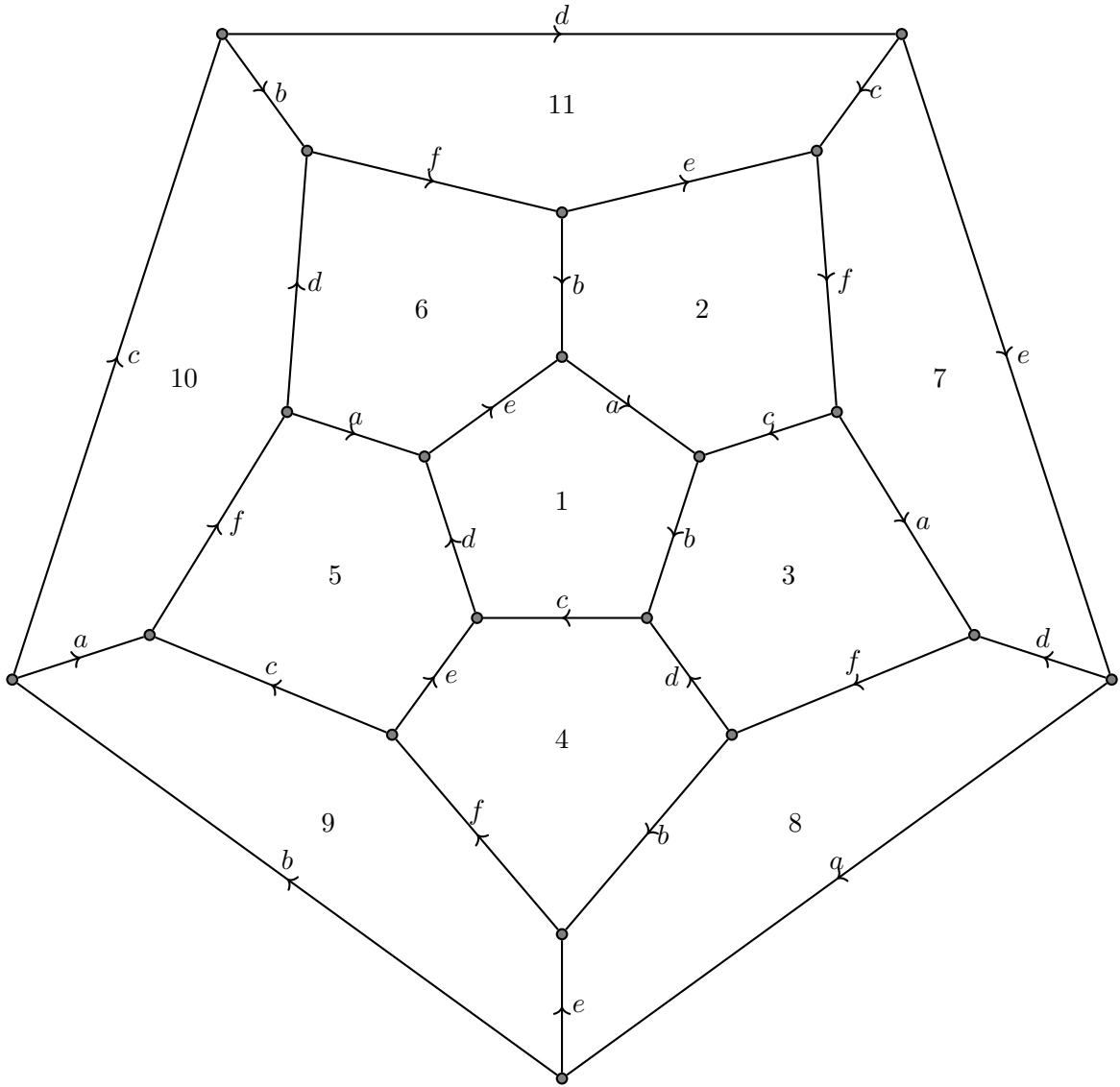

\vspace{0.5cm}
\section{Effective estimates for closed geodesics on the Seifert-Weber dodecahedral space}\label{geodesicSW}
For practical purposes, the most effective way to evaluate the eta invariants using an expression such as (\ref{etazero}) is to use a suitably dilated Gaussian (as demonstrated by the precise computations for the Weeks manifold in Section \ref{weeks}). The main complication is that the function is not compactly supported, and therefore we need to provide effective bounds on the tail sums
\begin{equation*}
\sum_{\ell(\gamma)\geq R} \ell(\gamma_0)\cdot\frac{2\sin(\mathrm{hol}(\gamma))}{|1 - e^{\mathbb{C}l(\gamma)}| \cdot |1 - e^{-\mathbb{C}l(\gamma)}|} \cdot \frac{1}{\ell} \cdot \left(- G (\ell / L) \right).
\end{equation*}
for the signature case and
\begin{equation*}  
\sum_{\ell(\gamma)\geq R} \ell(\gamma_0) \cdot \frac{2\sin( \mathrm{hol}(\tilde{\gamma}) )\cdot {\cos(\varphi_{\tilde{\gamma}})} }{|1 - e^{\mathbb{C} \ell(\gamma)} | \cdot |1 - e^{-\mathbb{C} \ell(\gamma)} |} \cdot \left(- G (\ell / L) \right),
\end{equation*}
in the Dirac case respectively. Here $L$ is the parameter at which we split the integral defining the eta function, and $R$ is the cutoff for geodesics (later on we will take them to be $1.7$ and $7.5$ respectively). For these purposes, we need a good understanding of the behavior of the lengths of geodesics, and more specifically of the quantity
\begin{equation*}
\sum_{\ell(\gamma)\in[T-1/2,T+1/2]} \ell(\gamma_0).
\end{equation*}
To put this in context, the quantity
\begin{equation}\label{geocheby}
\sum_{\ell(\gamma)\leq T} \ell(\gamma_0).
\end{equation}
naturally appears when studying the asymptotic number of prime geodesics, and in particular it plays a key role in the proof of the prime geodesic theorem
\begin{equation*}
\#\{\text{prime geodesics $\gamma$ with $\ell(\gamma)\leq T$}\}\sim e^{2T}/2T,
\end{equation*}
see for example \cite{Cha}. In fact, under the heuristic correspondence between closed geodesics on a hyperbolic manifold and prime numbers (under which $e^{\ell(\gamma)}$ is the analogue of $p$), the quantity (\ref{geocheby}) corresponds to the Chebyshev function
\begin{equation*}
\psi(x)=\sum_{p^k\leq x}\log(p).
\end{equation*}
It is well known \cite[Chapter $4$]{Apo} that the prime number theorem
\begin{equation*}
\#\{\text{prime numbers $p\leq x$}\}\sim x/\log(x)
\end{equation*}
is equivalent to the asymptotic
\begin{equation*}
\psi(x)\sim x.
\end{equation*}
Precise asymptotics for (\ref{geocheby}) can be obtained by evaluating the trace formula for functions in Theorem \ref{funformula} for a smoothed version of the function $\cosh(t)\mathbf{1}_{[-T,T]}$ (see also \cite{Bus} for the case of hyperbolic surfaces). In our case, in order to obtain reasonable effective constants, we will use instead combinations of Gaussian functions. Furthermore, while the general results of Subsection \ref{localweylfun} apply to our case, we can obtain considerably sharper bounds for the spectrum using as input our knowledge of the length spectrum of $\SW$ up to cutoff $R=8$ (rather than just the injectivity radius), as in \cite{LL2}. Furthermore, we have computed the spin$^c$ length spectrum of $\SW$ for all spin$^c$ structures up to cutoff $R=7.5$ using the method described in Section \ref{spinlength}.

\vspace{0.3cm}

\subsection{The spectral gap for the Laplacian on functions}
The first step is to understand the spectral gap for the Laplacian on functions - as we saw in Subsection \ref{localweylfun}, small eigenvalues are the main cause of large error estimates. In our case, we have.
\begin{prop}\label{SWgap}
The Laplacian on functions for $\SW$ has no small eigenvalues, and the smallest parameter satisfies $r_1>2.8$, corresponding to $\lambda_1> 8.84$.
\end{prop}
\begin{remark}
The absence of small eigenvalues in the statement of Proposition \ref{SWgap} was to be expected, as it is a consequence of the generalized Ramanujan conjecture for $\mathrm{GL}_2.$  See for example \cite{BloBru}.
\end{remark}
The proof of this proposition is based on the Booker-Strombergsson method applied to the trace formula for functions (which is indeed, in the case of surfaces, the original setup of their approach \cite{BS}). The main difference is that we now get two pictures, one for the imaginary parameters (Figure \ref{SWgraphfunctionsimaginary}) and one for the real ones (Figure \ref{SWgraphfunctionsreal})\footnote{Note the very different scales of Figures \ref{SWgraphfunctionsimaginary} and \ref{SWgraphfunctionsreal}.}. In both cases, to exclude the parameter $t,$ we minimized the quantity
\begin{equation}\label{functionsum}
\sum_{n=1}^{\infty} \widehat{H}(r_n)
\end{equation}
over the same space of functions used in \cite{LL}, subject to the constraint $\widehat{H}(t) = 1,$ using the trace formula in Theorem \ref{funformula}. Here the cutoff is $R=8$, as in \cite{LL2}. Notice that $r_0=i$ is always a spectral parameter (corresponding to $\lambda_0=0$), and is not included in the sum (\ref{functionsum}).

\begin{figure}
  \includegraphics[width=0.8\linewidth]{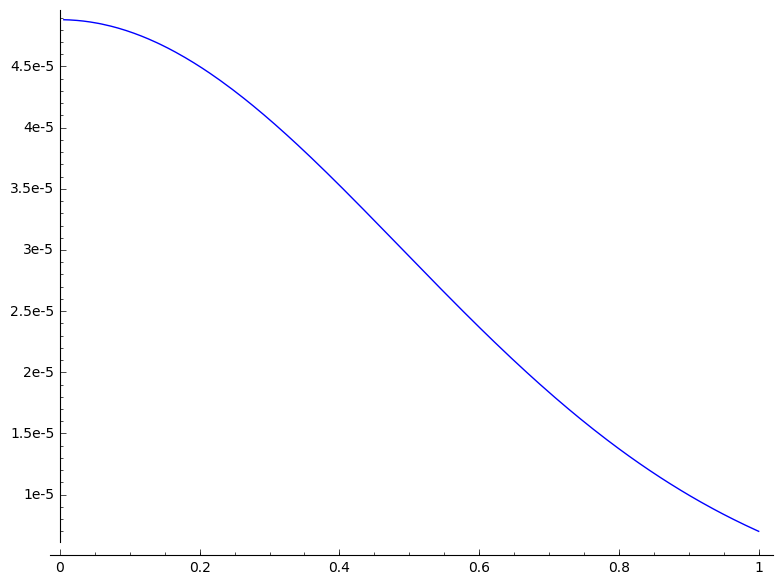}
  \caption{The picture for the imaginary parameters for the Laplacian on functions on $\SW$.}
  \label{SWgraphfunctionsimaginary}
\end{figure}

\begin{figure}
  \includegraphics[width=0.8\linewidth]{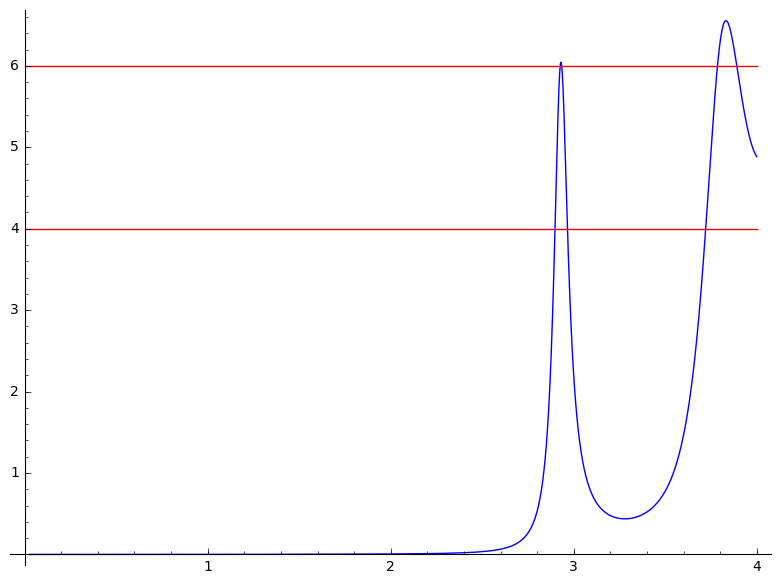}
  \caption{The picture for the real parameters for the Laplacian on functions on $\SW$.}
  \label{SWgraphfunctionsreal}
\end{figure}

\vspace{0.3cm}

\subsection{Spectral density on functions.} We now use the spectral gap and our explicit knowledge of the length spectrum to provide refined bounds on the spectral density. These two extra ingredients will allow us to greatly improve the estimates from Section \ref{localweyl}. Let us apply the trace formula to
\begin{align*}
H_{\nu} &=\left( \frac{1}{2}\mathbf{1}_{[-1,1]} \right)^{\ast 6}\cdot(e^{i\nu x}+e^{-i\nu x}) \\
&=2\cdot \left(\frac{1}{2}\mathbf{1}_{[-1,1]} \right)^{\ast 6}\cdot \cos(\nu x)
\end{align*}
which is the same kind of function we used in Section \ref{localweyl}. In this case, we have
\begin{align*}
-H_{\nu}^{''}(0) &=2*\left( \left( \left(\frac{1}{2}\mathbf{1}_{[-1,1]} \right)^{\ast 6} \right)^{''}(0)- \nu^2 \cdot \left( \frac{1}{2}\mathbf{1}_{[-1,1]} \right)^{\ast 6}(0)\right) \\
&=\frac{11}{20}\nu^2+\frac{1}{4},
\end{align*}
and
\begin{equation*}
\widehat{H_{\nu}}(t)=\left(\frac{\sin(t+\nu)}{t+\nu}\right)^6+\left(\frac{\sin(t-\nu)}{t-\nu}\right)^6.
\end{equation*}
We will use the fact that that
\begin{equation*}
\left(\frac{\sin(x)}{x}\right)^6\geq 0.777\text{ for }x\in[-0.5,0.5].
\end{equation*}
To prove upper bounds, we will look again at the trace formula. In our setup, we can compute explicitly an upper bound for the sum over geodesics
\begin{align*}
\left |\sum_{\gamma}\frac{\ell(\gamma_0)}{|1-e^{\mathbb{C}\ell(\gamma)}||1-e^{-\mathbb{C}\ell(\gamma)}|}H_{\nu}(\ell(\gamma))\right| &\leq \sum_{\gamma}\frac{\ell(\gamma_0)}{|1-e^{\mathbb{C}\ell(\gamma)}||1-e^{-\mathbb{C}\ell(\gamma)}|}H_0(\ell(\gamma)) \\
&\leq 4.827
\end{align*}
because $H_0$ is supported in $[-6,6]$ and we know the length spectrum up to cutoff $R=8$. Furthermore, recalling that
\begin{align*}
\lvert \sin(i+\nu)\lvert^2 &=\cos^2(\nu)\sinh^2(1)+\sin^2(\nu)\cosh^2(1) \\
&=\sinh^2(1)+\sin^2(\nu),
\end{align*}
we have that the contribution of the zero eigenvalue is bounded above by
\begin{equation}\label{justhere}
2 \cdot \left( \frac{\sinh(1)^2+\sin(\nu)^2}{1+\nu^2}\right)^3
\end{equation}
We are interested in upper bounds for the number of spectral parameters in $[\nu-1/2,\nu+1/2]$ for $T\geq 3$. The quantity (\ref{justhere}) is bounded above by $0.0055$ for $\nu\geq 3$.
Putting everything together, and recalling that $\SW$ has volume about $11.19$, we get the following refined local Weyl law.
\begin{prop}
For the eigenvalues of the Laplacian of functions on $\SW$, the upper bound
\begin{equation*}
\#\{\text{parameters in }[\nu-1/2,\nu+1/2]\}\leq 1.262 \nu^2+6.793.
\end{equation*}
holds for $\nu\geq 3$.
\end{prop}

\vspace{0.3cm}

\subsection{Bounds for closed geodesics}
We will apply the trace formula of Theorem \ref{funformula} to the function
\begin{equation*}
F_{c,T}=e^{-\frac{c(x-T)^2}{2}}+e^{-\frac{c(x+T)^2}{2}},
\end{equation*}
which has Fourier transform 
\begin{equation*}
\widehat{F}_{c,T}=\sqrt{\frac{2 \pi}{c}}e^{-\frac{x^2}{2c}}\cdot(e^{iTx}+e^{-iTx})=2\sqrt{\frac{{2 \pi}}{c}}e^{-\frac{x^2}{2c}}\cos(Tx).
\end{equation*}
Here $c>0$ is a parameter to be determined later. Let us discuss its value at the various terms in the trace formula, starting from the spectral side.
The contribution of the zero eigenvalue (corresponding to $r_0=i$) is
\begin{equation*}
2\sqrt{\frac{2 \pi}{c}}e^{1/{2c}}\cosh(T).
\end{equation*}
Using Proposition \ref{SWgap}, the contribution of the real parameters is bounded above (independenty of $T$) by
\begin{align} \label{regulartermsestimate}
\sum_{r_n\text{ real}}\widehat{F}_{c,T}(r_n)&=\sum_{n=3}^{\infty}\sum_{r_n\in[n-1/2,n+1/2]}\widehat{F}_{c,T}(r_n) \nonumber \\
&\leq\sum_{n=3}^{\infty}(1.262 n^2+6.793)\cdot2\sqrt{\frac{2 \pi}{c}}e^{-\frac{(n-1/2)^2}{2c}}.
\end{align}
The identity contribution is
\begin{equation} \label{identityestimate}
-\frac{\mathrm{vol}(Y)}{2\pi} F_{c,T}''(0)=\frac{\mathrm{vol}(Y)}{\pi}\cdot e^{-\frac{cT^2}{2}}\cdot(c-c^2T^2).
\end{equation}
(this is negative provided $T^2>1/c$).
\par
All these terms are explicitly computable, and we will use this information to provide upper bounds on
\begin{equation*}
\sum_{l(\gamma)\in[T-1/2,T+1/2]} \ell(\gamma_0).
\end{equation*}
as follows. Notice that
\begin{align*}
|1-e^{\mathbb{C}\ell(\gamma)}||1-e^{-\mathbb{C}\ell(\gamma)}| &\leq (e^\ell+1)(1+e^{-\ell}) \\
& =2\cosh(\ell)+2
\end{align*}
and
\begin{equation*}
F_{c,T}(\ell)\geq e^{-c/8}\text{ for } \ell \in[T-1/2,T+1/2]
\end{equation*}
We therefore have
\begin{align*}
&\sum_{\gamma}\frac{\ell(\gamma_0)}{|1-e^{\mathbb{C}\ell(\gamma)}||1-e^{-\mathbb{C}\ell(\gamma)}|}F_{c,T}(\ell(\gamma)) \\
\geq&\sum_{\ell(\gamma)\in[T-1/2,T+1/2]}\frac{\ell(\gamma_0)}{2\cosh(T+1/2)+2}e^{-c/8} \\
=& \frac{e^{-c/8}}{2\cosh(T+1/2)+2}\cdot\sum_{\ell(\gamma)\in[T-1/2,T+1/2]}{\ell(\gamma_0)}.
\end{align*}
and we can use the explicit upper bound for the first expression, obtained by combining the trace formula with the estimates \eqref{regulartermsestimate} and \eqref{identityestimate} , to get an upper bound for $\sum_{\ell(\gamma)\in[T-1/2,T+1/2]}{\ell(\gamma_0)}$, depending on a parameter $c$. We see empirically that we obtain the best estimate for $c=5$, and we have the following:
\begin{prop}\label{errorSWcutoff0}For $T\geq 7.5$, we have
\begin{equation*}
\sum_{\ell(\gamma)\in[T-1/2,T+1/2]}{\ell(\gamma_0)}\leq A(T),
\end{equation*}
where $A(T)$ is an explicit expression in $T$ readily obtained from the discussion above.
\end{prop}
We do not write explicitly the explicit expression of $A(T)$ as it is quite long and not particularly illuminating, but we remark that the leading term is of the form $C\cdot e^{2T}$.
\vspace{0.3cm}

\subsection{The geometric error for $\eta$}
Finally, we will conclude by bounding the error in the evaluation of the eta invariant using the standard Gaussian $G(x)=e^{-x^2/2}$ in the formula (\ref{etazero}), and evaluating the geometric side up to length $R=7.5$. We have $\widehat{G}(t)=\sqrt{2\pi}e^{-t^2/2}$, and
\begin{equation*}
\int_0^{\infty}\widehat{G}(t)dt=\pi,
\end{equation*}
these errors correspond to
\begin{equation*}
\frac{2}{\pi}\sum_{\ell(\gamma)\geq 7.5} \ell(\gamma_0)\cdot\frac{\sin(\mathrm{hol}(\gamma))}{|1 - e^{\mathbb{C}l(\gamma)}| \cdot |1 - e^{\mathbb{C}l(\gamma)}|} \cdot \frac{1}{\ell(\gamma)} \cdot \left(- G (\ell / L) \right)
\end{equation*}
for the signature case and
\begin{equation*}  
\frac{2}{\pi}\sum_{\ell(\gamma)\geq 7.5} \ell(\gamma_0) \cdot \frac{\sin( \mathrm{hol}(\tilde{\gamma}) )\cdot {\cos}(\varphi_{\tilde{\gamma}})}{|1 - e^{\mathbb{C} \ell(\gamma)} | \cdot |1 - e^{-\mathbb{C} \ell(\gamma)} |}\cdot \frac{1}{\ell(\gamma)} \cdot \left(- G (\ell / L) \right)
\end{equation*}
for the Dirac operator; again, $L$ is the point at which we split the integral. We can therefore bound the error in the geometric side in all cases by
\begin{equation*}
\frac{2}{\pi}\sum_{\ell(\gamma)\geq7.5}\frac{\ell(\gamma_0)}{|1-e^{\mathbb{C}\ell(\gamma)}||1-e^{-\mathbb{C}\ell(\gamma)}|}\frac{G(\ell(\gamma)/L)}{\ell(\gamma)}.
\end{equation*}
We can break the sum in two parts, one taking into account $\ell(\gamma)\in[7.5,8]$ and one for $\ell(\gamma)\geq 8$. The first part can be computed explicitly because we know the (standard) length spectrum of $\SW$ up to cutoff $R=8$. For the second part, noticing that
\begin{equation*}
|1-e^{\mathbb{C}\ell(\gamma)}||1-e^{-\mathbb{C}\ell(\gamma)}|\geq (e^\ell-1)(1-e^{-\ell})=2\mathrm{sinh}(\ell)-2,
\end{equation*}
we have
\begin{align*}
&\frac{2}{\pi}\sum_{\ell(\gamma)\geq8}\frac{\ell(\gamma_0)}{|1-e^{\mathbb{C}\ell(\gamma)}||1-e^{-\mathbb{C}\ell(\gamma)}|}\frac{G(\ell(\gamma)/L)}{\ell(\gamma)} \\
&\leq \frac{1}{\pi}\sum_{n=8}^{\infty}\frac{e^{-n^2/2L^2}}{(\sinh(n)-1)\cdot n}\cdot \sum_{\ell(\gamma)\in[n,n+1]} \ell(\gamma_0) \\
&= \frac{1}{\pi}\sum_{n=8}^{\infty}\frac{e^{-n^2/2L^2}}{(\sinh(n)-1)\cdot n}\cdot A(n+1/2)
\end{align*}
where $A$ is the quantity from Proposition \ref{errorSWcutoff0}. Putting the pieces together, we obtain an estimate on the total error for the truncated sum depending on the choice of the parameter $L$. It is clear that the larger the value of $L$ is, the better our estimate for the error in the truncated sum will be; on the other hand, we will see in the next section that larger values of $L$ lead to significantly worse errors coming from the spectral side of the formula for the eta invariant in (\ref{etazero}). We empirically found that the value $L=1.7$ provides a very good bound for the sum of these two errors in our situation. In this case, for the geometric side we have the following. 
\begin{prop}\label{errorSWcutoff}
When evaluating the $\eta$ invariant for either the signature or the Dirac operator using a standard Gaussian $G(x)=e^{-x^2/2}$, splitting point $L=1.7$ and length cutoff $R=7.5$, the error coming from truncating the sum on the geometric side is bounded above by $0.0376$.\footnote{For the evaluation of this sum, we used the fact that the general term decays extremely rapidly, as $A(n)=O( e^{2n})$; for example, the term $n=31$ is of the order of $10^{-57}$.}
\end{prop}

\vspace{0.5cm}

\section{The Fr\o yshov invariants of the Seifert-Weber dodecahedral space}\label{proofthm2}
In this section, we finally prove Theorem \ref{thm2}. Our proof is based on the fact that the Seifert-Weber is a minimal hyperbolic $L$-space \cite{LL2}. More generally, the approach of this final section can in principle be adapted to any minimal hyperbolic $L$-space, provided the linking form and a good portion of the length spectrum have been computed.
\subsection{The signature eta invariant.} We have the following.
\begin{prop}\label{etaSW}
The signature eta invariant of $\SW$ is $\eta_{\mathrm{sign}}=1.31111111\dots$
\end{prop}
Our result will be obtained by combining the Chern-Simons computations obtained via SnapPy, together with computations involving the trace formula. In principle, this result can be obtained directly via Snap \cite{CGHW}; unfortunately, we were not able to run the software on our laptops. Also, the proof provides a good example of our technique, in a simpler setup than the Dirac case (where our understanding of the spectrum is less precise).
\begin{proof}
Snappy computes the Chern-Simons invariant of $\SW$ to be
\begin{equation*}
\mathrm{cs}=-0.033333\dots
\end{equation*}
The relation
\begin{equation*}
3\eta_{\mathrm{sign}}=2\mathrm{cs}+\tau \mod 2\mathbb{Z}
\end{equation*}
holds \cite{APS2}, where $\tau$, which is the number of $2$-primary summands in $H_1(\SW,\mathbb{Z})$, is zero in our case. Thus
\begin{equation}\label{etavalue}
\eta_{\mathrm{sign}}=-0.022222\dots+\frac{2}{3}n
\end{equation}
for some $n\in\mathbb{Z}$. Therefore, approximating the eta invariant to within error less than $\frac{1}{3}$ pins down its value to high accuracy. Approximating $\eta_{\mathrm{sign}}$ well is straightforward because we have a good understanding of the small coexact 1-form eigenvalues on $\mathrm{SW}$ (see Figure \ref{SWstardBooker}, which we used in \cite{LL2} to show $\lambda_1^*>2$):
\begin{itemize}
\item The smallest eigenspace corresponds to
\begin{equation}\label{eigeninterval}
\sqrt{\lambda_1^*}\in [1.42787720680237, 1.43033743858337]
\end{equation}
and has multiplicity exactly $6$ (see \cite{LL2}). Furthermore, using the odd trace formula one readily shows that this eigenvalue is positive.
\item The next spectral parameter is larger than $2$.
\end{itemize}
We compute an approximation of the eta invariant using $G(x)=e^{-x^2/2}$ and $L=1.7$ in (\ref{etazero}), where
\begin{itemize}
\item we truncate the geometric sum at $R=7.5$;
\item in the spectral sum we approximate the first eigenvalue (with multiplicity $6$) with the midpoint $\widetilde{\sqrt{\lambda_1^*}}$ of the interval in (\ref{eigeninterval}), and consider the remaining part of the spectral sum as an error term.
\end{itemize}
As $\int_0^{\infty} \widehat{G}(t)dt=\pi$,
we have the approximation
\begin{align*}
\eta_{\mathrm{sign}}&\approx \frac{1}{\pi}\left(6\cdot \int_{1.7 \cdot \widetilde{\sqrt{\lambda_1^*}}}^{\infty}\sqrt{2\pi} e^{-t^2/2}dt-\sum_{\ell(\gamma)\leq 7.5} \ell(\gamma_0)\cdot\frac{2\sin(\mathrm{hol}(\gamma))}{|1 - e^{\mathbb{C}l(\gamma)}| \cdot |1 - e^{-\mathbb{C}l(\gamma)}} \cdot \frac{1}{\ell(\gamma)} \cdot e^{-\frac{\ell^2}{2} \cdot (1.7)^2 } \right)\\
&=1.31102382358\dots
\end{align*}
which is very close to the expression (\ref{etavalue}) for $n=2$.  All we need to do is provide an error bound on our computation.
\par
The error arising from truncating the geometric sum at $R=7.5$ is bounded above by $0.0376$ by Proposition \ref{errorSWcutoff}, while the error arising from approximating the value of the first eigenvalue is bounded above by
\begin{equation*}
\frac{6}{\pi}\cdot \int_{1.7\times 1.427877}^{1.7\times 1.4303375}\sqrt{2\pi}e^{-t^2/2}dt\leq0.00105. 
\end{equation*}
To estimate the error arising from truncating the spectral sum at the first eigenvalue, we can refine the estimates from Section \ref{localweyl} (which only involved volume and injectivity radius) because we have a direct knowledge of the length spectrum. In particular, we can use the test function
\begin{equation*}
\cos(Tx)\cdot \left(\frac{1}{2\delta}\mathbf{1}_{[-\delta,\delta]} \right)^{\ast 6}\text{ where }\delta=8/6
\end{equation*}
and compute explicitly the value of the sum over geodesics (as we know it up to cutoff $R=8$), and use it to provide upper bounds on the number of eigenvalues in specific intervals. In particular we get:
\begin{itemize}
\item there are at most $22$ spectral parameters in $[2,3]$, and each contributes at most $\frac{1}{\pi}\int_{1.7 \times 2}^{\infty}\sqrt{2\pi}e^{-t^2/2}dt<0.00068.$
\item there are at most $21$ spectral parameters in $[3,4]$, and each contributes at most $\frac{1}{\pi}\int_{1.7 \times 3}^{\infty}\sqrt{2\pi}e^{-t^2/2}dt<3.4\times 10^{-7}$
\item it is clear that the contribution of larger eigenvalues is negligible, as their number grows quadratically but the contribution decays superexponentially.
\end{itemize}
Taking all this into account, we obtain an error of at most $0.054$, and the result follows.
\end{proof}

\begin{figure}
  \includegraphics[width=0.8\linewidth]{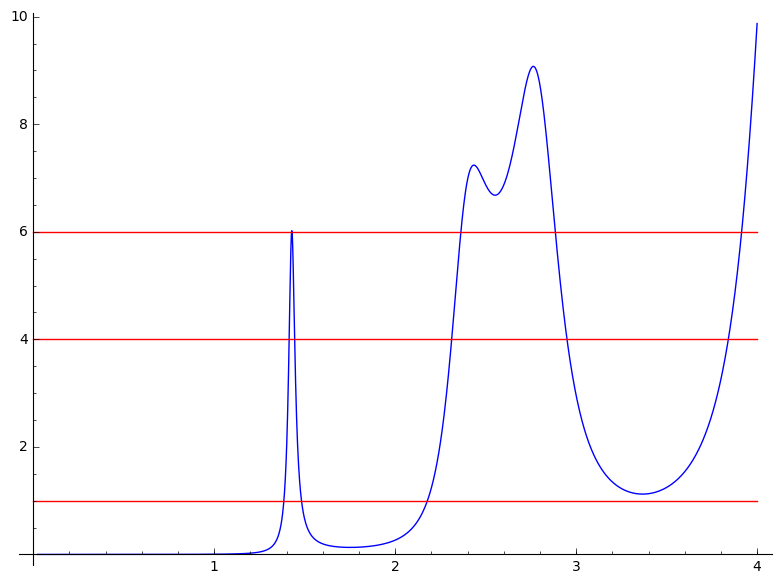}
  \caption{The function $J_{\mathrm{SW}}(t)$ for coexact $1$-forms, computed with cutoff $R=8$, cf. \cite{LL2}.}
  \label{SWstardBooker}
\end{figure}

\begin{remark}As the numerical result suggests, $\eta_{\mathrm{sign}}$ is a rational number. This follows because the invariant trace field $\mathbb{Q}(\sqrt{-1-2\sqrt{5}})$ of $\SW$ is $\mathrm{CM}$-embedded, see \cite{NY}. In fact, one can show that
\begin{equation*}
\eta_{\mathrm{sign}}=59/45
\end{equation*}
using the multiplicativity of the Chern-Simons invariants under covers, and the fact that $\SW$ admits a $60$-fold cover with an orientation reversing isometry (for which $\eta_{\mathrm{sign}}$ is therefore $0$).
\end{remark}
\vspace{0.3cm}

\subsection{The spin structure}
We now focus on the computation of $h(\SW,\spin_0)$ where $\spin_0$ is the unique spin structure on $\mathrm{SW}.$ This is the simplest spin$^c$ structure to handle because
\begin{equation*}
-2h(\SW,\spin_0)=\eta_{\mathrm{sign}}/4+\eta_{\mathrm{Dir}}=n/4
\end{equation*}
for some $n\in\mathbb{Z}$. To see this, recall the classical theorem that every spin three-manifold bounds a compact spin four-manifold $(X,\spin_X)$; in fact, we can choose $X$ to be simply connected \cite[Chapter VII]{Kir} (this will be useful later). Then, looking at the expression of the absolute grading (\ref{formulaabsgrX}) modulo integers, we have that
\begin{align*}
\eta_{\mathrm{sign}}/4+\eta_{\mathrm{Dir}} &= \frac{1}{4} \left(c_1(\spin_X)^2-2\chi(X)-3\sigma(X)-2 \right) \mod \mathbb{Z}\\
&=\frac{1}{4} \left(-2\chi(X)-3\sigma(X)-2 \right)\mod \mathbb{Z}.
\end{align*}
as $\mathrm{gr}(X^*,\spin_X,[\mathfrak{b}])$ is an integer (it is the expected dimension of a moduli space) and $c_1(\spin_X)$ is a torsion class, and the claim follows.

Therefore, we only need to approximate $\eta_{\mathrm{Dir}}$ with a good error. The spectral picture can be found in Figure \ref{SWspinc1}. Here, we computed the spin length spectrum up to cutoff $R=7.5$ using the algorithm described in Section \ref{spinlength}.

\begin{figure}
  \includegraphics[width=0.8\linewidth]{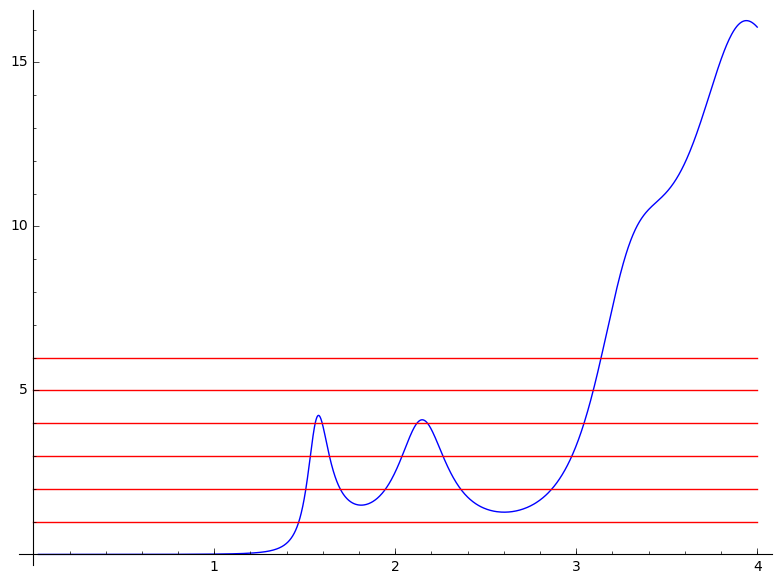}
  \caption{The function $J_{\mathrm{SW},\spin_0}(t)$, computed with cutoff $R=7.5$.}
  \label{SWspinc1}
\end{figure}

For a spin structure, the Dirac operator is quaternionic and therefore all eigenspaces have even multiplicity. We can therefore conclude that $|s_1|\geq1.45$. We compute an approximation to the eta invariant using $G(x)=e^{-x^2/2}$ and $L=1.7$ by truncating the geometric sum at $R=7.5$ and considering the whole spectral side as an error term. We obtain
\begin{align*}
\eta_{\mathrm{Dir}}&\approx -\frac{1}{\pi}\sum_{\ell(\gamma)\leq 7.5} \ell(\gamma_0)\cdot\frac{2\sin(\mathrm{hol}(\tilde{\gamma}))\cdot {\cos(\varphi_{\tilde{\gamma}})} }{|1 - e^{\ell + i \theta}| \cdot |1 - e^{-(\ell + i \theta)}|} \cdot \frac{1}{\ell(\gamma)} \cdot e^{-\frac{\ell^2}{2} \cdot (1.7)^2 }\\
&=0.641083369621\dots
\end{align*}
and therefore the approximate value
\begin{equation*}
-2h(\SW,\spin_0)\approx0.968861147398778
\end{equation*}
This shows $h(\SW,\spin_0)=-1/2$ provided we can bound the error term for $\eta_{\mathrm{Dir}}$ effectively. Again, the error from truncating the sum on the geometric side is bounded above by $0.0376$ by Proposition \ref{errorSWcutoff}. For the spectral side, we bound the number of spectral parameters as in the proof of Proposition \ref{etaSW} using the even spinor trace formula for the function 
\begin{equation}\label{boundfunction}
\cos(Tx)\cdot \left( \frac{1}{2\delta}\mathbf{1}_{[-\delta,\delta]} \right)^{\ast 6}\text{ where }\delta=7.5/6
\end{equation}
(whose geometric side we can compute as it is supported in $[-7.5,7.5]$). We get that there are at most $8$ spectral parameters in $[1.45,2]$, and each contributes at most
\begin{equation*}
\frac{1}{\pi}\int_{1.7 \cdot 1.45}^{\infty}\sqrt{2\pi}e^{-t^2/2}dt<0.0108.
\end{equation*}
Furthermore, there are at most $12$ parameters in $[2,3]$ and $31$ parameters in $[3,4]$; larger parameters are again negligible.
Putting everything together, we see that the error is bounded above by $0.1564$. Therefore, we have
\begin{equation*}
-2h(\SW,\spin_0)\in[0.812,1.126]
\end{equation*}
and the result in Theorem \ref{thm2} follows because $1$ is the only number of the form $n/4$ in that interval.

\vspace{0.3cm}
\subsection{Conclusion of the proof.}

We now complete the proof of Theorem \ref{thm2}. The arguments in this section will be quite ad hoc because the spin$^c$ structures have a smaller spectral gap than the spin case, and in particular the error bounds we will be able to provide will not be as good as in the spin case. Of course, one could in principle obtain better bounds by computing larger portions of the length spectrum, but this is a quite challenging task from the computational viewpoint (compare with the discussion in \cite{LL}).
\\
\par
A corollary of the computation of the previous section is the simply connected spin manifold $(X,\spin_X)$ bounding $\SW$ we fixed satisfies
\begin{equation*}
\frac{1}{4} \left(-2\chi(X)-3\sigma(X)-2 \right) \in \mathbb{Z}
\end{equation*}
As $H^3(X, \partial X;\mathbb{Z})=H_1(X,\mathbb{Z})=0$, the restriction map
\begin{equation*}
H^2(X,\mathbb{Z})\rightarrow H^2(\partial X,\mathbb{Z})
\end{equation*}
is surjective; therefore every spin$^c$ structure $\spin=\spin_0+x$ on $\SW$ extends to a spin$^c$ structure $\tilde{\spin}=\spin_X+\tilde{x}$ on $X$. As $c_1(\tilde{\spin})$ equals $2\tilde{x}$ up to torsion elements, we have
\begin{align*}
\eta_{\mathrm{sign}}/4+\eta_{\mathrm{Dir}} &=\frac{1}{4}(c_1^2(\tilde{\spin})-2\chi(X)-3\sigma(X)-2) \mod \mathbb{Z} \\
&=\tilde{x}^2 \mod \mathbb{Z}.
\end{align*}
As $H_1(X,\mathbb{Z})=0$, we have $\tilde{x}^2=\mathrm{lk}(x,x) \mod\mathbb{Z}$ hence we conclude,
\begin{equation*}
-2h(\SW,\spin) =\mathrm{lk}(x,x)\mod\mathbb{Z}.
\end{equation*}
By the computations of Proposition \ref{linkingSW} we know exactly the number of spin$^c$ which are not spin and for which the self-linking number equals a given value in $\left(\frac{1}{5}\mathbb{Z}\right)\slash \mathbb{Z}$. Also, we know from Corollary \ref{spinisometry} that the isometry group acts transitively on the set of non-spin spin$^c$ structures with fixed self-linking number, so we expect a priori to only have $5$ possible spectral pictures for the remaining spin$^c$ structures. This is indeed the case, see Figure \ref{SWnonspin}.
\begin{figure}
\centering
\begin{subfigure}{.45\linewidth}
\includegraphics[width=\linewidth]{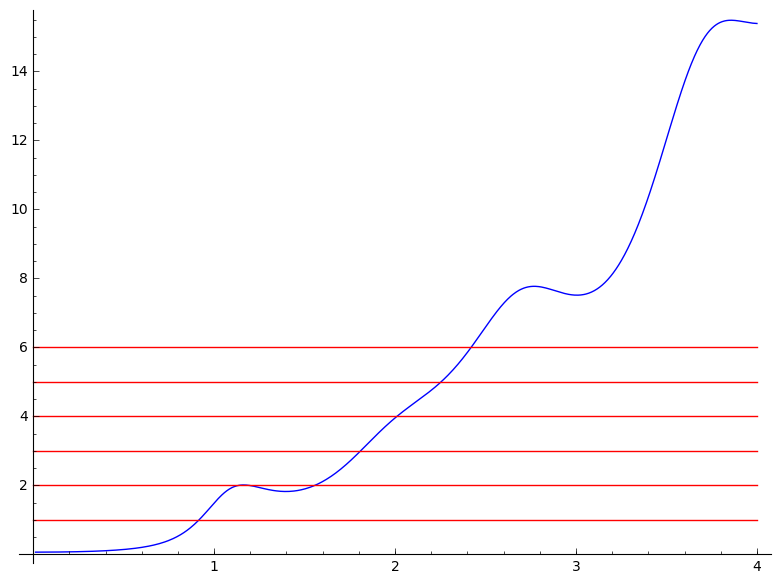}
\caption{Class $2$: $30$ elements.}
\end{subfigure}
\begin{subfigure}{.45\linewidth}
\includegraphics[width=\linewidth]{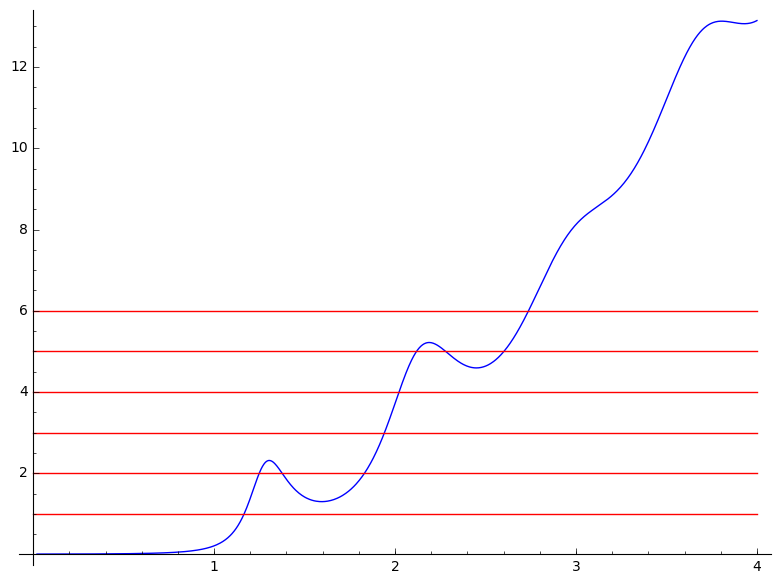}
\caption{Class $3$: $30$ elements}
\end{subfigure}

\begin{subfigure}{.45\linewidth}
\includegraphics[width=\linewidth]{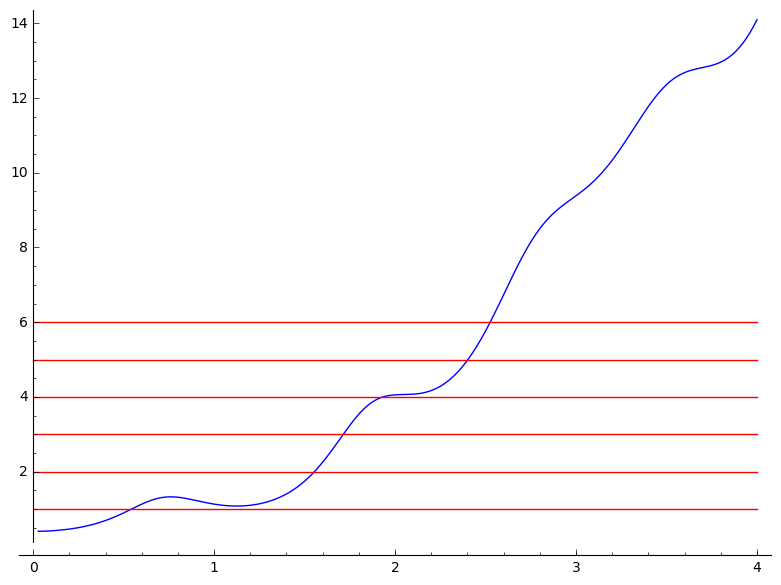}
\caption{Class $4$: $24$ elements.}
\end{subfigure}

\begin{subfigure}[b]{.45\linewidth}
\includegraphics[width=\linewidth]{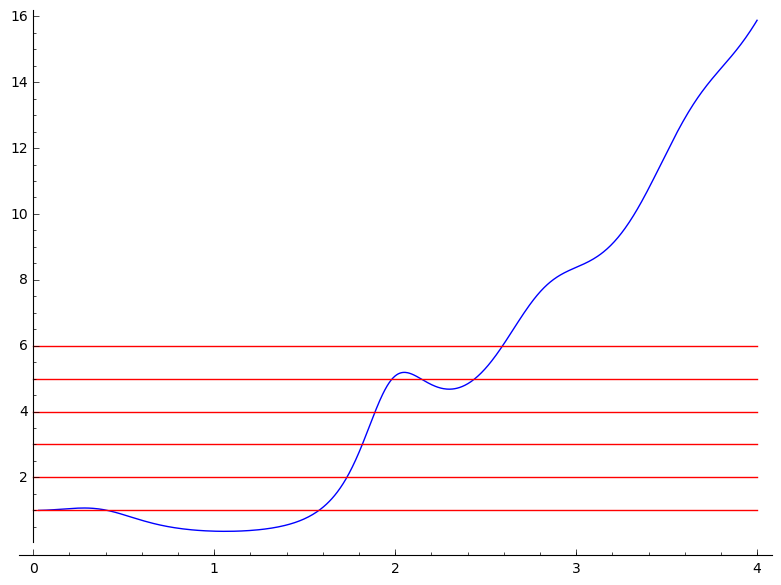}
\caption{Class $5$: $20$ elements}
\end{subfigure}
\begin{subfigure}[b]{.45\linewidth}
\includegraphics[width=\linewidth]{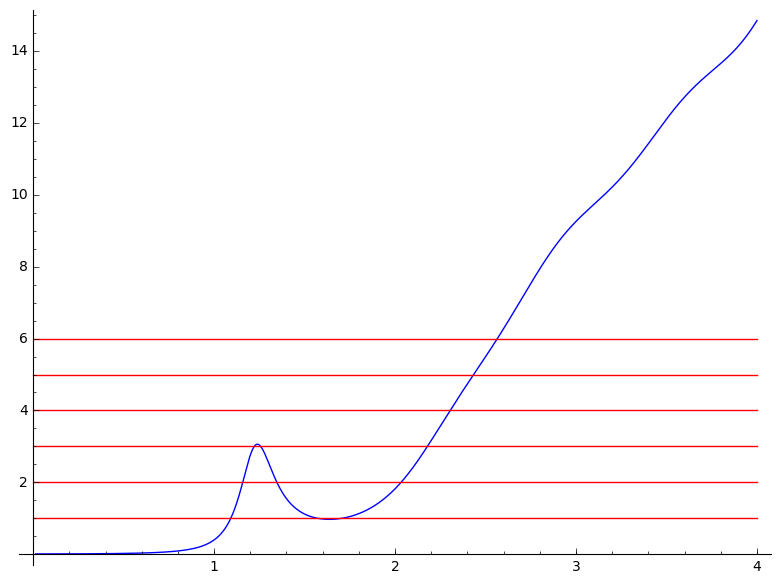}
\caption{Class $6$: $20$ elements}
\end{subfigure}
\caption{Pictures for $\SW$}
\label{SWnonspin}
\end{figure}

Comparing the number of spin$^c$ structures with a given linking form, we conclude the following.
\begin{lemma}\label{SWfrac}The following statements hold:
\begin{enumerate}
\item for the spin$^c$ structures in Class $2$ and $3$, $-2h$ has fractional part $1/5$ and $-1/5$, or vice versa;
\item for the spin$^c$ structures in Class $4$, $-2h$ has fractional part $0$;
\item for the spin$^c$ structures in Class $5$ and $6$, $-2h$ has fractional part $2/5$ and $-2/5$, or vice versa.
\end{enumerate}
As in the case of the spin structure, we compute an approximation to the Dirac eta invariant using $G(x)=e^{-x^2/2}$ and $L=1.7$ by truncating the geometric sum at $R=7.5$ and considering the whole spectral side as an error term. The results for the approximated value of $-2h(Y,\spin)=\eta_{\mathrm{sign}}/4+\eta_{\mathrm{Dir}}$ are as in the following table.
\end{lemma}
\begin{center}
\begin{tabular} { | c | c |c|c|}
\hline
Spin$^c$ class & approximate value of $\eta_{\mathrm{sign}}/4+\eta_{\mathrm{Dir}}$  \\
\hline
2 &   0.292743626654778\dots \\
\hline
3 &  $-0.133807427751222\dots$\\
\hline
4 & 0.645708045190778\dots\\
\hline
5 & 0.456382720492778\dots\\
\hline
6 & 0.542040336591778\dots\\
\hline
\end{tabular}
\end{center}

\vspace{0.3cm}

We then again need to bound errors, as in the case of the spin structure. The error from truncating the sum on the geometric side is bounded above by $0.0376$, see Proposition \ref{errorSWcutoff}. For the spectral side, using the even spinor trace formula for the function (\ref{boundfunction}) we can provide decent bounds on spectral parameters on any given interval. Using this, we can conclude; each of cases requires a slightly different argument, and we treat them separately.
\begin{proof}[Proof of Theorem \ref{thm2} for Classes $2$, $3$ and $4$.]
We begin with Class $4$, which is the most involved because it has the smallest spectral gap. Using our refined estimates on the number of spectral parameters as in the previous sections using the test functions (\ref{boundfunction}), we see that there are at most $1$ eigenvalue in $[0.545,1.05]$, $2$ in $[1.05,1.39]$, $4$ in $[1.39,1.85]$, $7$ in $[1.85,2.2]$, $15$ in $[2.2,3]$, $30$ in $[3,4]$, for a total error of $0.6279$. We therefore conclude
\begin{equation*}
-2h\in[0.017,1.274]
\end{equation*}
so that is $-2h=1$ in this case, as it is an integer by Lemma \ref{SWfrac}.
\par
For Classes $2$ and $3$ the gap is wider, so it is easier to get good bounds. For Class $2$ we can obtain for example
\begin{equation*}
-2h\in[-0.197,0.784].
\end{equation*}
Because the fractional part is either $-1/5$ or $1/5$, we obtain that $-2h=1/5$. This in turn implies that for Class $3$, $-2h$ has fractional part $-1/5$; as we obtain with the same approach
\begin{equation*}
-2h\in[-0.614,0.347],
\end{equation*}
we conclude that $-2h=-1/5$ in this case.
\end{proof}

The case of classes $5$ and $6$, which have fractional part $2/5$ and $-2/5$ or vice versa, is more delicate because for Class $5$ we can only prove a very small spectral gap. We first obtain more information on the fractional parts as follows.

\begin{lemma}
For class $5$, $-2h$ has fractional part $2/5$, while for Class $6$ it has fractional part $3/5$.
\end{lemma}
\begin{proof}
The output of our computations described in Section \ref{spinlength} provides the description of a spin$^c$ structure as a homomorphism
\begin{equation*}
\varphi: {\mathbb{Z}/5\mathbb{Z}}\oplus{\mathbb{Z}/5\mathbb{Z}}\oplus{\mathbb{Z}/5\mathbb{Z}}\rightarrow \mathbb{Q}/ \mathbb{Z},
\end{equation*}
where the natural basis is different from the one used in Section \ref{linkingSW}. It is easy to check that the linking form in this new basis is given by
\begin{equation*}\frac{1}{5}
\begin{bmatrix}
0&2&2\\
2&0&0\\
2&0&1\\
\end{bmatrix}.
\end{equation*}
Indeed, this is the unique symmetric matrix that takes the correct values for the spin$^c$ structures of class $2,3$ and $4$. The lemma is then proved by plugging in the explicit values for classes $5$ and $6$.
\end{proof}
\vspace{0.3cm}
\begin{proof}[Proof of Theorem \ref{thm2} for Class $6$.]
For Class $6$, we compute using again the test functions (\ref{boundfunction}) that there are at most $4$ parameters in $[1,1.5]$, $5$ in $[1.5,2]$, $16$ in $[2,3]$ and $30$ in $[3,4]$. We therefore have a total error of at most $0.459$,
and therefore
\begin{equation*}
-2h\in [0.083,1.002]
\end{equation*}
hence by the previous Lemma we have $-2h=3/5$.
\end{proof}

Finally, to deal with the case of Class $5$, we need some additional information about the small eigenvalues.

\begin{lemma}
For spin$^c$ class $5$, the smallest eigenvalue $s_1$ satisfies 
\begin{equation}\label{interval}
|s_1|\in[0.0408361, 0.4077692]
\end{equation}
and has multiplicity $1$. Furthermore, $|s_2|>1.55$.
\end{lemma}
\begin{proof}
Using the same method as above involving the test functions (\ref{boundfunction}) we conclude that there are at most two eigenvalues in this interval. 

To prove that there is exactly one, we apply the trace formula to the Gaussian $e^{-(x/1.7)^2/2}$ by truncating the sum over geodesics at $R=7.5$. The computations of the error bound involves the exact same quantities as in the proof of Proposition \ref{errorSWcutoff}, and we obtain the following estimate for the sum over the spectrum:
\begin{equation*}
\sum_{s_n} e^{-(1.7\cdot s_n)^2/2}\in[0.9678,1.0789].
\end{equation*}
Now, for $t\in[0.0408361, 0.4077692]$ we have $e^{-(1.7\cdot t)^2/2}\in[0.786417,0.9976]$ and the spectral picture shows that eigenvalues which are not in this interval are $\geq1.55$. The same trace formula argument used seversal times above to prove upper bounds on spectral density shows that there are at most $6$ parameters in $[1.55,2]$, $16$ in $[2,3]$, $30$ in $[3,4]$, from which we see that the contribution of the eigenvalues $\geq1.55$ has to be at most $0.2432$. Using the fact that there are at most two eigenvalues in the interval, we conclude by direct inspection that there has to be exactly one.
\end{proof}
\vspace{0.3cm}
\begin{proof}[Proof of Theorem \ref{thm2} for Class $5$.]
For our purposes the main complication arises from the fact that it is hard to determine the sign of the small eigenvalue using the trace formula. This is because in order to study the sign we need to use odd test functions, and these will be small near zero. We will therefore study the two possibilities separately.  We approximate the contribution of the small eigenvalue with the average $0.7157945$ of
\begin{align*}
\frac{1}{\pi}\int_{1.7 \cdot 0.04}^{\infty}\sqrt{2\pi}e^{-t^2/2}dt&\approx 0.945786\\
\frac{1}{\pi}\int_{1.7 \cdot 0.41}^{\infty}\sqrt{2\pi}e^{-t^2/2}dt&\approx 0.485803
\end{align*}
which introduces an error of at most $0.229992$. The total error is computed to be then at most $0.3289$ (luckily, $|s_2|$ is large). If the small eigenvalue $s_1$ were negative we would have
\begin{equation*}
-2h\in [-0.588,0.069]
\end{equation*}
which does not contain any number with fractional part $2/5$. Therefore $s_1>0$, and we have
\begin{equation*}
-2h\in[0.8432,1.5011]
\end{equation*}
from which we conclude $-2h=7/5$.
\end{proof}

\vspace{0.5cm}
\appendix

\section{Generalities on the group $G$ and the trace formula}
All the trace formulas we will need in the present paper are obtained by specializing the general trace formula for a cocompact torsion-free lattice $\widetilde{\Gamma}$ in the Lie group
$$G = \{ g \in \mathrm{GL}_2(\mathbb{C}): | \det(g) | = 1 \}.$$
The strategy of the proof follows very closely that of \cite[Appendix $B$]{LL}, where we studied the case of $\mathrm{PGL}_2(\mathbb{C})$:
\begin{enumerate}
\item express the very general trace formula in terms of irreducible representations and geometric data;
\item understand the representation theoretic incarnation of the differential operators under consideration;
\item choose suitable test functions to isolate the relevant representations.
\end{enumerate}
In \cite[Appendix $B$]{LL}, we tackled the first step for $\mathrm{PGL}_2(\mathbb{C})$: we derived a general trace formula for cocompact lattices in the group $\mathrm{PGL}_2(\mathbb{C})$ (which was called $G$ in the reference). So as not to repeat ourselves, we simply state the trace formula for $G$ and highlight the differences between it and the very similar trace formula from \cite{LL}.
\begin{remark}
As it will be clear from the discussion, this approach involving $G$ is only strictly necessary when dealing with spinors; the trace formulas for functions and forms, both even and odd, can be derived from the general trace formula for $\mathrm{PGL}_2(\mathbb{C})$.
\end{remark}

\subsection{Notation for subgroups of $G$.}\label{subgroups}
We will work with the maximal torus $T\subset G$ of diagonal matrices; we will use the parametrizations
\begin{equation}\label{parametrization1}
T = \left\{ \left( \begin{array}{cc} e^{u/2} e^{i \alpha} & 0 \\ 0 & e^{-u/2} e^{i\beta} \end{array} \right): u \in \mathbb{R}, \alpha, \beta \in \mathbb{R} / 2\pi \mathbb{Z} \right\}
\end{equation}
and 
\begin{equation}\label{parametrization2}
T =\left\{\left( \begin{array}{cc} e^{i\phi} e^{v + i\theta} & 0 \\ 0 & e^{i\phi} e^{-v-i\theta} \end{array} \right)v \in \mathbb{R}, \phi, \theta \in \mathbb{R} / 2\pi \mathbb{Z}\right\}.
\end{equation}
The center $Z$ of $G$ is the diagonal copy of $U_1$. We will equip $T$ with the Haar measure
\begin{equation}\label{Haar}
dt=du \frac{d\alpha}{2\pi}\frac{d\beta}{2\pi}.
\end{equation}
The Weyl group $W$ has two elements, and is generated by $(u,\alpha,\beta)\mapsto(-u,\beta,\alpha)$.
\begin{remark}
Using this notation, the complex length of the image in $\mathrm{PGL}_2(\mathbb{C})$ of an element of $T$ is $u+i2\theta$.
\end{remark}
We will denote by $B$ the subgroup of upper triangular matrices; the modular function is then
\begin{align*}\delta:B&\rightarrow \mathbb{R}^{>0}\\
\left( \begin{array}{cc} t_1 & \ast \\ 0 & t_2  \end{array}  \right) &\mapsto |t_1/t_2|^2
\end{align*}
Finally, $K=U_2$ is the maximal compact subgroup; we can identify
\begin{equation*}
G/K=\mathrm{PGL}_2(\mathbb{C})/PU_2=\mathbb{H}^3
\end{equation*}
using the upper half-space model $\mathbb{H}^3\equiv\mathbb{C} \times \mathbb{R}_{> 0}$. The tangent space at $(0,1)$ is then identified with $\mathfrak{p}_0=i\mathfrak{su}(2)$. We will also denote $\mathfrak{p}=\mathfrak{p}_0\otimes \mathbb{C}$ for its complexification. 
\vspace{0.3cm}

\subsection{Unitary representations of $G$} \label{unitaryrepsG}
Every irreducible unitary representation of $G$ is isomorphic to one of the following:  
\begin{itemize}
\item[(a)]
the one dimensional representation with character $\mathrm{det}^k$ for some $k\in \mathbb{Z}$.  

\item[(b)]
the representation $\pi_{s_1,s_2,n_1,n_2}$ with $s_1 = ir_1, s_2 = ir_2$ purely imaginary and $n_1, n_2 \in \mathbb{Z},$ defined as follows. Consider the character
\begin{align*}
\chi_{s_1,s_2, n_1, n_2}: B&\rightarrow \mathbb{C}^*\\
\left( \begin{array}{cc} t_1 & \ast \\ 0 & t_2  \end{array}  \right)& \mapsto |t_1|^{s_1} \cdot |t_2|^{s_2} \cdot \left(\frac{t_1}{|t_1|} \right)^{n_1} \cdot \left( \frac{t_2}{|t_2|} \right)^{n_2}.
\end{align*}
We then define $\pi_{s_1,s_2,n_1,n_2}$ to be the representation induced on $G$ by $\delta^{1/2}\chi_{s_1,s_2, n_1, n_2}$. Because $|t_1| \cdot |t_2| = 1,$ the representation $\pi_{s_1,s_2,n_1,n_2}$ only depends on $s_1 - s_2, n_1, n_2.$
\item[(c)]
complementary series representations, which are abstractly isomorphic to $\pi_{s_1,s_2,0,0}$ where the difference $s_1 - s_2$ is real and lies in the interval $[-1,1].$
\end{itemize}
Again, there are some coincidences among these representations.
\begin{lemma}
We have $\pi_{s_1,s_2,n_1,n_2} \cong \pi_{t_1,t_2,m_1,m_2}$ iff 
$$(t_1,t_2,m_1,m_2) = (s_1,s_2,n_1,n_2) \text{ or } (s_2,s_1,n_2,n_1) \in \frac{i \mathbb{R} \times i \mathbb{R}}{\langle i(1,1) \rangle} \times \mathbb{Z} \times \mathbb{Z}.$$
\end{lemma}
\begin{proof} The proof is very similar to the case of $\mathrm{PGL}_2(\mathbb{C})$ covered in \cite[Appendix $B$]{LL}. For every smooth compactly supported test function $f$ on $G,$ 
$$\mathrm{trace}( \pi_{s_1,s_2,n_1,n_2} ) = \widehat{Sf}( \chi_{s_1,s_2,n_1,n_2}^{-1} ),$$
where we view $Sf$ as a $W$-invariant function on the diagonal torus $T.$  By the main theorem of Bouaziz \cite{Bouaziz}, the Satake transform $f \mapsto Sf$ maps onto the $W$-invariant compactly supported smooth function on $T.$  Therefore, two representations $\pi_{s_1,s_2,n_1,n_2}$ and $\pi_{t_1,t_2,m_1,m_2}$ are isomorphic iff  
$$\widehat{F}(\chi_{s_1,s_2,n_1,n_2}) = \widehat{F}(\chi_{t_1,t_2,m_1,m_2})$$
for all smooth, compactly supported functions $F$ on $T$ which is $W$-invariant, i.e. 
\begin{equation*}
F(u,\alpha,\beta) = F(-u,\beta,\alpha).
\end{equation*}
The function $F$ is $W$-invariant iff $\widehat{F}(s,n,m) = \widehat{F}(-s,m,n)$ (the forward implication is immediate from the definition, and the reverse implication follows from Fourier inversion).  Since
\begin{align*}
\widehat{F}(\chi_{s_1,s_2,n_1,n_2}^{-1}) &=\int F(t)\chi_{s_1,s_2,n_1,n_2}(t)dt\\
&= \int F(u,\alpha,\beta)  \cdot e^{u/2 \cdot (s_1 - s_2) } \cdot e^{i n_1 \alpha} \cdot e^{i n_2 \beta} du \frac{d\alpha}{2\pi} \frac{d\beta}{2\pi} \\
&= \widehat{F}(-(s_1-s_2), -n_1,-n_2 ),
\end{align*}
the result follows.
\end{proof}

\vspace{0.3cm}

\subsection{The trace formula for $G$} We have the following general result.
\begin{prop} \label{preliminaryreptheorytraceformula}
Let $F$ be any smooth, compactly supported, $W$-invariant function on the diagonal torus $T$. Let $\widetilde{\Gamma} \subset G$ be a cocompact torsion-free lattice. The identity
\begin{align*}
&{} \sum_{s_1 - s_2,n, m} m_{\widetilde{\Gamma}}(\pi_{s_1,s_2,n,m}) \cdot \widehat{F}(\chi_{s_1,s_2,n,m}^{-1}) + \sum_k \frac{1}{|W|} \int_T |D(t^{-1})|^{1/2} \cdot \det(t)^k \cdot F(t)  dt \\
&=  -{\frac{1}{8\pi}} \cdot \vol(Y) \cdot \sum_{{z \in \widetilde{\Gamma} \cap Z}} \left(  \frac{d^2}{dv^2} + \frac{d^2}{d\theta^2} \right) F |_{t = z} + \sum_{[\gamma] \neq 1} \ell(\gamma_0) \cdot |D(t_\gamma^{-1})|^{-1/2} \cdot F(t_\gamma).
\end{align*}
holds. Here the second derivatives in $v$ and $\theta$ are taken with respect to the parametrization (\ref{parametrization2}) of $T$, and $|W|=2$.
\end{prop}
This is in fact a generalization of the trace formula from \cite{LL}, which can be recovered by using test functions on the diagonal torus $T$ invariant by the center $Z$ of $G.$   
The proof of this result is essentially identical to the one of the formula in \cite{LL}, so we will focus on pointing out similarities and differences between Proposition \ref{preliminaryreptheorytraceformula} and the trace formula from our first paper.
\begin{itemize}

\item
The ``identity contribution" to the trace formula has a contribution for each element of $\widetilde{\Gamma} \cap Z.$  The group $Z$ is compact so this sum is finite. In our case of interest, $\widetilde{\Gamma}$ is the lift of a lattice $\Gamma \subset \mathrm{PGL}_2(\C),$ and $\widetilde{\Gamma} \cap Z = \{ 1 \}$.

\item
Recall that the function $F(u,\alpha,\beta)$ is $W$-invariant iff
$$F(u,\alpha,\beta) = F(-u,\beta,\alpha).$$
In particular, we will use $F(u,\alpha,\beta) = K_{\pm}(\alpha,\beta) \cdot H_{\pm}(u),$ where $K_{\pm}$ are symmetric and antisymmetric respectively and $H_{\pm}$ are even and odd respectively, define $W$-invariant functions on $T.$

\item
The universal constant $-\frac{1}{8\pi}$ can be determined in several ways; for example, by specializing this formula in the case of the even trace formula for coexact $1$-forms and comparing it to the one from \cite{LL}.

\end{itemize}

\vspace{0.3cm}
\section{The odd signature and Dirac operators via representation theory} 
We need to specialize the general trace formula in Proposition \ref{preliminaryreptheorytraceformula} to our cases of interest, namely to isolate the eigenvalue spectrum of the odd signature and Dirac operators. In \cite{LL}, we showed how to interpret the Hodge Laplacian and its spectrum in terms of the representation theory of $\mathrm{PGL}_2$; in particular, we identified coexact $1$-eigenforms in terms of isotypic vectors in the irreducible representations, and the corresponding eigenvalue using the Casimir eigenvalue. We will do the same here, where we consider the maximal compact subgroup $K=U_2$. In particular we show the following two results.\footnote{Recall that our convention is that \begin{equation*}
\sigma_1 = \left( \begin{array}{cc} 0 & 1 \\ 1 & 0  \end{array} \right) \quad\sigma_2 = \left( \begin{array}{cc} 0 & -i \\ i & 0  \end{array} \right) \quad
\sigma_3 = \left( \begin{array}{cc} 1 & 0 \\ 0 & -1  \end{array} \right)
\end{equation*} forms a positively oriented basis of $\mathfrak{p}_0=i\mathfrak{su}(2)$.} 
\begin{prop} Eigenforms of $\ast d$ correspond to $\mathfrak{p}$-isotypic vectors where $\mathfrak{p}=\mathfrak{p}_0\otimes \mathbb{C}$. An irreducible unitary representation of $G$ contains a $\mathfrak{p}$-isotypic vector if and only if it is isomorphic to $\pi_{s_1,s_2,1,-1}$ with $s_1,s_2$ purely imaginary. Furthermore, each copy of $\pi_{s_1,s_2,1,-1}\subset L^2(\widetilde{\Gamma} \backslash G)$ contains exactly a one dimensional eigenspace of $\ast d$, spanned by an eigenform with eigenvalue $i(s_1-s_2)/2$.
\end{prop}
Recall that a $\mathfrak{p}$-isotypic vector in a $G$-representation $\pi$ is a non-zero element of $\mathrm{Hom}_K(\mathfrak{p},\pi)$. Notice that as the Hodge Laplacian on coexact $1$-forms is the square of $\ast d$, this result implies (after changing the parametrization) those in \cite{LL}; in that case we computed the corresponding eigenvalue via Kuga's lemma. Notice that while this result can be proved directly using the group $\mathrm{PGL}_2$, the following one relies essentially on the larger group $G$.
\begin{prop}
Eigenspinors correspond to $S^{\vee}$-isotypic vectors where $S^{\vee}$ is the dual of the standard representation of $K=U_2$ on $\mathbb{C}^2$. An irreducible unitary representation of $G$ contains an $S^\vee$-isotypic vector if and only if it is isomorphic to $\pi_{s_1,s_2,-1,0}$ with $s_1,s_2$ purely imaginary. Furthermore, each copy of $\pi_{s_1,s_2,-1,0}\subset L^2(\widetilde{\Gamma} \backslash G)$ contains exactly a one dimensional eigenspace of the Dirac operator, spanned by an eigenspinor with eigenvalue $i(s_1-s_2)/2$.
\end{prop}

The present Appendix is dedicated to the proof of these results. The new main ingredient is to describe the Dirac and $\ast d$ operators in terms of representation theory, which we do in Subsections \ref{repprin}, \ref{automorphicdirac} and \ref{automorphicstard}. We then identify the relevant representations in Subsection \ref{frobrec}, and the corresponding eigenvalues in Subsection \ref{eigencalc}.

\vspace{0.3cm}
\subsection{Generalities on associated bundles and connections for the principal $K$-bundle $G \rightarrow G / K$}\label{repprin}

\subsubsection{The invariant connection on $G \rightarrow G/K$} \label{invariantconnection}

We follow the notation of Subsection \ref{subgroups}. Note that $G \xrightarrow{\pi} G / K$ (with $K$ acting on the right by multiplication) is a principal $K$-bundle.  The assignment 
$$H_g := (L_g)_{\ast} \mathfrak{p}_0 \subset T_g(G)$$
defines a subbundle $H \subset TG.$  Observe that 

\begin{enumerate}
\item \label{first}
The projection $H_g \rightarrow T_g(G) \xrightarrow{\pi_{\ast}} T_{gK}(G/K)$ is an isomorphism.

\item \label{second}
The subbundle $H$ is right $K$-invariant:
\begin{align*}
(R_k)_{\ast} H_g &= (R_k)_{\ast} (L_g)_{\ast} \mathfrak{p}_0 \\
&= (L_g)_{\ast} (R_k)_{\ast} \mathfrak{p}_0 \hspace{2.0cm} \text{because left and right multiplications commute} \\
&= \left[ (L_g)_{\ast} (L_k)_{\ast} \right] \left[ (L_k)_{\ast}^{-1} (R_k)_{\ast} \right] \mathfrak{p}_0 \\
&= (L_g L_k)_{\ast} (\text{conjugation by } k^{-1})_{\ast} \mathfrak{p}_0 \\
&= (L_{gk})_{\ast} \mathrm{Ad}(k^{-1}) \mathfrak{p}_0 \\
&= (L_{gk})_{\ast} \mathfrak{p}_0  \hspace{2.8cm} \text{because } \mathfrak{p}_0 \text{ is invariant under } \mathrm{Ad}(K)
\end{align*}
\item \label{third}
The subbundle $H$ is left $G$-invariant: $(L_h)_{\ast} H_g = (L_h)_{\ast} (L_g)_{\ast} \mathfrak{p}_0 = H_{hg}.$
\end{enumerate}
Items \eqref{first} and \eqref{second} imply that $H$ defines a connection on $G \rightarrow G/K.$  Item \eqref{third} implies that $H$ is left $G$-invariant.

\subsubsection{Associated bundles and covariant derivatives} \label{covariantderivative}
Let $\rho: K \rightarrow \GL(V)$ be a representation of $K.$  The associated vector bundle $\mathcal{F}_V := G \times_{\rho} V \rightarrow G/K$ has sections 
$$\Gamma(\mathcal{F}_V) = \{ s: G \rightarrow V: s(gk) = \rho(k)^{-1} s(g) \}.$$
The covariant derivative $\nabla^V$ on $\mathcal{F}_V$ associated to the connection $H \subset TG$ is given by 
\begin{align*}
\nabla^V_X: \Gamma(\mathcal{F}_V) &\rightarrow \Gamma(\mathcal{F}_V) \\
s &\mapsto ds(\widetilde{X}),
\end{align*}
where $\widetilde{X}$ denotes the horizontal lift of $X$ to $TG$ via the invariant connection $H$ from \S \ref{invariantconnection}.  Via the trivialization 
\begin{align*}
\mathfrak{p}_0 \times G &\rightarrow H \\
(X,g) &\mapsto (L_g)_{\ast} X
\end{align*}
from \S \ref{invariantconnection}, we may regard $\nabla^V s \in \Gamma \left( \mathcal{F}_V \otimes \Omega^1(G / K) \right),$ which is the bundle associated to the representation $V \otimes \mathfrak{p}_0^\vee \simeq \mathrm{Hom}(\mathfrak{p}_0,V)$ of $K,$ as the function
\begin{align} \label{covariantderivativeformula}
\nabla^V s: G &\rightarrow \mathrm{Hom}(\mathfrak{p}_0,V) \\
g &\mapsto \left(X \mapsto ds( (L_g)_{\ast} X) = X(s \circ L_g) \right). \nonumber
\end{align}
Let $r$ denote the representation 
\begin{align*}
r: K &\rightarrow \mathrm{End}\left( \mathrm{Hom}(\mathfrak{p}_0,V) \right) \\
k &\mapsto \left(T \mapsto \rho(k) \circ T \circ \mathrm{Ad}(k)^{-1} \right)
\end{align*}
Note that $\nabla^V s,$ as defined by formula \eqref{covariantderivativeformula}, satisfies the $K$-equivariance property
 \begin{equation*}
 \left( \nabla^V s \right)(gk) = r(k)^{-1} \cdot \left(\nabla^V s\right)(g).
 \end{equation*}

\subsubsection{Dual picture for sections of $\mathcal{F}_V$ and the covariant derivative $\nabla^V$} \label{dualpicture} 
For our purposes it will be convenient to rephrase everything in terms of a dual picture, which is more manifestly representation theoretic. There is a canonical isomorphism

\begin{align} \label{sectionsdualpicture}
\Phi^V: C^\infty(G,V) &\xrightarrow{\sim} \mathrm{Hom}(V^\vee, C^\infty(G)) \nonumber \\
s &\mapsto \left(\; \ell \mapsto (g \mapsto \ell(s(g))) \; \right)
\end{align}

The inverse mapping $(\Phi^V)^{-1}$ is most easily expressed using a basis $\{ v_i \}$ for $V$ and its corresponding dual basis $\{v_i^\vee\}$ for $V^\vee$:
\begin{align} \label{sectionsdualpictureinverse}
(\Phi^V)^{-1}:  \mathrm{Hom}(V^\vee, C^\infty(G))  &\xrightarrow{\sim} C^\infty(G,V) \nonumber \\
T &\mapsto \left(g \mapsto \sum_i T(v_i^\vee)(g) \cdot v_i \right)
\end{align}
The map $\Phi := \Phi^V$ satisfies the following equivariance properties:  
\begin{enumerate}
\item
$\Phi( L_g s) = L_g \Phi(s)$ and $\Phi( R_g s) = R_g \Phi(s).$  In $C^{\infty}(G,V)$,
\begin{equation*}
L_g s(x) = s(gx), R_g s(x) = s(xg)
\end{equation*}
and on $\mathrm{Hom}(V^\vee, C^\infty(G)),$ the corresponding actions are
$$V^\vee \longrightarrow C^{\infty}(G) \xrightarrow{\text{left or right translation by } g} C^\infty(G).$$

\medskip

\item
$\Phi$ carries the subspace $\Gamma(\mathcal{F}_V) = \{ s: G \rightarrow S: s(gk) = \rho(k)^{-1} s(g) \}$ to the subspace $\mathrm{Hom}_K(V^\vee,C^\infty(G)),$ where $K$ acts by right translation on $C^\infty(G)$ and by $\rho^\vee$ on $V^\vee.$

\medskip

\item
Upon combining \eqref{covariantderivativeformula} with the explicit formula for $\Phi^{-1}$ from \eqref{sectionsdualpictureinverse}, we see that $\Phi$ intertwines the covariant derivative $\nabla^V$ with operator
\begin{align*}
\Theta^V: \mathrm{Hom}(V^\vee, C^{\infty}(G)) &\rightarrow \mathrm{Hom}((V \otimes \mathfrak{p}_0^\vee)^\vee, C^{\infty}(G)) \\
T &\mapsto \left(\ell \otimes X \mapsto \left(  g \mapsto \sum_i \frac{d}{dt}|_{t = 0} \left[  T(v_i^\vee)(g e^{tX}) \right] \cdot \ell(v_i) \right)  \right), 
\end{align*}
where $\ell \otimes X \in V^\vee \otimes \mathfrak{p}_0 \cong (V \otimes \mathfrak{p}_0^\vee)^\vee,$ i.e. there is a commutative diagram
$$\begin{CD}
C^{\infty}(G,V) @>{\nabla^V}>>C^{\infty}(G,V \otimes \mathfrak{p}_0^\vee) \\
@V{\Phi}VV @VV{\Phi}V \\
\mathrm{Hom}(V^\vee, C^{\infty}(G)) @>{\Theta^V}>> \mathrm{Hom}(V^\vee\otimes \mathfrak{p}_0, C^{\infty}(G)).
\end{CD}$$  

The above restricts to a commutative diagram
$$\begin{CD}
\Gamma(\mathcal{F}_V) @>{\nabla^V}>> \Gamma(\mathcal{F}_{V\otimes \mathfrak{p}_0^\vee})  \\
@V{\Phi}VV @VV{\Phi}V \\
\mathrm{Hom}_K(V^\vee, C^{\infty}(G)) @>{\Theta^V}>> \mathrm{Hom}_K(V^\vee\otimes \mathfrak{p}_0, C^{\infty}(G)).
\end{CD}$$  

\medskip

\item
$K$-equivariant maps $\alpha: V \rightarrow W$ induce maps $\Gamma(\mathcal{F}_V) \rightarrow \Gamma(\mathcal{F}_W)$ (by post-composition with $\alpha$) and on $\mathrm{Hom}_K(V^\vee, C^\infty(G)) \rightarrow \mathrm{Hom}_K(W^\vee, C^\infty(G))$ (by pre-composition with $\alpha^\vee$).  These induced maps commute with the isomorphism $\Phi.$
\end{enumerate}

\subsubsection{Going automorphic} \label{goingautomorphic}
Everything in \S \ref{invariantconnection}, \S \ref{covariantderivative} and \S \ref{dualpicture} is left invariant.  So the connections, vector bundles, covariant derivatives and bundle maps defined thus far all descend to $\widetilde{\Gamma} \backslash G / K.$  We denote descended objects with an $\overline{\bullet}$, e.g. for the descended vector bundle $\overline{\mathcal{F}}_V$ over $\widetilde{\Gamma} \backslash G/K$ 
$$\Gamma(\overline{\mathcal{F}}_V) = \{  s: \widetilde{\Gamma} \backslash G \rightarrow S: s(\widetilde{\Gamma}gk) = \rho(k)^{-1} s(\widetilde{\Gamma}g) \}$$
which is identified with, via the isomorphism $\Phi^V,$ with $\mathrm{Hom}_K(V^\vee, C^\infty(\widetilde{\Gamma} \backslash G)).$

\medskip

The descended covariant derivative $\overline{\nabla}^V$ and its incanation $\overline{\Theta}^V$ in the dual picture admit the ``same formulas mod $\widetilde{\Gamma}$".  In particular, 

\begin{align*}
\overline{\Theta}^V: \mathrm{Hom}_K(V^\vee, C^{\infty}(\widetilde{\Gamma}\backslash G)) &\rightarrow \mathrm{Hom}_K((V \otimes \mathfrak{p}_0^\vee)^\vee, C^{\infty}( \widetilde{\Gamma} \backslash G)) \\
T &\mapsto \left(\ell \otimes X \mapsto \left(  g \mapsto \sum_i \frac{d}{dt}|_{t = 0} \left[  T(v_i^\vee)(g e^{tX}) \right] \cdot \ell(v_i) \right)  \right), 
\end{align*}
and there is a commutative diagram
$$\begin{CD}
\Gamma(\overline{\mathcal{F}}_V) @>{\overline{\nabla}^V}>> \Gamma(\overline{\mathcal{F}}_{V\otimes \mathfrak{p}_0^\vee})  \\
@V{\overline{\Phi}}VV @VV{\overline{\Phi}}V \\
\mathrm{Hom}_K(V^\vee, C^{\infty}(\widetilde{\Gamma} \backslash G)) @>{\overline{\Theta}^V}>> \mathrm{Hom}_K(V^\vee\otimes\mathfrak{p}_0, C^{\infty}(\widetilde{\Gamma} \backslash G)).
\end{CD}$$  
Given a $K$-equivariant map $\alpha: V \rightarrow W,$ we denote the associated bundle map (obtained by post-composition by $\alpha$ on $\Gamma(\overline{\mathcal{F}}_V)$ or by precomposition by $\alpha^\vee$ in the dual picture) by $\overline{\alpha}.$  We also have commutative diagrams
$$\begin{CD}
\Gamma(\overline{\mathcal{F}}_V) @>{\overline{\alpha}}>> \Gamma(\overline{\mathcal{F}}_{V\otimes \mathfrak{p}^\vee_0})  \\
@V{\overline{\Phi}}VV @VV{\overline{\Phi}}V \\
\mathrm{Hom}_K(V^\vee, C^{\infty}(\widetilde{\Gamma} \backslash G)) @>{\overline{\alpha}}>> \mathrm{Hom}_K(V^\vee\otimes \mathfrak{p}_0, C^{\infty}(\widetilde{\Gamma} \backslash G)).
\end{CD}$$

Because $\overline{\nabla}^V$ is built from the derivative of the right translation action of $G$ on $C^{\infty}(\widetilde{\Gamma} \backslash G),$ the operators $\overline{\Theta}^V$ are amenable to representation-by-representation analysis.  Likewise, the bundle maps $\overline{\alpha}$ are amenable to representation-by-representation analysis.  For irreducible (unitary) representations $\pi$ of $G,$ denote by $\pi^{\infty}\subset\pi$ the dense subspace of smooth elements. We define the following $\pi$-isotypic versions of the operators $\overline{\Theta}^V$ and $\overline{\alpha}$: 
\begin{align} \label{connectiondualpicture}
\overline{\Theta}^V_{\pi}: \mathrm{Hom}_K(V^\vee, \pi^\infty ) &\rightarrow \mathrm{Hom}_K((V \otimes \mathfrak{p}_0^\vee)^\vee, \pi^\infty ) \nonumber \\
T &\mapsto \left(\ell \otimes X \mapsto \sum_i \left(X \cdot T(v_i^\vee) \right)  \cdot \ell(v_i) \right), 
\end{align}
and 
\begin{align*}
\overline{\alpha}_{\pi}: \mathrm{Hom}_K(V^\vee, \pi^\infty ) &\rightarrow \mathrm{Hom}_k(W^\vee, \pi^\infty ) \\
T &\mapsto T \circ \alpha^\vee. 
\end{align*}
Given a subrepresentation $\pi \subset L^2(\widetilde{\Gamma} \backslash G),$ the above $\pi$-isotypic operators fit into commutative diagrams 
\begin{equation} \label{connectiondualpicturediagram}
\begin{CD}
\mathrm{Hom}_K(V^\vee, \pi^{\infty} ) @>{\overline{\Theta}_{\pi}^V}>> \mathrm{Hom}_K((V \otimes \mathfrak{p}_0^\vee)^\vee, \pi^\infty)  \\
@V{\iota}VV @VV{\iota}V \\
\mathrm{Hom}_K(V^\vee, C^{\infty}(\widetilde{\Gamma} \backslash G)) @>{\overline{\Theta}^V}>> \mathrm{Hom}_K((V \otimes \mathfrak{p}_0^\vee)^\vee, C^{\infty}(\widetilde{\Gamma} \backslash G))
\end{CD}
\end{equation}  
and
\begin{equation} \label{bundlemapdualpicturediagram}
\begin{CD}
\mathrm{Hom}_K(V^\vee, \pi^{\infty} ) @>{\overline{\alpha}_{\pi}}>> \mathrm{Hom}_K( W^\vee, \pi^\infty)  \\
@V{\iota}VV @VV{\iota}V \\
\mathrm{Hom}_K(V^\vee, C^{\infty}(\widetilde{\Gamma} \backslash G)) @>{\overline{\alpha}}>> \mathrm{Hom}_K( W^\vee, C^{\infty}(\widetilde{\Gamma} \backslash G)),
\end{CD}
\end{equation}  
where the vertical maps $\iota$ in both diagrams are induced by the inclusion $\pi \subset L^2(\widetilde{\Gamma}\backslash G).$

\vspace{0.3cm}
\subsection{Automorphic incarnation of the spinor bundle and Dirac operator} \label{automorphicdirac}

Denote as before by $S$ be the standard representation of $\mathrm{U}_2$ on $\mathbb{C}^2$.  

\begin{defn}
Let $\widetilde{\Gamma} \subset G$ be a lattice lifting the closed hyperbolic 3-manifold lattice $\Gamma \subset \mathrm{PGL}_2(\C).$  The spinor bundle over $\Gamma \backslash \mathbb{H}^3 = \widetilde{\Gamma} \backslash \mathbb{H}^3 = \widetilde{\Gamma} \backslash G/K$ associated to the lift $\widetilde{\Gamma} \subset G$ of $\Gamma$ is the vector bundle $\overline{\mathcal{F}}_S.$  It comes equipped with the covariant derivative $\overline{\nabla}^S.$
\end{defn}

\subsubsection{Clifford multiplication on $\overline{\mathcal{F}}_S$} \label{cliffordmultiplication}
The linear map  $C: S \otimes \mathfrak{p}_0 \rightarrow S$ given by {$2i$} times matrix multiplication is $\mathrm{U}_2$-equivariant: 
\begin{align*}
C( (\mathrm{std} \otimes \mathrm{Ad})(k) (s \otimes X)) &= C( \mathrm{std}(k)s \otimes \mathrm{Ad}(k) X ) \\
&= {2i \cdot}(k X k^{-1})(k s) \\
&= k \left( {2i \cdot }Xs \right) \\
&= \mathrm{std}(k) C(s \otimes X). 
\end{align*}
In what follows, we will make use of the (physicists') Pauli matrices
\begin{equation}\label{physpauli}
\sigma_1 = \left( \begin{array}{cc} 0 & 1 \\ 1 & 0  \end{array} \right) \quad
\sigma_2 = \left( \begin{array}{cc} 0 & -i \\ i & 0  \end{array} \right) \quad
\sigma_3 = \left( \begin{array}{cc} 1 & 0 \\ 0 & -1  \end{array} \right).
\end{equation}
\begin{remark}\label{paulinorm}
In the hyperbolic metric, the Pauli matrices $\sigma_1,\sigma_2,\sigma_3$ are orthogonal to each other, but have norm $2$.  To see why the latter holds, notice that$$e^{t \sigma_3} = \left(  \begin{array}{cc} e^t & 0 \\ 0 & e^{-t}  \end{array} \right).$$  In the upper half-space model, this maps $(0,0,1)$ to $(0,0,e^{2t}),$ and the geodesic segment connecting them has length $2t.$  On the other hand, it has length $|| \sigma_3 || \cdot t$, so $||\sigma_3|| = 2$.  So the multiplier $2$ in the above definition of $C$ guarantees that Clifford multiplication is isometric.
\end{remark} 

\begin{remark}
The above Clifford multiplication is compatible with the orientation $\sigma_1 \wedge \sigma_2 \wedge \sigma_3$ in the sense that 
$$C(\sigma_1) C(\sigma_2) C(\sigma_3) = 8$$
is positive.
\end{remark}

The hyperbolic metric on $\mathfrak{p}_0$ defines a $K$-equivariant isomorphism $\iota: \mathfrak{p}_0^\vee \rightarrow \mathfrak{p}_0.$  We define $C'$ to be the composition of $C$ with $\iota$:
$$C' = S \otimes \mathfrak{p}^\vee_0 \xrightarrow{ 1 \otimes \iota} S \otimes \mathfrak{p}_0 \xrightarrow{C} S.$$

We will refer to the associated map on sections
\begin{equation*}
\Gamma(\overline{\mathcal{F}}_{S \otimes \mathfrak{p}_0^\vee}) \xrightarrow{\overline{C}'} \Gamma( \overline{\mathcal{F}}_S)
\end{equation*}
as \emph{Clifford multiplication}.

\subsubsection{The Dirac operator on $\Gamma(\overline{\mathcal{F}}_S)$} 

\begin{defn}
The Dirac operator $\overline{D}$ on $\Gamma(\overline{\mathcal{F}}_S)$ is defined to be the composition
$$\Gamma(\overline{\mathcal{F}}_S) \xrightarrow{\overline{\nabla}^S} \Gamma( \overline{\mathcal{F}}_{S \otimes \mathfrak{p}_0^\vee}) \xrightarrow{C'} \Gamma( \overline{\mathcal{F}}_S).$$
\end{defn}
Let $X_1,X_2,X_3$ be an oriented orthonormal basis for $\mathfrak{p}_0.$  The canonical isomorphism $\mathrm{Hom}(\mathfrak{p}_0,S) \simeq \mathfrak{p}_0^\vee \otimes S$ can be expressed as
$$T \mapsto \sum_{i = 1}^3 \langle \bullet, X_i \rangle \otimes T(X_i),$$
where $X_i$ is any orthonormal basis for $\mathfrak{p}.$  By the discussion from \S \ref{covariantderivative} and \S \ref{cliffordmultiplication}, the Dirac operator is given by the formula 
\begin{align*}
\overline{D}: \Gamma(\overline{\mathcal{F}}_S) &\rightarrow \Gamma(\overline{\mathcal{F}}_S) \\
(g \mapsto s(g)) &\mapsto \left(g \mapsto \sum_{i = 1}^3 {2i \cdot} m_{X_i}  X_i(s \circ L_g) \right) \\
&= \left(g \mapsto \sum_{i = 1}^3 {2i \cdot} m_{X_i} \frac{d}{dt}|_{t = 0}s(g e^{tX_i})  \right). 
\end{align*}
where $m_{X_i}$ denotes matrix multiplication by $X_i\in\mathfrak{p}_0=i\mathfrak{su}(2)$.   
\par
In the dual picture, $\overline{D} = \overline{C}' \circ \overline{\Theta}^S.$  Unravelling the formula for $\overline{\Theta}^S$ from \eqref{connectiondualpicture}, we find that $\overline{D}$ equals
\begin{align*}
\overline{D}&: \mathrm{Hom}_K(S^\vee, C^\infty( \widetilde{\Gamma} \backslash G) ) \rightarrow \mathrm{Hom}_K(S^\vee, C^\infty(\widetilde{\Gamma} \backslash G)) \\
T &\mapsto \sum_{i = 1}^3 \left( S^\vee \xrightarrow{ {2i \cdot} m_{X_i}^\vee} S^\vee \xrightarrow{T} C^\infty(\widetilde{\Gamma} \backslash G) \xrightarrow{f \mapsto \left(g \mapsto \frac{d}{dt}|_{t = 0}f(ge^{tX_i})  \right) } C^{\infty}(\widetilde{\Gamma} \backslash G) \right).
\end{align*}

\subsubsection{Representation-by-representation analysis of the Dirac operator} \label{repbyrepdirac}
For every irreducible representation $\pi$ of $G$, we define the $\pi$-isotypic Dirac operators
$$\overline{D}_{\pi} := \overline{C}_{\pi}' \circ \overline{\Theta}_{\pi}^S.$$

Suppose $\pi \subset L^2(\widetilde{\Gamma} \backslash G)$ is a subrepresentation.  By the commutativity of the diagrams \eqref{connectiondualpicturediagram} and \eqref{bundlemapdualpicturediagram}, we obtain the commutative diagram
\begin{equation} \label{diracoperatordualpicturediagram}
\begin{CD}
\mathrm{Hom}_K(S^\vee, \pi^{\infty} ) @>{\overline{D}_{\pi}}>> \mathrm{Hom}_K( S^\vee, \pi^\infty)  \\
@V{\iota}VV @VV{\iota}V \\
\mathrm{Hom}_K(S^\vee, C^{\infty}(\widetilde{\Gamma} \backslash G)) @>{\overline{D}}>> \mathrm{Hom}_K( S^\vee, C^{\infty}(\widetilde{\Gamma} \backslash G)),
\end{CD}
\end{equation}  
where the vertical maps $\iota$ in both diagrams are induced by the inclusion $\pi \subset L^2(\widetilde{\Gamma}\backslash G).$

\medskip

We will show in Proposition \ref{spinorreps} that if $\pi$ is an irreducible (unitary) representation of $G$ and $\mathrm{Hom}_K(S^\vee, \pi^{\infty})$ is non-zero, then it is $1$-dimensional.  For all such representations $\pi,$ $\overline{D}_{\pi}$ thus necessarily acts on $\mathrm{Hom}_K(S^\vee, \pi^\infty)$ by some scalar $\lambda_{\mathrm{Dirac},\pi}$ (which will be later computed in Proposition \ref{diraceigenvaluecomputation}).  Thus, $\iota \left( \mathrm{Hom}_K(S^\vee, \pi^\infty)\right)$ corresponds to an eigenline for the Dirac operator $\overline{D}$ with eigenvalue $\lambda_{\mathrm{Dirac},\pi}.$ In particular, the natural decomposition
\begin{equation*}
L^2(\overline{\mathcal{F}}_S)=\bigoplus_{\pi\in\widehat{G}}m_{\tilde{\Gamma}}(\pi)\cdot \mathrm{Hom}_K(S^\vee, \pi^{\infty}),
\end{equation*}
which is the analogue of the Matsushima decomposition we used in \cite{LL}, corresponds to the eigenspace decomposition of the Dirac operator.

\vspace{0.3cm}
\subsection{Automorphic incarnation of $q$-form bundles and the $\ast d$ operator} \label{automorphicstard}
Reprise all notation from \S \ref{automorphicdirac}.  As in \cite{LL}, in this setup it is more convenient to consider the operators on complex valued forms; to this effect, we will consider the $K$-representation $\mathfrak{p}=\mathfrak{p}_0\otimes \mathbb{C}$. In particular, the vector bundle $\overline{\mathcal{F}}_{\wedge^q \mathfrak{p}^\vee}$ is naturally identified with the vector bundle of (complex-valued) $q$-forms over $\widetilde{\Gamma} \backslash \mathbb{H}^3$.


\subsubsection{Hodge star on $\overline{\mathcal{F}}_{\wedge^q \mathfrak{p}^\vee}$} \label{hodgestar}
Consider the Hodge star operator
$$\ast: \wedge^q \mathfrak{p}^\vee \rightarrow \wedge^{3-q} \mathfrak{p}^\vee.$$
The action of $\mathrm{U}_2$ on $\mathfrak{p}^\vee$ factors through $\mathrm{PU}_2,$ the group of orientation preserving isometries of $\mathbb{H}^3 = G/K$ stabilizing $eK.$  Thus, the actions of $K$ on $\wedge^q \mathfrak{p}^\vee$ and $\wedge^{3-q} \mathfrak{p}^\vee$ commute with $\ast.$  It follows immediately:

\begin{lemma}
The associated map on sections $\Gamma(\overline{\mathcal{F}}_{\wedge^q \mathfrak{p}^\vee} ) \xrightarrow{\overline{\ast}} \Gamma( \overline{\mathcal{F}}_{\wedge^{3-q} \mathfrak{p}^\vee})$ is identified with Hodge star for the hyperbolic metric on $\widetilde{\Gamma} \backslash \mathbb{H}^3.$
\end{lemma}

\subsubsection{Exterior derivative}
Let $\overline{\nabla} := \overline{\nabla}^{\wedge^q \mathfrak{p}^\vee}$ be the covariant derivative on $\overline{\mathcal{F}}_{\wedge^q \mathfrak{p}^\vee}$ induced by the invariant connection on $G \rightarrow G/K$ described in \S \ref{covariantderivative}.  Let $\overline{\Theta} :=  \overline{\Theta}^{\wedge^q \mathfrak{p}^\vee}$ be the corresponding operator in the dual picture, described in \S \ref{dualpicture}. Composing with the wedge product map, $\wedge: \mathfrak{p}^\vee \otimes \wedge^q \mathfrak{p}^\vee \rightarrow \wedge^{q+1} \mathfrak{p}^\vee,$ which is $K$-equivariant, recovers the exterior derivative:
$$d = \Gamma \left( \overline{\mathcal{F}}_{\wedge^q \mathfrak{p}^\vee} \right) \xrightarrow{\overline{\nabla}} \Gamma \left( \overline{\mathcal{F}}_{\mathfrak{p}^\vee \otimes \wedge^q \mathfrak{p}^\vee} \right) \xrightarrow{\overline{\wedge}}  \Gamma \left( \overline{\mathcal{F}}_{\wedge^{q+1} \mathfrak{p}^\vee} \right).$$

In the dual picture, applying formula \eqref{connectiondualpicture} for $V = \wedge^q \mathfrak{p}^\vee$ yields the following formula for $d = \overline{\wedge} \circ \overline{\Theta}^{\wedge^q \mathfrak{p}^\vee}$: 
\begin{align} \label{exteriorderivativedualpicture}
d: \mathrm{Hom}_K( \wedge^q \mathfrak{p}, C^{\infty}(\widetilde{\Gamma} \backslash G)) &\rightarrow \mathrm{Hom}_K( \wedge^{q+1} \mathfrak{p}, C^{\infty}(\widetilde{\Gamma} \backslash G)) \nonumber \\
\omega &\mapsto \left( X_0 \wedge \cdots \wedge X_q \mapsto \sum_i X_i \cdot \omega( X_0 \wedge \cdots \wedge \widehat{X_i} \wedge \cdots \wedge X_q) \right).
\end{align}

See \cite[Chapter 1, \S 1]{BW}. Implicitly in \eqref{exteriorderivativedualpicture}, we have used the canonical duality between $\wedge^q \mathfrak{p}^\vee$ and $\wedge^q \mathfrak{p}.$

\subsubsection{Representation-by-representation analysis of $\ast d$} \label{repbyrepstard}
For every irreducible representation $\pi$ of $G$, we define the $\pi$-isotypic $\ast d$ operators 
$$(\ast d)_{\pi} := \overline{\ast}_{\pi} \circ \ast \overline{\wedge}_{\pi} \circ \overline{\Theta}_{\pi}^{\wedge^1 \mathfrak{p}^\vee}.$$

Suppose $\pi \subset L^2(\widetilde{\Gamma} \backslash G)$ is a subrepresentation.  By the commutativity of the diagrams \eqref{connectiondualpicturediagram} and \eqref{bundlemapdualpicturediagram}, it follows that the diagram
\begin{equation} \label{diracoperatordualpicturediagram}
\begin{CD}
\mathrm{Hom}_K(\wedge^1 \mathfrak{p}, \pi^{\infty} ) @>{(\ast d)_{\pi}}>> \mathrm{Hom}_K( \wedge^1 \mathfrak{p} , \pi^\infty)  \\
@V{\iota}VV @VV{\iota}V \\
\mathrm{Hom}_K(\wedge^1 \mathfrak{p}, C^{\infty}(\widetilde{\Gamma} \backslash G)) @>{\ast d}>> \mathrm{Hom}_K( \wedge^1 \mathfrak{p}, C^{\infty}(\widetilde{\Gamma} \backslash G)),
\end{CD}
\end{equation}  

where the vertical maps $\iota$ in both diagrams are induced by the inclusion $\pi \subset L^2(\widetilde{\Gamma}\backslash G).$

\medskip

Similarly to the case of the Dirac operator, we will show in Proposition \ref{reptheory1formeigenvector} that if $\pi$ is an irreducible (unitary) representation of $G$ and $\mathrm{Hom}_K(\wedge^1 \mathfrak{p}, \pi^{\infty})$ is non-zero, then it is 1-dimensional.  For all such representations $\pi,$ $(\ast d)_{\pi}$ thus necessarily acts on $\mathrm{Hom}_K(\wedge^1 \mathfrak{p}, \pi^\infty)$ by some scalar $\lambda_{\ast d,\pi}$ (which will be computed in Proposition \ref{stardeigenvaluecomputation}).  Thus, $\iota \left( \mathrm{Hom}_K(\wedge^1 \mathfrak{p}, \pi^\infty)\right)$ is an eigenline for the odd signature operator with eigenvalue $\lambda_{\ast d,\pi}.$ In particular, the natural decomposition
\begin{equation*}
L^2(\Omega^1)=\bigoplus_{\pi\in\widehat{G}}m_{\tilde{\Gamma}}(\pi)\cdot \mathrm{Hom}_K(\wedge^1\mathfrak{p}, \pi^{\infty}),
\end{equation*}
which is the essentially the Matsushima decomposition we used in \cite{LL}, corresponds to the eigenspace decomposition of odd signature operator. Notice that of course closed forms are in the kernel of $\ast d$; as in \cite{LL}, we will identify exactly which representations correspond to coexact forms.

\vspace{0.3cm}
\subsection{Frobenius reciprocity and $K$-isotypic vectors for induced representations} \label{frobrec}

Let $M = K\cap B$ be the subgroup of diagonal unitary matrices.  Let $\chi_{n_1,n_2}$ denote the character of $M$ given by
\begin{align*}
\chi_{n_1,n_2}: M& \rightarrow \mathbb{C}^\times  \\
 \left( \begin{array}{cc} t_1 & 0 \\ 0 & t_2  \end{array}  \right) &\rightarrow t_1^{n_1} \cdot t_2^{n_2},
\end{align*}
which is obtained by restricting the family of characters $\chi_{s_1,s_2,n_1,n_2}$ of $B$.
\par
Suppose $\pi \subset L^2( \widetilde{\Gamma} \backslash G)$ is a subrepresentation isomorphic to $\pi_{s_1,s_2,n_1,n_2}.$
The following Lemma will allow us to characterize whether $\pi$ contributes to the spinor (resp. $\ast d$) eigenvalue spectrum.  And if $\pi$ contributes, it will allow us to write down explicit vectors in $\mathrm{Hom}_K(S^\vee, \pi^\infty)$ (resp. $\mathrm{Hom}_K(S^\vee, \pi^\infty)$) whose images in $\mathrm{Hom}_K(S^\vee, C^\infty(\widetilde{\Gamma} \backslash G) )$ (resp. in $\mathrm{Hom}_K(S^\vee, C^\infty(\widetilde{\Gamma} \backslash G) )$) are Dirac eigenspinors (resp. $\ast d$ eigen 1-forms), per the discussion from \S \ref{automorphicdirac} (resp. \S \ref{automorphicstard}).

\begin{lemma} \label{frobreclemma}
Let $(V,\rho)$ be a finite dimensional representation of $K.$  The the map 
\begin{align*}
\mathrm{Hom}_M(V^\vee, \chi_{n_1,n_2}) &\rightarrow \mathrm{Hom}_K( V^\vee, \pi_{s_1,s_2,n_1,n_2} ) \\
T &\mapsto \left(\ell \mapsto \chi_{s_1,s_2,n_1,n_2}(b) \delta(b)^{1/2} \cdot T(\rho^\vee(k) \ell)\right),
\end{align*}
where $\delta$ is the modular function, is an isomorphism.
\end{lemma} 

In particular, this reduces the problem of determining $V^\vee$-isotypic vectors to the (essentially trivial) problem of understanding of $V^\vee$ as a representation of $M$.

\begin{proof}
This is proved in \cite[Appendix $B$]{LL}, to which we refer for a more detail discussion (and unraveling of the formulas). The isomorphism is the composite of the Frobenius reciprocity isomorphism 
\begin{align*}
\mathrm{Hom}_M(V^\vee, \chi_{n_1,n_2}) &\rightarrow \mathrm{Hom}_K(V^\vee, \mathrm{Ind}_M^K \chi_{n_1,n_2} ) \\
T &\mapsto \left(  \ell \mapsto f_{\ell}(k):= T(\rho^\vee(k) \ell) \right)
\end{align*}
and the map 
\begin{align*}
\mathrm{Hom}_K(V^\vee, \mathrm{Ind}_M^K \chi_{n_1,n_2} ) &\rightarrow \mathrm{Hom}_K(V^\vee, \pi_{s_1,s_2,n_1.n_2}) \\
S &\mapsto \left( \ell \mapsto g_{\ell}(bk) :=  \chi_{s_1,s_2,n_1,n_2}(b) \delta(b)^{1/2} \cdot  [S(\ell)](k)  \right).
\end{align*}
The latter is an isomorphism because $G = BK, M = B \cap K,$ and $\chi_{s_1,s_2,n_1,n_2} \cdot \delta^{1/2}$ restricted to $M$ equals $\chi_{n_1,n_2}.$
\end{proof}

\vspace{0.3cm}
\subsubsection{Representations contributing to the spinor spectrum}
 We need to apply Lemma \ref{frobreclemma} to the representation $S^{\vee}$, where again $S=\mathbb{C}^2$ is the standard representation of $\mathrm{U}_2$.

\begin{prop} \label{spinorreps}
Let $\pi$ be an irreducible unitary representation of $G.$  Then $\mathrm{Hom}_K(S^\vee, \pi^\infty)$ is either 0 or 1 dimensional.  It is non-zero (and hence 1-dimensional) if and only if $\pi$ is isomorphic to $\pi_{s_1,s_2,-1,0}$ (or, equivalently $\pi_{s_2,s_1,0,-1}$),  in which case a basis vector for $\mathrm{Hom}_K(S^\vee, \pi^\infty)$ is given by
$$\ell \mapsto f_{\ell}(bk) := \chi_{s_1,s_2,-1,0}(b) \delta(b)^{1/2} \cdot  [\mathrm{std}^\vee(k) \ell](e_1).$$
\end{prop}

\begin{proof}
We easily check that $\mathrm{Hom}_K(S^\vee, \pi^\infty) = 0$ if $\pi$ is isomorphic to some power of the determinant.  So we focus on the case $\pi \cong \pi_{s_1,s_2,n_1,n_2}.$

By Lemma \ref{frobreclemma}, $\mathrm{Hom}_K(S^\vee, \pi^\infty)$ is naturally isomorphic to $\mathrm{Hom}_M(S^\vee, \chi_{n_1,n_2}).$  Since $S^\vee |_M$ is isomorphic to $\chi_{-1,0} \oplus \chi_{0,-1},$ if follows that 
\begin{itemize}
\item
$\mathrm{Hom}_K(S^\vee, \pi_{s_1,s_2,n_1,n_2}^\infty)$ is non-zero iff $(n_1,n_2) = (-1,0)$ or $(0,-1).$  If non-zero, it is exactly 1-dimensional.

\item
The intertwining map 
$$\ell \mapsto \ell(e_1)$$
is a basis vector for the 1-dimensional space $\mathrm{Hom}_M(S^\vee, \chi_{-1,0}).$  Its image under the isomorphism from Lemma \ref{frobreclemma}, namely
$$\ell \mapsto f_{\ell}(bk) := \chi_{s_1,s_2,-1,0}(b) \delta(b)^{1/2} \cdot  [\mathrm{std}^\vee(k) \ell](e_1),$$

thus defines a basis vector for the 1-dimensional space $\mathrm{Hom}_K(S^\vee, \pi_{s_1,s_2,-1,0}).$
\end{itemize}
This proves the Lemma.
\end{proof}

\vspace{0.3cm}

\subsubsection{Representations contributing to the coclosed 1-form spectrum}
 We need to apply Lemma \ref{frobreclemma} to the representation $\wedge^1\mathfrak{p}=\mathfrak{p}$, where again $\mathfrak{p}=\mathfrak{p}_0\otimes \mathbb{C}$ and $\mathfrak{p}_0=i\mathfrak{su}(2)$ is acted on by $K$ via the adjoint representation. Denoting again the Pauli matrices in (\ref{physpauli}) by $\sigma_i$, we have that
\begin{equation}\label{pdual}
w_1 = \sigma_1 + i \sigma_2,\quad w_0 = \sigma_3,\quad w_{-1} = \sigma_1 - i \sigma_2,
\end{equation}
are the $M$-weight vectors for $\wedge^1 \mathfrak{p}$ of respective weights $(1, -1), (0,0)$ and $(-1,1)$. We denote the corresponding dual basis by $w_1^\vee, w_0^\vee, w_{-1}^\vee.$
\begin{prop} \label{reptheory1formeigenvector}
Let $\pi$ be an irreducible unitary representation of $G.$  Then $\mathrm{Hom}_K(\wedge^1\mathfrak{p}, \pi^\infty)$ is either 0 or 1 dimensional.  It is non-zero (and hence 1-dimensional) if and only if $\pi$ is isomorphic to $\pi_{s_1,s_2,1,-1}$ (or equivalently $\pi_{s_2,s_1,-1,1}$) or $\pi_{s_1,s_2,0,0}$; only the representations of the form $\pi_{s_1,s_2,1,-1}$ contribute to the coclosed 1-form spectrum, in which case a basis vector for $\mathrm{Hom}_K(\wedge^1\mathfrak{p}, \pi^\infty)$ is given by
$$v \mapsto f_{v}(bk) := \chi_{s_1,s_2,1,-1}(b) \delta(b)^{1/2} \cdot w_1^\vee \left(  \mathrm{Ad}(k) v \right).$$
\end{prop}

\begin{proof}
We easily check that $\mathrm{Hom}_K(\wedge^1 \mathfrak{p}, \pi^\infty) = 0$ if $\pi$ is isomorphic to some power of the determinant.  So we focus on the case $\pi \cong \pi_{s_1,s_2,n_1,n_2}.$

By Lemma \ref{frobreclemma}, $\mathrm{Hom}_K(\wedge^1 \mathfrak{p}, \pi^\infty)$ is naturally isomorphic to $\mathrm{Hom}_M(\wedge^1 \mathfrak{p}, \chi_{n_1,n_2}).$  Since $\wedge^1\mathfrak{p} |_M$ is isomorphic to $\chi_{1,-1} \oplus \chi_{0,0} \oplus \chi_{-1,1},$ if follows that 
\begin{itemize}
\item
$\mathrm{Hom}_K(\wedge^1\mathfrak{p}, \pi_{s_1,s_2,n_1,n_2}^\infty)$ is non-zero iff $(n_1,n_2) = (1,-1), (0,0),$ or $(-1,1).$  If non-zero, it is exactly 1-dimensional.  Furthermore, as argued in \cite[Lemma B.11]{LL}, the representations $\pi_{s_1,s_2,1,-1}$ and $\pi_{s_1,s_2,-1,1},$ isomorphic respectively to $\pi_{s_1-s_2,\pm 1}$ in the notation from \cite[Appendix B]{LL}, only contribute to the coclosed 1-form spectrum while the representations $\pi_{s_1,s_2,0,0},$ isomorphic to $\pi_{s_1-s_2,0}$ in the notation from \cite[Appendix B]{LL}, only contribute to the closed 1-form spectrum.  

\medskip

\item
The intertwining map $v \mapsto w_1^\vee(v)$ is a basis vector for the 1-dimensional space $\mathrm{Hom}_M(\wedge^1 \mathfrak{p} , \chi_{1,-1}).$  Its image under the isomorphism from Lemma \ref{frobreclemma}, namely
$$\ell \mapsto f_v(bk) := \chi_{s_1,s_2,1,-1}(b) \delta(b)^{1/2} \cdot  w_1^\vee \left(  \mathrm{Ad}(k)v \right),$$
thus defines a basis vector for the 1-dimensional space $\mathrm{Hom}_K(\wedge^1\mathfrak{p}, \pi_{s_1,s_2,1,-1}).$
\end{itemize}
This proves the Lemma.
\end{proof}

\vspace{0.3cm}
\subsection{Representation-by-representation calculation of Dirac and $\ast d$ eigenvalues}\label{eigencalc}

We conclude the proofs on the main results of this Appendix by computing the eigenvalues $\lambda_{\ast d,\pi}$ and $\lambda_{\mathrm{Dirac},\pi}$ of the eigenforms/eigenspinors corresponding to isotypic vectors.

\begin{remark}
While in \cite[Appendix $B$]{LL} we performed the computations in the natural Killing metric, and then normalized the result to obtain the answer for the hyperbolic metric, we will now work directly with the hyperbolic metric.  
\end{remark}

\begin{prop}[Calculation of $\lambda_{\mathrm{Dirac}, \pi}$] \label{diraceigenvaluecomputation}
Let $\pi = \pi_{s_1,s_2,-1,0}.$  Then
$$\lambda_{\mathrm{Dirac}, \pi} =  {i \cdot \left( \frac{s_1 - s_2}{2} \right)}.$$
\end{prop}

\begin{proof}
The space $\mathrm{Hom}_{T \cap K} \left( S^\vee, \chi_{-1,0} \right)$ has basis $\ell \mapsto \ell(e_1).$  Thus according to Lemma \ref{frobreclemma}, the unique (up to scaling) element $T$ of $\mathrm{Hom}_K \left( S^\vee, \pi_{s_1,s_2,-1,0} \right)$ equals
\begin{align*}
T(\ell) = f_{\ell}:G&\rightarrow\mathbb{C}\\
f_{\ell}(bk) &=  \delta^{1/2}(b) \chi_{s_1,s_2,0,-1}(b) \cdot \left[ \mathrm{std}^\vee(k) \ell \right](e_1).
\end{align*}
Suppose $\overline{D}_{\pi_{s_1,s_2,-1,0}} T = c T.$  Then, as $T(e_1^\vee)(1)=1$, we have
\begin{align*}
\lambda_{\mathrm{Dirac},\pi} &=c \\
&= c T(e_1^\vee)(1) \\
&= (\overline{D}_{\pi_{s_1,s_2,-1,0}} T)(e_1^\vee)(1),  
\end{align*}
so we will evaluate the latter next.  We calculate:
\begin{align*}
m_{\sigma_1}^\vee: e_1^\vee &\mapsto e_2^\vee, \\
e_2^\vee &\mapsto e_1^\vee \\
m_{\sigma_2}^\vee: e_1^\vee &\mapsto -i e_2^\vee, \\
e_2^\vee &\mapsto i e_1^\vee \\
m_{\sigma_3}^\vee: e_1^\vee &\mapsto e_1^\vee, \\
e_2^\vee &\mapsto -e_2^\vee 
\end{align*}
Therefore we have the explicit expression
\begin{align*}
 \overline{D}_{\pi_{s_1,s_2,-1,0}}T(e_1^\vee) &=\frac{2i}{2^2} \cdot \left( \sigma_1 \cdot T m_{\sigma_1}^\vee(e_1^\vee) 
+  \sigma_2 \cdot T m_{\sigma_2}^\vee(e_1^\vee) 
+ \sigma_3 \cdot T m_{\sigma_3}^\vee(e_1^\vee) \right)\\
&= \frac{i}{2}\left( \sigma_1 \cdot f_{e_2^\vee} 
+ \sigma_2 \cdot f_{-ie_2^\vee} 
+ \sigma_3 \cdot f_{e_1^\vee} \right),
\end{align*}
where the factor $2^2$ comes the fact that the Pauli matrices have norm $2$ in the hyperbolic metric, see Remark \ref{paulinorm}. To evaluate the above expression at 1, we express each Pauli matrix in the form $K_i + B_i$ where $K_i \in \mathfrak{k}$ and $B_i \in \mathfrak{b}$ (here $\mathfrak{k}$ and $\mathfrak{b}$ are the Lie algebras of $K$ and $B$ respectively):  
\begin{align*}
\sigma_1 &= \left( \begin{array}{cc} 0 & -1 \\ 1 & 0  \end{array} \right) + 2  \left( \begin{array}{cc} 0 & 1 \\ 0 & 0  \end{array} \right) =: K_1 + B_1 \\
\sigma_2 &= \left( \begin{array}{cc} 0 & i \\ i & 0  \end{array} \right) - 2  \left( \begin{array}{cc} 0 & i \\ 0 & 0  \end{array} \right)  =: K_2 + B_2 \\
\sigma_3 &= \left( \begin{array}{cc} 0 & 0 \\ 0 & 0  \end{array} \right) +  \left( \begin{array}{cc} 1 & 0 \\ 0 & -1  \end{array} \right) =: K_3 + B_3 
\end{align*}
Notice that $B_1,B_2\in \mathfrak{n}$, the Lie algebra of the group of unipotent matrices $N$. Because the functions $f_{\ell}$ are invariant under left translation by elements of $N,$
\begin{align*}
 \overline{D}_{\pi_{s_1,s_2,-1,0}}T(e_1^\vee) &= \frac{i}{2} \cdot \left( \sigma_1 \cdot f_{e_2^\vee} 
+ \sigma_2 \cdot f_{-ie_2^\vee} 
+ \sigma_3 \cdot f_{e_1^\vee} \right) \\
&=  \frac{i}{2} \cdot \left( K_1 \cdot f_{e_2^\vee} 
+ K_2 \cdot f_{-ie_2^\vee} 
+ B_3 \cdot f_{e_1^\vee} \right).
\end{align*}
We calculate each summand, evaluated at 1, individually:
\begin{align*}
(K_1 \cdot f_{e_2^\vee})(1) &= 1 \cdot \frac{d}{dt}|_{t = 0}[e_2^\vee]( e^{-tK_1} e_1) \\
&= [e_2^\vee]( -K_1 e_1) \\
&= -1, \\
(K_2 \cdot f_{-ie_2^\vee})(1) &= 1 \cdot \frac{d}{dt}|_{t = 0}[-ie_2^\vee]( e^{-tK_2} e_1) \\
&= [-ie_2^\vee]( -K_2 e_1) \\
&= [-ie_2^\vee](-ie_2) \\
&= -1, \\
(B_3 \cdot f_{e_1^\vee})(1) &= \frac{d}{dt}|_{t = 0} \delta^{1/2} \left( \begin{array}{cc} e^t & 0 \\ 0 & e^{-t}  \end{array} \right) \cdot \chi_{s_1,s_2,-1,0}  \left( \begin{array}{cc} e^t & 0 \\ 0 & e^{-t}  \end{array} \right)  \cdot [e_1^\vee]( e_1) \\
&= \frac{d}{dt}|_{t = 0} e^{2t}\cdot e^{t \cdot (s_1 - s_2)} \cdot 1 \\
&= s_1 - s_2 + 2.
\end{align*}
Summing these contributions, $\overline{D}_{\pi_{s_1,s_2,-1,0}}$ acts on $\mathrm{Hom}_K(S^\vee, \pi_{s_1,s_2,-1,0})$ by the scalar
\begin{align} \label{diraceigenvalue}
c &=  \overline{D}_{\pi_{s_1,s_2,-1,0}} T(e_1^\vee)(1) \nonumber \\
&= \frac{i}{2} \cdot\left((-1) + (-1) + (s_1 - s_2 + 2)  \right) \nonumber\\
&= {i \cdot \left( \frac{s_1 - s_2}{2} \right)}\end{align}
and the proof is completed.
\end{proof}

\vspace{0.3cm}
We conclude by computing $\lambda_{\ast d, \pi}$, obtaining the desired refinement of the result of \cite{LL}.
\begin{prop}[Calculation of $\lambda_{\ast d, \pi}$] \label{stardeigenvaluecomputation}
Let $\pi = \pi_{s_1,s_2,1,-1}.$  Then
$$\lambda_{\ast d, \pi} =  {i \cdot \left( \frac{s_1 - s_2}{2} \right)}.$$
\end{prop}

\begin{proof}
We will use the notation from (\ref{pdual}). According to Proposition \ref{reptheory1formeigenvector}, the unique (up to scaling) element $T$ of $\mathrm{Hom}_K \left( \wedge^1 \mathfrak{p}, \pi_{s_1,s_2,1,-1} \right)$ equals
\begin{align*}
T(v) = f_v: G&\rightarrow \mathbb{C}\\
 f_v(bk)&= \delta^{1/2}(b) \chi_{s_1,s_2,1,-1}(b) \cdot w_1^\vee \left( \mathrm{Ad}(k) v \right).
\end{align*}
Recall the $\pi$-isotypic version $(\ast d)_{\pi}$ of $\ast d$ from \S \ref{repbyrepstard}.  Because $\mathrm{Hom}_K( \wedge^1 \mathfrak{p}, \pi_{s_1,s_2,1,-1})$ is 1-dimensional, $(\ast d)_{\pi}$ acts on it by a scalar.  We next calculate the eigenvalue of $(\ast d)_{\pi}$ acting on $\mathrm{Hom}_K( \wedge^1 \mathfrak{p}, \pi)$ for $\pi = \pi_{s_1,s_2,1,-1}.$  To lighten notation, we refer to $d_{\pi}, \ast_{\pi}$ and $(\ast d)_{\pi}$ simply by $d, \ast,$ and $\ast d$ respectively. Suppose $(\ast d) T = cT.$  Then
\begin{align*}
\lambda_{\ast d, \pi} &= c \\
&= c T(w_1)(1) \\
&= ((\ast d) T)(w_1) (1),  
\end{align*}
so we will evaluate the latter next. Observe that: 
\begin{align*}
dT(\sigma_1 \wedge \sigma_2) &= \sigma_1 \cdot T(\sigma_2) - \sigma_2 \cdot T(\sigma_1) \\
dT(\sigma_2 \wedge \sigma_3) &= \sigma_2 \cdot T(\sigma_3) - \sigma_3 \cdot T(\sigma_2) \\
dT(\sigma_3 \wedge \sigma_1) &= \sigma_3 \cdot T(\sigma_1) - \sigma_1 \cdot T(\sigma_3),
\end{align*}
implying that 
\begin{align*}
2 \cdot \ast d T(\sigma_3) &= dT(\sigma_1 \wedge \sigma_2) = \sigma_1 \cdot T(\sigma_2) - \sigma_2 \cdot T(\sigma_1) \\
2 \cdot \ast d T(\sigma_1) &= dT(\sigma_2 \wedge \sigma_3) = \sigma_2 \cdot T(\sigma_3) - \sigma_3 \cdot T(\sigma_2) \\
2 \cdot \ast d T(\sigma_2) &= dT(\sigma_3 \wedge \sigma_1) = \sigma_3 \cdot T(\sigma_1) - \sigma_1 \cdot T(\sigma_3),
\end{align*}
where the factor of $2$ appears because the norm of the Pauli matrices in the hyperbolic metric is $2$, see Remark \ref{paulinorm}. It thus follows:
\begin{align*}
2 \cdot \ast dT(w_1) &= 2 \cdot \left(  \ast dT(\sigma_1) + i \ast dT(\sigma_2) \right) \\
&= \sigma_2 \cdot T(\sigma_3) - \sigma_3 \cdot T(\sigma_2) + i \left(  \sigma_3 \cdot T(\sigma_1) - \sigma_1 \cdot T(\sigma_3) \right).
\end{align*}

To evaluate the above expression at 1, recall the decomposition of the Pauli matrices $\sigma_i=K_i+B_i$ introduced in the proof of Proposition \ref{diraceigenvaluecomputation}. Again, the functions $f_v$ are invariant under left translation by elements of $N$, and we therefore compute
\begin{align*}
2 \cdot \ast d T(w_1) &=  \sigma_2 \cdot T(\sigma_3) - \sigma_3 \cdot T(\sigma_2) + i \left(  \sigma_3 \cdot T(\sigma_1) - \sigma_1 \cdot T(\sigma_3) \right) \\
&=   K_2 \cdot f_{\sigma_3} - B_3 \cdot f_{\sigma_2} + i \left( B_3 \cdot f_{\sigma_1} - K_1 \cdot f_{\sigma_3} \right).
\end{align*}
We calculate each summand, evaluated at 1, individually. The terms involving $B_3$ are
\begin{align*}
(B_3 \cdot f_{\sigma_2})(1) &= \frac{d}{dt}|_{t = 0} \delta^{1/2} \left( \begin{array}{cc} e^t & 0 \\ 0 & e^{-t}  \end{array} \right) \cdot \chi_{s_1,s_2,1,-1}  \left( \begin{array}{cc} e^t & 0 \\ 0 & e^{-t}  \end{array} \right)  \cdot w_1^\vee( \sigma_2) \\
&= \frac{d}{dt}|_{t = 0} e^{2t}\cdot e^{t \cdot (s_1 - s_2)} \cdot \frac{1}{2} \left(-i \right) \\
&= (2 + s_1 - s_2) \cdot \frac{1}{2} \left(-i \right),\\
(B_3 \cdot f_{\sigma_1})(1) &= \frac{d}{dt}|_{t = 0} \delta^{1/2} \left( \begin{array}{cc} e^t & 0 \\ 0 & e^{-t}  \end{array} \right) \cdot \chi_{s_1,s_2,1,-1}  \left( \begin{array}{cc} e^t & 0 \\ 0 & e^{-t}  \end{array} \right)  \cdot w_1^\vee( \sigma_1) \\
&= \frac{d}{dt}|_{t = 0} e^{2t}\cdot e^{t \cdot (s_1 - s_2)} \cdot \frac{1}{2} \\
&= (2 + s_1 - s_2) \cdot \frac{1}{2},
\end{align*}
while those involving the $K_i$ are
\begin{align*}
(K_2 \cdot f_{\sigma_3})(1) &= 1 \cdot \frac{d}{dt}\lvert_{t = 0} w_1^\vee( \mathrm{Ad}(e^{tK_2}) \sigma_3) \\
&= w_1^\vee(  [K_2, \sigma_3] ) \\
&= w_1^\vee( 2 \sigma_2 ) \\
&= -i,\\
(K_1 \cdot f_{\sigma_3})(1) &= 1 \cdot \frac{d}{dt}|_{t = 0}w_1^\vee( Ad(e^{tK_1}) \sigma_3) \\
&= w_1^\vee( [ K_1,\sigma_3] ) \\
&= w_1^\vee( 2\sigma_1 ) \\
&= 1. 
\end{align*}
Summing these contributions, $\ast d$ acts on $\mathrm{Hom}_K(\wedge^1 \mathfrak{p}, \pi_{s_1,s_2,1,-1})$ by the scalar
\begin{align*}
c &=  \ast d T(w_1) (1) \\
&= \frac{1}{2} \cdot \left( -i -  (2 + s_1 - s_2) \cdot \frac{1}{2} \left(-i \right) + i \left( (2 + s_1 - s_2) \cdot \frac{1}{2} - 1 \right) \right) \\
&= {i \cdot \left( \frac{s_1 - s_2}{2} \right)}\end{align*}
and the proof is concluded.
\end{proof}

\vspace{0.5cm}
\section{Proof of the explicit trace formulas}
We will now prove the results in Section \ref{traceformulas}, starting with the more involved case of spinors in Theorem \ref{spinorformula}.

\subsection{Isolating the spinor spectrum} \label{isolatingspinors}
We discuss suitable choices of test function to specialize the general trace formula for $G$ in Proposition \ref{preliminaryreptheorytraceformula}. According to Proposition \ref{spinorreps}, the irreducible unitary representations of $G$ contributing to the spinor spectrum are precisely those isomorphic to $\pi_{s_1,s_2,-1,0}$ or $\pi_{s_1,s_2,0,-1}$ for purely imaginary $s_1,s_2.$  To isolate the latter representations, we use the test functions of the form

$$F \left( \begin{array}{cc} e^{u/2} e^{i \alpha} & 0 \\ 0 & e^{-u/2} e^{i \beta} \end{array} \right) = \begin{cases} H_{+}(u) \left( e^{i \alpha} + e^{i \beta} \right) &H_{+} \text{ even} \\ H_{-}(u) \left( e^{i \alpha} - e^{i \beta} \right) &H_{-} \text{ odd} \end{cases}$$
We now discuss how each term in the trace formula simplifies when choosing test functions of the latter shapes.

\subsubsection{Regular geometric terms for spinors} \label{regulargeometricspinor}
For our normalization of Haar measure (\ref{Haar}), the covolume of the centralizer of $\gamma$ equals $\ell(\gamma_0).$ If its conjugacy class in $G$ is represented by
\begin{equation*}
\left( \begin{array}{cc} e^{\ell/2} e^{i \theta} e^{i \phi} & 0 \\ 0 & e^{-\ell/2} e^{-i \theta} e^{i \phi} \end{array} \right),
\end{equation*}
then its complex length equals $\ell + i 2\theta$ and its corresponding summand on the geometric side of the trace formula from Proposition \ref{preliminaryreptheorytraceformula} equals
\begin{align*}
= \ell_0 \cdot \frac{1}{|1 - e^{\ell + i 2\theta}| \cdot |1 - e^{-\ell - i 2\theta}| } \cdot H_{\pm}(\ell) \cdot \left(e^{i (\theta + \phi)} \pm e^{i (-\theta + \phi)} \right).
\end{align*}
Here we use that the Weyl discriminant is the same as in the trace formula from \cite{LL}:
$$|D(t_{\gamma})|^{1/2} = \frac{1}{|1 - e^{\C \ell(\overline{\gamma})}| \cdot |1 - e^{- \C \ell(\overline{\gamma}) }|},$$
where $\overline{\gamma}$ denotes the image of $\gamma$ in $\mathrm{PGL}_2(\C).$  In particular, it is independent of $\phi.$

\vspace{0.3cm}

\subsubsection{Contribution of central elements for spinors} \label{identityspinor}
Let $z \in Z \subset G$ be central, so
\begin{equation*}
z =  \left( \begin{array}{cc}  e^{i\phi} & 0 \\ 0 &  e^{i\phi} \end{array} \right)
\end{equation*}
for some $\phi\in\mathbb{R}/2\pi\mathbb{Z}$. Recalling that for the two different parametrizations of $T$ we have $u=2v$, the contribution of the conjugacy class of $z$ in the trace formula of Proposition \ref{preliminaryreptheorytraceformula} is then
\begin{align*}
&-\frac{1}{8\pi}\cdot \vol(Y)\cdot \left( \frac{\partial^2}{\partial v^2} + \frac{\partial^2}{\partial \theta^2} \right)|_{t = z} F\\
=&-\frac{1}{8\pi}\cdot \vol(Y)\cdot \left( \frac{\partial^2}{\partial v^2} + \frac{\partial^2}{\partial \theta^2} \right)|_{(v,\theta)=(0,0)}H_+(2v)(e^{i(\theta+\phi)}+e^{i(-\theta+\phi)})\\
=&\frac{1}{8\pi} \cdot \vol(Y) \cdot 2 e^{i\phi} \cdot \left(H_{+}(0) - 4 H_{+}''(0) \right)\\
=&\frac{1}{4\pi} \cdot \vol(Y) \cdot e^{i\phi} \cdot \left(H_{+}(0) - 4 H_{+}''(0) \right).
\end{align*}
In our case, we will be interested in lifts of torsion-free lattices $\Gamma\subset\mathrm{PSL}_2(\mathbb{C})$, so that we will only need to consider the case $z=1$, corresponding to $\phi=0$.

\subsubsection{Regular spectral terms for spinors} \label{regularspectralspinor}
Let  $s_1 = i r_1, s_2 = ir_2 \in i \mathbb{R}.$  The summand on the spectral side of the trace formula from Proposition \ref{preliminaryreptheorytraceformula} corresponding to $\pi_{s_1,s_2,n_1,n_2}$ then equals
\begin{align*}
\widehat{F}(\chi^{-1}_{s_1,s_2,n_1,n_2}) &= \int_T F(t) \chi_{s_1,s_2,n_1,n_2}(t) dt \\
&= \int\int\int H_{\pm}(u) \cdot \left(e^{i \alpha} \pm e^{i \beta} \right) \cdot e^{u/2 \cdot s_1} e^{-u/2 \cdot s_2} \cdot e^{i n_1 \alpha} \cdot e^{i n_2 \beta} du \frac{d \alpha}{2\pi} \frac{d\beta}{2\pi}. \\
\end{align*}
In the case of $H_{+} \cdot (e^{i \alpha} + e^{i \beta}),$ the latter equals 
$$\widehat{H_{+}} \left( \frac{r_2 - r_1}{2} \right) \cdot \begin{cases} + 1 & \text{ if } (n_1,n_2) = (-1,0) \text{ or } (0,-1) \\ 0 &\text{ otherwise}  \end{cases}.$$
while in the case of $H_{-} \cdot (e^{i \alpha} - e^{i \beta}),$ it equals
$$\widehat{H_{-}} \left( \frac{r_2 - r_1}{2} \right) \cdot \begin{cases} + 1 & \text{ if } (n_1,n_2) = (-1,0) \\ -1 &\text{ if } (n_1,n_2) = (0,-1) \\ 0 &\text{ otherwise.}  \end{cases}.$$

\subsubsection{Contribution of $\det^k$ to the spectral side}

For every $k\in\mathbb{Z}$, the contribution of $\det^k$ to the trace formula for the test function $H_{\pm}(u) \cdot (e^{i \alpha} \pm e^{i \beta})$ equals 
\begin{align*}
& \frac{1}{|W|} \int_T |D(t^{-1})|^{1/2} \cdot \det(t)^k \cdot F(t)  dt \\
&= \frac{1}{|W|} \int \int \int \left(  e^u + e^{-u} - \left( e^{i(\alpha - \beta)} + e^{i(\beta - \alpha)} \right) \right) \cdot e^{k \cdot i(\alpha + \beta)} \cdot H_{\pm}(u) \cdot (e^{i \alpha} \pm e^{i \beta}) \; \frac{d\alpha}{2\pi} \frac{d\beta}{2\pi} du \\
&= 0,
\end{align*}
since the integral of every non-trivial character over the torus $\mathrm{U}_1 \times \mathrm{U}_1$ equals 0. 

\vspace{0.3cm}
\subsection{Trace formula for spinors} \label{traceformulaspinors}
We are finally ready to prove the trace formulas that appear in Theorem \ref{spinorformula}.

\subsubsection{Even spinor trace formula} 
Let $H_{+}$ be a smooth, compactly supported, even test function on $\mathbb{R}.$  Let $\widetilde{\Gamma}$ be the lift to $G$ of the closed hyperbolic 3-manifold group $\Gamma \subset \mathrm{PGL}_2(\C)$, (so that $\tilde{\Gamma}\cap Z=\{1\}$).  For $\gamma \in \widetilde{\Gamma},$ suppose that $\gamma$ is conjugate to 
$$\left( \begin{array}{cc} z_1 & 0 \\ 0 & z_2 \end{array} \right) = \left( \begin{array}{cc} e^{u/2} e^{i\alpha}  & 0 \\ 0 & e^{-u} e^{i \beta} \end{array} \right) = \left( \begin{array}{cc} e^{u/2} e^{i\theta} e^{i \phi} & 0 \\ 0 & e^{-u/2} e^{-i \theta} e^{i \phi} \end{array} \right).$$
Let $z = z_1 / z_2.$  Let $\ell_0$ denote the length of the primitive closed geodesic some multiple of which corresponds to $\gamma.$  Unwinding all terms in the trace formula from Proposition \ref{preliminaryreptheorytraceformula} as in \S \ref{isolatingspinors} yields the following trace formula:  
\begin{align*}
& \sum m_{\widetilde{\Gamma}}(\pi_{s_1,s_2,-1,0}) \cdot \widehat{H_{+}} \left( \frac{r_2 - r_1}{2}  \right) \\
&= {\frac{1}{4\pi}} \cdot \vol(Y) \cdot \left( H_{+}(0) - 4 H_{+}''(0) \right) \\
&+ \sum_{[\gamma]} \ell_0 \cdot \frac{1}{|1 - z| \cdot |1 - z^{-1}|} \cdot \left( e^{i \alpha} + e^{i \beta} \right) \cdot H_{+}(u),
\end{align*}
where the LHS is summed over all isomorphism classes of irreducible unitary representations of $G$ isomorphic to some $\pi_{s_1,s_2,-1,0}.$
Since $H_{+}$ is even, the left side is real.  So taking real parts:

\begin{align} \label{traceformulaevenspinorsrepresentationtheory}
& \sum m_{\widetilde{\Gamma}}(\pi_{s_1,s_2,-1,0}) \cdot  \widehat{H_{+}} \left( \frac{r_2 - r_1}{2}  \right)= \nonumber \\
=& {\frac{1}{4\pi}} \cdot \vol(Y) \cdot \left( H_{+}(0) - 4 H_{+}''(0) \right) \nonumber + \sum_{[\gamma]} \ell_0 \cdot \frac{1}{|1 - z| \cdot |1 - z^{-1}|} \cdot \left( \cos \alpha + \cos \beta \right) \cdot H_{+}(u) \nonumber \\
=& {\frac{1}{4\pi}} \cdot \vol(Y) \cdot \left( H_{+}(0) - 4 H_{+}''(0) \right) + \sum_{[\gamma]} \ell_0 \cdot \frac{1}{|1 - z| \cdot |1 - z^{-1}|} \cdot \left( 2 \cos \theta \cos \phi \right) \cdot H_{+}(u). 
\end{align}


Finally, by Proposition \ref{diraceigenvaluecomputation} the Dirac eigenvalue on $\mathrm{Hom}_K( S^\vee, \pi_{s_1,s_2,0,-1} )$, with orientation $\sigma_1 \wedge \sigma_2 \wedge \sigma_3,$ equals $i \left( \frac{s_1 - s_2}{2} \right) =  \frac{r_2 - r_1}{2}.$  Thus, by Proposition \ref{spinorreps} we can reexpress the trace formula from \eqref{traceformulaevenspinorsrepresentationtheory} geometrically as
\begin{align*}
&\frac{1}{2} \cdot \sum_{\text{Dirac eigenvalues for } \widetilde{\Gamma}} m_{\widetilde{\Gamma}}(\lambda) \cdot \widehat{H_{+}} \left( \lambda \right)= \\
&= \frac{1}{2\pi} \cdot \vol(Y) \cdot \left( \frac{1}{4} \cdot H_{+}(0) - H_{+}''(0) \right)+ \sum_{[\gamma]} \ell_0 \cdot \frac{1}{|1 - z| \cdot |1 - z^{-1}|} \cdot \left( \cos \theta \cos \phi \right) \cdot H_{+}(u). \\
\end{align*}
This is the even trace formula in Theorem \ref{spinorformula}.

\subsubsection{Odd spinor trace formula: representation theoretic form} 
Let $H_{-}$ be a smooth, compactly supported, odd test function on $\mathbb{R}.$ Using the notation of the previous subsection, unwinding all terms in the trace formula from Proposition \ref{preliminaryreptheorytraceformula} as in \S \ref{isolatingspinors} yields the following trace formula:  
\begin{align*}
& \sum m_{\widetilde{\Gamma}}(\pi_{s_1,s_2,-1,0}) \cdot  \widehat{H_{-}} \left( \frac{r_2 - r_1}{2}  \right)  \\
&= -\sum_{[\gamma]} \ell_0 \cdot \frac{1}{|1 - z| \cdot |1 - z^{-1}|} \cdot \left( e^{i \alpha} - e^{i \beta} \right) \cdot H_{-}(u)
\end{align*}
where the LHS is summed over all isomorphism classes of irreducible unitary representations of $G$ isomorphic to some $\pi_{s_1,s_2,-1,0}.$  Since $H_{-}$ is odd, its Fourier transform is purely imaginary.  So looking imaginary parts:

\begin{align} \label{traceformulaoddspinorsrepresentationtheory}
& \sum m_{\widetilde{\Gamma}}(\pi_{s_1,s_2,-1,0}) \cdot  \widehat{H_{-}} \left( \frac{r_2 - r_1}{2}  \right)  \nonumber \\
&= i \sum_{[\gamma]} \ell_0 \cdot \frac{1}{|1 - z| \cdot |1 - z^{-1}|} \cdot \left( \sin \alpha - \sin \beta \right) \cdot H_{-}(u) \nonumber \\
&= i \sum_{[\gamma]} \ell_0 \cdot \frac{1}{|1 - z| \cdot |1 - z^{-1}|} \cdot \left( 2 \sin \theta \cos \phi \right) \cdot H_{-}(u).
\end{align}
By \eqref{diraceigenvaluecomputation}, the Dirac eigenvalue on $\mathrm{Hom}_K( S^\vee, \pi_{s_1,s_2,-1,0} )$, with orientation $\sigma_1 \wedge \sigma_2 \wedge \sigma_3,$ equals $i \left( \frac{s_1 - s_2}{2} \right) = \frac{r_2 - r_1}{2}.$  Thus by Proposition \ref{spinorreps}, we can reexpress the above trace formula from \eqref{traceformulaoddspinorsrepresentationtheory} geometrically:

\begin{align} \label{geoetricoddspinortraceformula}
& \frac{1}{2} \sum_{\text{Dirac eigenvalues for } \widetilde{\Gamma}} m_{\widetilde{\Gamma}}(\lambda) \cdot  \widehat{H_{-}} \left( \lambda  \right) \nonumber \\
&= i \sum_{[\gamma]} \ell_0 \cdot \frac{1}{|1 - z| \cdot |1 - z^{-1}|} \cdot \left( \sin \theta \cos \phi \right) \cdot H_{-}(u).
\end{align}
This is the odd trace formula in Theorem \ref{spinorformula}; let us point out again that in \eqref{geoetricoddspinortraceformula}, the Dirac operator is taken relative to the orientation $\sigma_1 \wedge \sigma_2 \wedge \sigma_3$ on $\mathbb{H}^3.$

\vspace{0.3cm}
\subsection{The trace formula for coexact 1-forms}

Specializing the trace formula to isolate coexact 1-forms is more straightforward than specializing the trace formula to isolate spinors as in \S \ref{traceformulaspinors}.  We content ourselves with highlighting the main differences between specializing to 1-forms versus specializing to spinors.
\begin{itemize}
\item
Irreducible unitary subrepresentations of $L^2(\widetilde{\Gamma} \backslash G)$ contributing to the coclosed 1-form spectrum are precisely those isomorphic to $\pi_{s_1,s_2,1,-1}$ and $\pi_{s_1,s_2,-1,1}$ for $s_1, s_2 \in i \mathbb{R}.$  We isolate those representations using the test functions $H_{\pm}(u) \cdot \left(  e^{i(\alpha - \beta)} \pm e^{i(\beta - \alpha)} \right),$ where $H_{+}$ is even and $H_{-}$ is odd.

\medskip

\item
For the latter test functions $H_{\pm}(u) \cdot \left(e^{i(\alpha - \beta)} \pm e^{i(\beta - \alpha)} \right),$ the summand corresponding to the representation $\pi_{s_1,s_2,1,-1}$ equals $\widehat{H_{\pm}} \left( \frac{r_2 - r_1}{2} \right).$  By Proposition \ref{stardeigenvaluecomputation}, $\frac{r_2 - r_1}{2}$ is the eigenvalue of $\ast d$ acting on the $\wedge^1 \mathfrak{p}$ isotypic vector of $\pi$ (for the orientation $\sigma_1 \wedge \sigma_2 \wedge \sigma_3$ on $\mathbb{H}^3$).  For a subrepresentation $\pi \subset L^2(\widetilde{\Gamma} \backslash G)$ isomorphic to $\pi_{s_1,s_2,1,-1},$ the latter equals the eigenvalue of $\ast d$ acting on the $\wedge^1 \mathfrak{p}$-isotypic vectors of $\pi_{s_1,s_2,1,-1}.$  

\medskip

\item
The contribution of $\det^k$ to the trace formula for 1-forms is non-trivial only if $k = 0.$  If $k = 0,$ the contribution equals 0 for test functions $H_{-}(u) \cdot \left( e^{i(\alpha - \beta)} - e^{i(\beta - \alpha)} \right)$ for odd $H_{-}$ and equals $-\widehat{H_{+}}(0)$ for the test function
\begin{equation*}
H_{+}(u) \cdot \left( e^{i(\alpha - \beta)} + e^{i(\beta - \alpha)} \right)=2H_+(2v)\cos(2\theta)
\end{equation*}
with $H_{+}$ even.

\medskip

\item
In the even case for the test function $2H_+(2v)\cos(2\theta)$ the identity contribution is $\frac{1}{2\pi}\cdot \vol(Y)\cdot 2(H_+(0)-H_+''(0))$.

\end{itemize}
From this, one readily obtains the formulas for coexact $1$-forms in Theorem \ref{coexformula}.

\vspace{0.5cm}

\section{Admissibility of Gaussian functions}\label{gaussiantest}

Gaussian functions were convenient to use at several points in our arguments.  The next result proves that functions of sufficiently fast decay, e.g. Gaussians, are admissible for use in the trace formula.
\begin{prop}
Let $f$ be an even or odd smooth function on $\mathbb{R}.$  Suppose that 
$$\sum_{n = 0}^\infty e^n \cdot \sup_{x \in [n,n+1]} |f(x)| < \infty.$$
Then both the geometric and spectral sides of the spinor trace formula for closed hyperbolic 3-manifolds converge absolutely for the test function $f$ and they are equal.  The same statement holds for the 1-form trace formula applied to even or odd test functions and the 0-form trace formula applied to even test functions $f.$
\end{prop}

\begin{proof}
We focus on the spinor case (arguments in the other cases being identical). Convergence of the spectral side of the trace formula for $f$ follows because $\widehat{f}$ is Schwartz and the number of spectral parameters in $[\nu,\nu+1]$ is quadratic in $\nu$ (see Section \ref{localweyl}) 

Convergence of the geometric side of the trace formula follows by our hypothesis. Indeed, 
\begin{align} \label{geometricsideconverges}
& \sum_{\gamma}\left| \ell(\gamma_0) \cdot \frac{(\cos \text{ or } \sin) \theta \cos \phi}{|1 - e^{\C \ell(\gamma)}| \cdot |1 - e^{-\C \ell(\gamma)}|} \cdot f( \ell(\gamma)) \right| \nonumber \\
&\leq \sum_{n = 0}^\infty \left( \sum_{\ell(\gamma) \in [n,n+1]} \ell(\gamma_0) \right) \cdot \frac{C}{e^n} \cdot \sup_{x \in [n,n+1]} |f(x)| \nonumber \\
&\leq \sum_{n = 0}^\infty D \cdot e^{2n} \cdot \frac{C}{e^n} \cdot \sup_{x \in [n,n+1]} |f(x)| \nonumber \\
&=CD \cdot \sum_{n = 0}^\infty e^{n} \cdot \sup_{x \in [n,n+1]} |f(x)| \nonumber \\
&< \infty.
\end{align}
Above, we have used that $\sum_{\ell(\gamma) \in [L,L+1]} \ell(\gamma_0)$ has order of magnitude $e^{2L}$ (cf. the discussion in Section \ref{geodesicSW}) and that $|1 - e^{\C \ell(\gamma)}| \cdot |1 - e^{-\C \ell(\gamma)}|$ has order of magnitude $e^{\ell(\gamma)}$ for $\ell(\gamma)$ large.
\\
\par
The trace formula for $f$ will be a consequence of the limit of the trace formulas applied to the test functions $f(x) \cdot b(x / R),$ where $b$ is a smooth even, compactly supported bump function with $\widehat{b}$ everywhere positive and $b(0) = 1.$  As $R \rightarrow \infty,$ the geometric and spectral sides of the trace formula for $f(x) \cdot b(x/R)$ respectively converge pointwise to the geometric and spectral sides of the trace formula for $f$; this is immediate for the geometric side and it follows on the spectral side because $f(x) \cdot b(x/R)$ has Fourier transform $\left(\widehat{f} \ast g_R\right)(t),$ where
\begin{equation*}
g_R(t) := \frac{1}{2\pi} \cdot R \cdot \widehat{b}(Rt)
\end{equation*}
is an approximate identity (i.e. it is everywhere positive of total mass 1 and concentrated increasingly near the origin as $R \to \infty$). Here we use that for our definition of Fourier transform $\widehat{f\cdot g}=\frac{1}{2\pi}\widehat{f} {\ast} \widehat{g}$. It therefore suffices to prove that the tails of the geometric and spectral sides of the trace formula, applied to the test function $f(x) \cdot b(x/R),$ approach 0 uniformly as $R \rightarrow \infty.$
\par
Uniform smallness of the tail on the geometric side follows by our hypothesis. Indeed,
\begin{align*}
&\sum_{\ell(\gamma) \geq k} \left| \ell(\gamma_0) \cdot \frac{(\cos \text{ or } \sin) \theta \cos \phi}{|1 - e^{\C \ell(\gamma)}| \cdot |1 - e^{-\C \ell(\gamma)}|} \cdot f( \ell(\gamma)) \right|\\
&\leq CD \cdot \sum_{n = k}^\infty e^{n} \cdot \sup_{x \in [n,n+1]} |f(x) b(x/R)| \hspace{0.5cm} \text{exactly as in \eqref{geometricsideconverges}} \\
&\leq CD \cdot ||b||_{\infty} \cdot \sum_{n = k}^\infty e^{n} \cdot \sup_{x \in [n,n+1]} |f(x)|, 
\end{align*}
which converges uniformly to 0 as $k \to \infty,$ being the tail of one fixed convergent sum (independent of $R$).
\\
\par
Uniform smallness of the tail on the spectral side follows by the following estimate:
\begin{align*}
\widehat{f} \ast g_R(t) &= \int_{\mathbb{R}} \widehat{f}(t-y) g_R(y) dy \\
&= \int_{y \in [-1/2 |t|, + 1/2 |t|]} \widehat{f}(t-y) g_R(y) dy  +  \int_{y \notin [-1/2 |t|, + 1/2 |t|]} \widehat{f}(t-y) g_R(y) dy \\ 
&= O \left( \sup_{y \in [1/2 |t|, 3/2 |t|]} |\widehat{f}(y)| \right) + O \left( ||\widehat{f}||_{\infty} \cdot \int_{y \notin [-1/2 R |t|, + 1/2 R |t| ]} \widehat{b}(z) dz \right).
\end{align*}
Since $\widehat{b}$ and $\widehat{f}$ are Schwartz, for every $p > 0$ there is some constant $C_p$ for which the latter is bounded above by $C_p \langle t \rangle^{-p}$ uniformly in $R.$  In particular, since the number of spectral parameters in $[\nu,\nu+1]$ grows quadratically with $\nu,$ taking any $p > 3,$ the tail of the spectral side of the trace formula converges to 0 uniformly in $R.$  
\end{proof}


\vspace{0.5cm}

\begin{thebibliography}{40}

\bibitem{Apo} T. Apostol. \emph{Introduction to analytic number theory. }Undergraduate Texts in Mathematics. Springer-Verlag, New York-Heidelberg, 1976.

\bibitem{APS1} M. F. Atiyah, V. K. Patodi, I. M. Singer. \emph{Spectral asymmetry and Riemannian geometry. I.} Math. Proc. Cambridge Philos. Soc. 77 (1975), 43-69. 

\bibitem{APS2} M. F. Atiyah, V. K. Patodi, I. M. Singer. \emph{Spectral asymmetry and Riemannian geometry. II.} Math. Proc. Cambridge Philos. Soc. 78 (1975), no. 3, 405-432.
\bibitem{BGV} N. Berline, E. Getzler, M. Vergne. \emph{Heat kernels and {D}irac operators.}  Grundlehren Text Editions. Springer-Verlag, Berlin, 2004

\bibitem{BP} R. Benedetti, C. Petronio \emph{Lectures on hyperbolic geometry.} Universitext. Springer-Verlag, Berlin, 1992.

\bibitem{BloBru} V. Blomer, F. Brumley. \emph{The role of the Ramanujan conjecture in analytic number theory. }Bull. Amer. Math. Soc. (N.S.) 50 (2013), no. 2, 267-320. 

\bibitem{BS} A. Booker, A. Strombergsson. \emph{Numerical computations with the trace formula and the Selberg eigenvalue conjecture. }J. Reine Angew. Math. 607 (2007), 113-161.

\bibitem{BW}
A. Borel, N. Wallach.  \emph{Continuous cohomology, discrete subgroups, and representations of reductive groups}. Second edition. Mathematical Surveys and Monographs, 67. American Mathematical Society, Providence, RI, 2000.

\bibitem{Bouaziz}
A. Bouaziz.  \emph{Int\'{e}grales orbitales sur les groupes de Lie r\'{e}ductifs}.  Ann. Sci. École Norm. Sup. (4) 27 (1994), no. 5, 573-609.

\bibitem{Bus} P. Buser. \emph{Geometry and spectra of compact Riemann surfaces.}  Modern Birkhauser Classics. Birkhauser Boston, Ltd., Boston, MA, 2010.

\bibitem{BRT} B. Burton, H. Rubinstein, S. Tillmann. \emph{The Weber-Seifert dodecahedral space is non-Haken.} Trans. Amer. Math. Soc. 364 (2012), no. 2, 911-932. 

\bibitem{Car} P. Cartier. \emph{An introduction to zeta functions. }From number theory to physics (Les Houches, 1989), 1-63, Springer, Berlin, 1992. 

\bibitem{CDGW}
M. Culler, N. Dunfield, M. Goerner, J. Weeks. \emph{SnapPy, a computer program for studying the geometry and topology of 3-manifolds.}

\bibitem{Cha} I. Chavel. \emph{Eigenvalues in Riemannian geometry.} Pure and Applied Mathematics, 115. Academic Press, Inc., Orlando, FL, 1984. xiv+362 pp.

\bibitem{CGHiso} V. Colin, P. Ghiggini, K. Honda. \emph{Equivalence of Heegaard Floer homology and em- bedded contact homology via open book decompositions.} Proc. Natl. Acad. Sci. USA, 108(20):8100-8105, 2011.


\bibitem{CGHW} D. Coulson, O. Goodman, C. Hodgson, W. Neumann \emph{Computing arithmetic invariants of 3-manifolds.} Experiment. Math. 9 (2000), no. 1, 127-152. 

\bibitem{Eve} B. Everitt. \emph{3-manifolds from Platonic solids. } Topology Appl. 138 (2004), no. 1-3, 253-263. 

\bibitem{GMM} D. Gabai, R. Meyerhoff, P. Milley. \emph{Volumes of Tubes in Hyperbolic 3-Manifolds.} J. Differential Geom. 57 (2001), no. 1, 23–46.

\bibitem{Hit} N. Hitchin. \emph{Harmonic spinors.}Advances in Math. 14 (1974), 1-55. 
\bibitem{HW}
C. Hodgson, J. Weeks. \emph{A census of closed hyperbolic 3-manifolds.} ftp://ftp.northnet.org/pub/weeks/SnapPea/ClosedCensus.

\bibitem{Kir}
R. Kirby. \emph{The topology of $4$-manifolds} Lecture Notes in Mathematics, 1374. Springer-Verlag, Berlin, 1989. vi+108 pp.

\bibitem{Knapp}
A. Knapp.  \emph{Representation theory of semisimple groups. An overview based on examples}. Princeton Mathematical Series, 36. Princeton University Press, Princeton, NJ, 1986.

\bibitem{KM} P. Kronheimer, T. Mrowka. \emph{Monopoles and three-manifolds.} New Mathematical Monographs, 10. Cambridge University Press, Cambridge, 2007.

\bibitem{KMOS}
P. Kronheimer, T. Mrowka, P. Ozsvath, Z. Szabo. \emph{Monopoles and lens space surgeries.} Ann. of Math. (2) 165 (2007), no. 2, 457-546. 

\bibitem{KLTiso} C. Kutluhan, Y.-J. Lee, C. Taubes. \emph{HF=HM I : Heegaard Floer homology and Seiberg-Witten Floer homology.} preprint, 2011.


\bibitem{LidLev} A. Levine, T. Lidman. \emph{Simply connected, spineless 4-manifolds.} Forum Math. Sigma 7 (2019), Paper No. e14, 11 pp.

\bibitem{Lin}F. Lin. \emph{A {M}orse-{B}ott approach to monopole {F}loer homology and the triangulation conjecture.} Mem. Amer. Math. Soc. 255 (2018), no. 1221, v+162 pp.

\bibitem{LL}
F. Lin, M. Lipnowski. \emph{The {S}eiberg-{W}itten equations and the length spectrum of hyperbolic three-manifolds}, to appear in Journal of the AMS.

\bibitem{LL2}
F. Lin, M. Lipnowski. \emph{Monopole Floer homology, eigenform multiplicities, and the Seifert-Weber dodecahedral space}, to appear in IMRN.

\bibitem{Mas}
B. Maskit. \emph{Kleinian groups.} Grundlehren der Mathematischen Wissenschaften, 287. Springer-Verlag, Berlin, 1988. xiv+326 pp.

\bibitem{Med} A. Mednykh. \emph{The isometry group of the hyperbolic space of a {S}eifert-{W}eber dodecahedron.} Siberian Math. J. 28 (1987), no. 5, 798-806.

\bibitem{MedVes}
A. Mednykh, A. Vesnin. \emph{Visualization of the isometry group action on the Fomenko-Matveev-Weeks manifold.} J. Lie Theory 8 (1998), no. 1, 51-66.

\bibitem{Mey} R. Meyerhoff. \emph{A lower bound for the volume of hyperbolic 3-manifolds.} Canad. J. Math. 39 (1987), 1038–1056.

\bibitem{Mil} J. Millson. \emph{Closed geodesics and the $\eta$-invariant.}  Ann. of Math. (2) 108 (1978), no. 1, 1-39.

\bibitem{MilSta} J. Milnor, J. Stasheff. \emph{Characteristic classes. }Annals of Mathematics Studies, No. 76. Princeton University Press, Princeton, N. J. (1974) vii+331 pp. 

\bibitem{MO} C. Manolescu, B. Owens. \emph{A concordance invariant from the Floer homology of double branched covers.}  Int. Math. Res. Not. IMRN 2007, no. 20, Art. ID rnm077.

\bibitem{MosSta} H. Moscovici, R. Stanton. \emph{Eta invariants of Dirac operators on locally symmetric manifolds. }Invent. Math. 95 (1989), no. 3, 629-666. 



\bibitem{Mul}W. Muller. \emph{Weyl's law in the theory of automorphic forms}. Preprint 2007.

\bibitem{NY} W. Neumann, J. Yang. \emph{Rationality problems for {K}-theory and {C}hern-{S}imons invariants of hyperbolic 3-manifolds. }Enseign. Math. (2) 41 (1995), no. 3-4, 281-296. 

\bibitem{Ou} M. Ouyang. \emph{A simplicial formula for the $\eta$-invariant of hyperbolic 3-manifolds.} Topology 36 (1997), no. 2, 411-421.

\bibitem{OSalt} P. Ozsvath, Z. Szabo. \emph{On the Heegaard Floer Homology of Branched Double-Covers.} Advances in Mathematics 194 (2005): 1-33.

\bibitem{OSgr} P. Ozsvath, Z. Szabo. \emph{Absolutely graded Floer homologies and intersection forms for four-manifolds with boundary.} Adv. Math. 173 (2003), no. 2, 179-261.

\bibitem{OSun} P. Ozsvath, Z. Szabo. \emph{Knots with unknotting number one and Heegaard Floer homology.} Topology 44 (2005), no. 4, 705-745.

\bibitem{OStr} B. Owens, S. Strle \emph{Immersed disks, slicing numbers and concordance unknotting numbers.} Comm. Anal. Geom. 24 (2016), no. 5, 1107-1138.

\bibitem{Sch} W. Scharlau. \emph{Quadratic and Hermitian forms. }Grundlehren der Mathematischen Wissenschaften, 270. Springer-Verlag, Berlin, 1985.

\bibitem{Tauiso} C. H. Taubes. \emph{Embedded contact homology and Seiberg-Witten Floer cohomology I.} Geom. Topol., 14(5):2497-2581, 2010.

\bibitem{WS} C. Weber, H. Seifert. \emph{Die beiden Dodekaederraume.} Math. Z. 37 (1933), no. 1, 237-253.

\end{thebibliography}
\end{document}